\documentclass[11pt]{article}
\usepackage[T1]{fontenc}
\usepackage{amsmath,amssymb,amsthm,mathtools,bm}
\usepackage{lmodern}
\usepackage[margin=1in]{geometry}
\usepackage{booktabs,array,longtable,multirow,tabularx}
\usepackage{microtype,enumitem}
\usepackage[dvipsnames]{xcolor}
\usepackage{tikz}
\usetikzlibrary{arrows.meta,positioning,fit,calc}
\usepackage[numbers,sort&compress]{natbib}
\usepackage[hidelinks]{hyperref}
\usepackage[nameinlink,capitalize,noabbrev]{cleveref}
\allowdisplaybreaks
\setlist[itemize]{leftmargin=1.5em,itemsep=1pt,topsep=3pt}
\setlist[enumerate]{leftmargin=1.65em,itemsep=1pt,topsep=3pt}

\newcommand{\R}{\mathbb R}
\newcommand{\E}{\mathbb E}
\newcommand{\PP}{\mathbb P}
\newcommand{\Lip}{\operatorname{Lip}}

\newcommand{\argmin}{\operatorname*{arg\,min}}
\newcommand{\norm}[1]{\left\lVert #1\right\rVert}
\newcommand{\ip}[2]{\left\langle #1,#2\right\rangle}

\newcommand{\SFO}{\mathsf{SFO}}
\newcommand{\T}{\mathsf T}
\newcommand{\eps}{\epsilon}

\newtheorem{theorem}{Theorem}[section]
\newtheorem{lemma}[theorem]{Lemma}

\newtheorem{corollary}[theorem]{Corollary}
\newtheorem{assumption}[theorem]{Assumption}
\theoremstyle{definition}
\newtheorem{definition}[theorem]{Definition}
\newtheorem{algorithm}[theorem]{Algorithm}
\theoremstyle{remark}

\crefname{theorem}{theorem}{theorems}
\crefname{lemma}{lemma}{lemmas}
\crefname{corollary}{corollary}{corollaries}
\crefname{proposition}{proposition}{propositions}
\crefname{assumption}{assumption}{assumptions}
\Crefname{assumption}{Assumption}{Assumptions}
\crefname{definition}{definition}{definitions}
\Crefname{definition}{Definition}{Definitions}
\crefname{remark}{remark}{remarks}
\Crefname{remark}{Remark}{Remarks}
\crefname{section}{section}{sections}
\crefname{algorithm}{algorithm}{algorithms}

\title{\textbf{Closing the Accuracy Gap in Stochastic First-Order Bilevel Optimization}}
\newcommand{\authornames}{Haihan Zhang, Wendao Wu, Chenheng Zhang, Yanyi Li, Chunyuan Zheng, Cong Fang, Haoxuan Li, Zhouchen Lin}
\newcommand{\affiliation}{Peking University}

\newcommand{\mail}[1]{\href{mailto:#1}{\texttt{#1}}}
\newif\ifanonymous
\anonymousfalse
\ifanonymous
\author{Anonymous Authors}
\hypersetup{pdfauthor={Anonymous Authors}}
\else
\author{Haihan Zhang$^{1,*}$\quad Wendao Wu$^{1,*}$\quad Chenheng Zhang$^{1,*}$\\[0.20em]
Yanyi Li$^1$\quad Chunyuan Zheng$^1$\\[0.20em]
Cong Fang$^{1,\dagger}$\quad Haoxuan Li$^{1,\dagger}$\quad Zhouchen Lin$^{1,\dagger}$\\[0.60em]
\small $^1$\affiliation\\[0.35em]
\footnotesize\mail{colombo.2026@outlook.com}\\[-0.05em]
\small $^*$Equal contribution.\quad $^\dagger$Corresponding authors.}
\hypersetup{pdfauthor={\authornames}}
\fi
\hypersetup{pdftitle={Closing the Accuracy Gap in Stochastic First-Order Bilevel Optimization}}

\date{Revised November 23, 2026}

\begin{document}
\maketitle

\begin{abstract}
We study the stochastic first-order complexity of smooth nonconvex--strongly-convex bilevel optimization. The oracle returns fresh conditionally unbiased gradients with bounded variance; the primitive assumptions are globally Lipschitz gradients, bounded $f_y$, and Lipschitz lower Hessian blocks. Write $\kappa=L/\mu$, $\vartheta=\rho/\mu$, and $\chi=1+\vartheta$. For $\kappa\ge2$, $1\le\vartheta\le\kappa^2$, every positive initial-gap budget $\Delta$, and
$0<\epsilon\le c\min\{L,\sqrt{L\Delta}\}$, we establish
\[
 \mathsf T_{\rm SFO}(\epsilon)=\widetilde\Theta\!\left(
 L\chi\kappa^{5/2}\Delta\epsilon^{-2}
 +L\kappa^4\sigma_f^2\Delta\epsilon^{-4}
 +L^3\vartheta^2\chi\kappa^8\sigma_g^2\Delta\epsilon^{-6}\right).
\]
The comparison assumes a finite lower-initialization certificate and is uniform over both nonnegative noise budgets. Its accuracy constant is independent of conditioning. Three lower-bound families yield the deterministic, upper-noise, and lower-noise terms. Their information arguments cover arbitrary randomized adaptive queries, terminal outputs, and worst-case expected costs, even for conservative sample gradients. The upper bound combines accelerated frozen-linear-tilt solves with displacement-based reuse of the full upper gradient. An independently chosen base-chain length gives the uniform accuracy range. A projection and homotopy argument separates undisclosed coordinates from a newly disclosed source, allowing $\vartheta/\kappa$ to grow without bound. The result matches polynomial orders and noise transitions; exact logarithmic exponents and coarser accuracies remain open.
\end{abstract}


\begin{samepage}
\noindent\textbf{AI Usage.}
Nearly the entire research pipeline for this paper was carried out by Colombo, our laboratory’s internal auto-research system, powered by GPT-6 Astra. The system also conducted a Lean-backed article audit of the resulting manuscript. The authors subsequently reviewed and approved the mathematical claims, presentation, and formal artifacts, and take responsibility for the final manuscript. The complete Lean audit report and the system’s technical report will be made public at a later date. \par
\end{samepage}

\paragraph{Publication Intent and Recognition of Independent Work.} 
This preprint is released as a research output of our study of
automated scientific research, for public scrutiny and discussion.
As we do not intend to submit this manuscript to a journal or
conference, it has not undergone the full editorial refinement
and internal review we would normally require for submission.
The authors remain responsible for its contents and for correcting
any errors or inaccurate attribution.

We recognize that releasing these results may affect researchers
who have independently been pursuing the same questions, including
work that is not yet public.
If this applies to your work, we invite you to contact us at
\texttt{colombo.2026@outlook.com}.
A brief description of the overlap, together with any dated
materials you are comfortable sharing, would help us understand
its scope and chronology.
Such materials might include draft excerpts, presentation slides,
or revision records; a complete manuscript or unpublished proof
is not required to begin the conversation.

We will review these cases carefully and update citations,
acknowledgments, and accounts of independent or concurrent work
as appropriate, including revising priority claims where warranted.
Nonpublic materials will be treated confidentially and used only
to assess the overlap unless their authors authorize otherwise.
We will obtain explicit permission from the authors before
publicly discussing such materials or providing them to an
automated research system.

\section{Introduction}\label{sec:intro}
Consider the bilevel problem
\begin{equation}\label{eq:intro-bilevel}
 \min_x F(x)=f(x,y^*(x)),\qquad
 y^*(x)=\argmin_y g(x,y),
\end{equation}
where $g(x,\cdot)$ is strongly convex and $F$ may be nonconvex.
We ask how many fresh stochastic gradient queries are necessary and sufficient to obtain
$\mathbb E\|\nabla F(\widehat x)\|\le\epsilon$.
The oracle returns gradients of $f$ and $g$ at an arbitrary queried point in their full Euclidean domains. It provides neither a lower minimizer nor a Hessian-vector product.

The dependence on accuracy alone does not describe this cost. Computing a hypergradient requires recovering the response of the lower problem, whose condition number can be large. Noise in the upper gradient and noise in the lower gradient enter this calculation differently. We therefore keep the condition number, lower Hessian variation, initial gap, and two variance budgets explicit.

Our main result matches three contributions: a deterministic term of order
$L(1+\vartheta)\kappa^{5/2}\Delta\epsilon^{-2}$,
an upper-noise term of order $L\kappa^4\sigma_f^2\Delta\epsilon^{-4}$,
and a lower-noise term of order
$L^3\vartheta^2(1+\vartheta)\kappa^8\sigma_g^2\Delta\epsilon^{-6}$,
up to logarithmic factors. The comparison holds simultaneously for zero, positive, and accuracy-dependent noise budgets. Its sufficient accuracy threshold is
$c\min\{L,\sqrt{L\Delta}\}$, with a numerical $c$ independent of $\kappa$ and $\vartheta$.

\paragraph{Lower bounds.}
The global bound on $f_y$ is a central constraint: putting the entire initial gap in one long lower-variable carrier would violate it. We store the gap in an outer macro chain and activate bounded lower-variable chains through disjoint selectors. A recursive hierarchy increases the number of necessary activations while a partial Legendre transform realizes its exact lower response with the prescribed primitive budgets. The base-chain length is chosen independently of the response gain, which permits the uniform accuracy range. Separate exact-gradient endpoint transport and conservative upper-gradient noise families supply the other two terms. Projection observers account for arbitrary query directions, exact scalar information, and unqueried terminal outputs.

\paragraph{Upper bound.}
We freeze an estimate of $f_y$ and solve a linearly tilted lower problem. Its hypergradient bias depends explicitly on lower Hessian variation. Accelerated inner solves give the deterministic cost. Reusing the full upper-gradient estimate until the outer point has moved a prescribed distance gives the upper-noise cost. Conditional confidence bounds and a computable path cap control both the final expected gradient and the number of calls on every history.

\subsection{Related work and attribution}\label{sec:related}
Fully first-order methods avoid the Hessian-vector products used by earlier bilevel algorithms \citep{JiYangLiang2021}. The methods of \citet{KwonEtAl2023} and \citet{CMZ} establish stochastic accuracy orders $\widetilde O(\epsilon^{-7})$ and $\widetilde O(\epsilon^{-6})$, respectively. The latter work also obtains a fourth-power accuracy order when only the upper objective is noisy. Higher-order smoothness permits faster rates \citep{CLCZ}; it is not assumed here.

\citet{CJZ} give condition-dependent lower bounds and deterministic acceleration results. We use their endpoint-amplified quadratic matrix and prove the full-response and transport properties required by our construction. Their standard-SFO stochastic lower bound has order $\Omega(\kappa_y^{9/2}\epsilon^{-4})$ under its stated constrained and zero-respecting qualifications. \citet{KKL} characterize a distinct $y^*$-aware interface. For standard stochastic first-order access, \citet{gu2026firstorderlower} prove an $\Omega(\Delta\sigma^2\kappa_y^8\epsilon^{-6})$ lower bound for arbitrary randomized adaptive algorithms in their parameterization. Our statements keep Hessian variation and the two noise budgets separate; no literature-priority claim is required for the proofs.

The zero-chain and randomized-information ingredients follow the framework of \citet{Carmon,ACD}. The inner acceleration principle is established in prior work, including \citet{Aybat2019}; our conditional-noise proof and upper-gradient refresh analysis are given in full. The confidence selector is related to geometric-median amplification \citep{Minsker2015}.

\paragraph{Organization.}
\Cref{sec:setup} fixes the oracle model. The main rates and proof map appear in \cref{sec:results,sec:proof-map}; the appendices contain their complete proofs.

\section{Problem Formulation}\label{sec:setup}
We study \eqref{eq:intro-bilevel} over $x\in\R^{d_x}$ and $y\in\R^{d_y}$.  The following assumptions are used throughout.

\begin{assumption}[Primitive regularity]\label[assumption]{ass:primitive}
The lower objective $g(x,\cdot)$ is $\mu$-strongly convex.  Moreover,
\[
\norm{f_y(x,y)}\le C_f,
\qquad
\Lip(\nabla f)\le L_f,
\qquad
\Lip(\nabla g)\le L_g,
\]
and the blocks $g_{yy}$ and $g_{xy}$ are jointly $\rho$-Lipschitz in $(x,y)$.  The initial hyperobjective gap satisfies $F(x_0)-\inf F\le\Delta$.
\end{assumption}

For uniform upper bounds, we also assume a finite initialization certificate
\[
\norm{y_{\rm init}-y^*(x_0)}\le R_0.
\]
Let
\[
k=L_g/\mu,
\qquad
\kappa=L/\mu,
\qquad
\vartheta=\rho/\mu,\quad\theta=\vartheta,
\qquad
\chi=1+\theta.
\]
The normalized class imposes $C_f,L_f,L_g\le L$. All lower families meet these exactly prescribed bounds; numerical slack used while deriving their formulas is removed by explicit coefficient choices. The gap ratio $\Delta/L$ is unrestricted.

\begin{definition}[Standard stochastic first-order oracle]\label[definition]{def:sfo}
At an adaptively chosen query $(x,y)$, an SFO returns random vectors $(\widehat\nabla f,\widehat\nabla g)$ satisfying
\[
\E[\widehat\nabla f\mid\mathcal F, x,y]=\nabla f(x,y),
\qquad
\E[\widehat\nabla g\mid\mathcal F, x,y]=\nabla g(x,y),
\]
\[
\E[\norm{\widehat\nabla f-\nabla f}^2\mid\mathcal F,x,y]\le\sigma_f^2,
\qquad
\E[\norm{\widehat\nabla g-\nabla g}^2\mid\mathcal F,x,y]\le\sigma_g^2.
\]
Each call uses a fresh sample and returns only the two gradient vectors at its queried point.  For a conservative SFO, every sample response is the gradient of a globally defined scalar function.
\end{definition}

The target is an output $\widehat x$ satisfying $\E\norm{\nabla F(\widehat x)}\le\eps$. Let $\mathcal I$ denote the admissible mean functions and oracle laws, over all finite dimensions, with the stated public budgets. For randomized adaptive algorithms $\mathcal A$, define the worst-case expected query complexity by
\[
\T_{\SFO}(\eps)=\inf_{\mathcal A}\sup_{I\in\mathcal I}\E_I N_{\mathcal A},
\quad\text{subject to}\quad
\sup_{I\in\mathcal I}\E_I\norm{\nabla F_I(\widehat x_{\mathcal A})}\le\eps.
\]
Here $N_{\mathcal A}$ counts original SFO calls, and expectations include oracle and algorithm randomness. The output may be any measurable function of the transcript and private randomness. Fixed-horizon lower bounds use the additional constraint $N_{\mathcal A}\le N$ almost surely. The algorithms use the public regularity, variance, gap, and initialization bounds. The notation $\widetilde O$ suppresses logarithmic factors in the public parameter ratios, accuracy, and initialization; $\widetilde\Theta$ denotes matching upper and lower orders up to these factors.

Fix a numerical $C_F$ large enough for the three curvature-controlled lower families in \cref{app:lower,app:upper-noise-lower,app:det-transport-geometry}, and let
$\mathcal C_{\rm curv}$ be the primitive class with the additional public bound
\[
 \operatorname{Lip}(\nabla F)\le C_F L\kappa^2.
\]
We write $\T_{\SFO}^{\rm curv}$ for minimax complexity over this class. The bound is on the true hyperobjective. All dimensions remain finite; lower bounds choose them sufficiently large as specified in the information lemmas.

\section{Main Results}\label{sec:results}

Define
\begin{equation}\label{eq:joint-scales}
 \begin{aligned}
 D&=L\chi\kappa^{5/2}\Delta\epsilon^{-2},&
 U&=L\kappa^4\sigma_f^2\Delta\epsilon^{-4},\\
 G&=L^3\vartheta^2\chi\kappa^8\sigma_g^2\Delta\epsilon^{-6}.&&
 \end{aligned}
\end{equation}
The letters distinguish the deterministic, upper-noise, and lower-noise contributions.

\begin{theorem}[Uniform deterministic and two-noise complexity]
\label[theorem]{thm:joint-noise-comparison}
There are numerical $c,c_1>0$ and $q<\infty$ such that the following holds.
Let $L,\Delta>0$, $\kappa\ge2$, $1\le\vartheta\le\kappa^2$, and
\begin{equation}\label{eq:uniform-accuracy-main}
 \boxed{\quad 0<\epsilon\le c\min\{L,\sqrt{L\Delta}\}.\quad}
\end{equation}
For the primitive class with $\mu=L/\kappa$, $C_f,L_f,L_g\le L$,
Hessian-block variation at most $\rho=\mu\vartheta$, and a finite lower-initialization certificate,
\begin{equation}\label{eq:joint-bracket}
 c_1\left[U+\frac{D+G}{[\log(2+\vartheta)]^q}\right]
 \ \le\ \mathsf T_{\rm SFO}(\epsilon)
 \ \le\ \widetilde O(D+U+G),
\end{equation}
uniformly for $\sigma_f,\sigma_g\ge0$.
The same comparison holds when sample gradients must be conservative.
The lower bounds permit arbitrary randomized adaptive queries and arbitrary measurable terminal outputs, and apply to worst-case expected costs. Their instances have $y^*(0)=0$ and sufficiently large finite dimensions fixed before execution.
The accuracy constant is independent of $\kappa$, $\vartheta$, and the noise budgets.
\end{theorem}

The quadratic variation range permits an unbounded ratio $\vartheta/\kappa$:
for example, $\vartheta=\kappa^2$ gives $\vartheta/\kappa=\kappa$.
It does not assert a matching lower bound for arbitrary $\vartheta$ at fixed
$\kappa$. The three-term upper bound holds for every $\vartheta\ge1$,
as proved in \cref{cache:eq:three-terms}.

\begin{proof}
The uniform calibration and information certificates in
\cref{app:uniform-accuracy-geometry,waI:sec:information} yield the two lower families for $D$ and $G$.
The deterministic family adds the endpoint delay in
\cref{app:det-transport-geometry,app:det-transport-info};
the lower-noise family uses the fixed conservative law of
\cref{rec:lem:conjugate-envelope}.
The independent upper-noise family in \cref{uf:thm:lower} gives $U$ without a logarithmic loss.
Every family belongs to the class with the prescribed variance upper bounds.
The minimax cost is therefore at least their maximum, hence at least one third of their sum. Taking the largest of the fixed logarithmic exponents gives the displayed lower bound.

For the upper bound, use \cref{thm:main-upper-informal} below.
Under \eqref{eq:uniform-accuracy-main}, its normalized expansion
\eqref{eq:main-upper-normalized} has deterministic contribution
$\widetilde O(D)$ and upper-noise contribution $\widetilde O(U)$.
The total fourth-power lower-noise contribution is
$V=L\kappa^4(\chi+\kappa^2)\sigma_g^2\Delta\epsilon^{-4}$.
For $\sigma_g>0$,
$V/G=\epsilon^2(1+\kappa^2/\chi)/(L^2\vartheta^2\kappa^4)\le Cc^2$;
when $\sigma_g=0$ both terms vanish.
The outer-count floors are also absorbed in this range.
This proves the comparison for all nonnegative noise budgets.
\end{proof}

\subsection{An implementable upper bound}
Let
\begin{equation}\label{eq:main-Erho}
 E_\rho=\frac{(1+k)\rho C_f^2}{2\mu^2},\qquad
 L_a=L_f(1+k)^2,\qquad
 \Lambda\ge\max\{\Lip(\nabla F),L_a\},\quad\Lambda>0,
\end{equation}
and set $t=\min\{1,\epsilon/(CE_\rho)\}$, with $t=1$ when $E_\rho=0$.

\begin{theorem}[Accelerated frozen tilts with upper-gradient reuse]
\label[theorem]{thm:main-upper-informal}
Under \cref{ass:primitive} and the finite initialization certificate,
the fresh SFO admits an algorithm with
$\mathbb E\|\nabla F(\widehat x)\|\le\epsilon$ and cost
\begin{equation}\label{eq:main-upper}
\begin{split}
 \widetilde O\!\Bigg[&
 \left(1+\frac{\Lambda\Delta}{\epsilon^2}\right)
 \left(1+\sqrt k+\frac{k^2\sigma_g^2}{\epsilon^2t^2}\right)\\
 &+\left(1+\frac{L_a\Delta}{\epsilon^2}\right)
 \left(1+\frac{(1+k)^2\sigma_f^2}{\epsilon^2}
       +\frac{(1+k)^2L_f^2\sigma_g^2}{\mu^2\epsilon^2}\right)
 \Bigg].
\end{split}
\end{equation}
Including the explicit logarithms in \eqref{cache:eq:explicit-budget}, this is a deterministic bound on the number of calls on every history.
\end{theorem}
\begin{proof}
The tilt bias and sensitivity are proved in \cref{app:upper}.
The algorithm, inner accuracies, conditional confidence levels, path cap, and complete budget are proved in \cref{ac:thm:cached-upper}.
\end{proof}

The universal choice
$\Lambda=(1+k)^2(L_f+\rho C_f/\mu)$
(with $\Lambda=L$ when this expression vanishes)
gives, under the normalized primitive bounds, the following expression
for every $\vartheta\ge0$. Write $X=L\kappa^2\Delta/\epsilon^2$:
\begin{equation}\label{eq:main-upper-normalized}
\begin{split}
 \widetilde O\!\Bigg[&
 (1+\chi X)\left(\sqrt\kappa+
       \frac{\kappa^2\sigma_g^2}{\epsilon^2}
       +\frac{L^2\vartheta^2\kappa^6\sigma_g^2}{\epsilon^4}\right)\\
 &+(1+X)\left(1+\frac{\kappa^2\sigma_f^2}{\epsilon^2}
         +\frac{\kappa^4\sigma_g^2}{\epsilon^2}\right)
 \Bigg].
\end{split}
\end{equation}
This expression retains the outer floors and is useful before simplifying to \eqref{eq:joint-bracket}.

\subsection{Noise transitions and normalization}
Up to logarithmic factors, the largest of $D,U,G$ determines the complexity. The noise terms dominate when
\begin{equation}\label{eq:noise-dominance-new}
 U+G\gtrsim D.
\end{equation}
The corresponding pairwise boundaries are
\begin{equation}\label{eq:noise-thresholds-new}
 \sigma_f^2\asymp\frac{\chi\epsilon^2}{\kappa^{3/2}},
 \qquad
 \sigma_g^2\asymp\frac{\epsilon^4}{L^2\vartheta^2\kappa^{11/2}},
\end{equation}
and the upper-noise/lower-noise boundary is
\begin{equation}\label{eq:noise-crossover-new}
 \sigma_g^2\asymp
 \frac{\epsilon^2\sigma_f^2}{L^2\vartheta^2\chi\kappa^4}.
\end{equation}
These are polynomial-scale boundaries; exact logarithmic crossover locations are not asserted.
At zero noise, \eqref{eq:joint-bracket} becomes
$\widetilde\Theta(L\chi\kappa^{5/2}\Delta\epsilon^{-2})$.
For $L=\Delta=\rho=1$ and $\mu=1/\kappa$, the full comparison is
\[
 \mathsf T_{\rm SFO}(\epsilon)
 =\widetilde\Theta\!\left(
 \kappa^{7/2}\epsilon^{-2}
 +\kappa^4\sigma_f^2\epsilon^{-4}
 +\kappa^{11}\sigma_g^2\epsilon^{-6}\right),
 \qquad 0<\epsilon\le c.
\]

\begin{corollary}[A prescribed hyperobjective curvature bound]
\label[corollary]{cor:curvature-controlled}\label[corollary]{cor:matching}
Under the initialization, accuracy, and variance conditions of
\cref{thm:joint-noise-comparison}, with $\kappa\ge2$ and the additional
restriction $1\le\vartheta\le\kappa$, on $\mathcal C_{\rm curv}$,
\begin{equation}\label{eq:curv-joint-matching}
 \mathsf T_{\rm SFO}^{\rm curv}(\epsilon)
 =\widetilde\Theta\!\left(
 L\kappa^{5/2}\Delta\epsilon^{-2}
 +L\kappa^4\sigma_f^2\Delta\epsilon^{-4}
 +L^3\vartheta^2\kappa^8\sigma_g^2\Delta\epsilon^{-6}\right).
\end{equation}
\end{corollary}
\begin{proof}
Use \cref{thm:main-upper-informal} with
$\Lambda=O(L\kappa^2)$ and $\Lambda\ge L_a$.
The lower families are the exact-quadratic transport family in
\cref{tr:thm:quadratic}, the base lower-noise family in
\cref{thm:main}, and the upper-noise family in \cref{uf:thm:lower}.
Their true hyperobjectives have curvature $O(L\kappa^2)$.
Choose $C_F$ once to contain all three, take their minimax maximum, and absorb floors as in the main theorem.
\end{proof}

\section{Proof Architecture}\label{sec:proof-map}\label{sec:lower-overview}

\subsection{Which family supplies each lower term}
The lower bound uses three separate families. Table~\ref{tab:proof-families} records their roles; combining their costs requires only a minimax maximum.
\begin{table}[ht]
\centering
\caption{Lower-bound witnesses for the three terms in \eqref{eq:joint-scales}. All families satisfy the same prescribed primitive budgets and have exact lower initialization.}
\label{tab:proof-families}
\small
\renewcommand{\arraystretch}{1.2}
\begin{tabularx}{\linewidth}{p{0.12\linewidth}p{0.29\linewidth}X}
\toprule
Term & Oracle in the hard family & Source of the cost\\
\midrule
$D$ & Exact upper and lower gradients & A recursive hierarchy followed by a serial endpoint delay of order $\sqrt\kappa$.\\
$U$ & Noisy upper gradient; public exact quadratic lower problem & A bounded upper carrier with signal of order $\epsilon/\kappa$.\\
$G$ & Exact upper gradient; noisy lower gradient & Recursive activation count and a full-gradient signal of order $\epsilon^2/(L\vartheta\kappa^3)$, up to logarithms.\\
\bottomrule
\end{tabularx}
\end{table}

\paragraph{Gap and primitive budgets.}
An outer macro chain stores the initial gap. At each level, at most one selected carrier has a nonzero jet at any point. This keeps the lower-variable upper gradient bounded independently of the macro length. Common radial penalties and terminal receivers give a gradient margin at every point with an incomplete macro terminal, including points far from the response graph. \Cref{app:lower} proves the base construction and all bounded-parameter bridges.

\paragraph{Recursive variation gain.}
For an effective curvature budget $T$, choose
$a_j=T^{1-4^{-j}}$. Then
$a_j(a_j/a_{j-1})^{1/3}=T$,
so every level meets the same third-derivative allocation. After
$J=O(\log\log(2+\vartheta))$ levels, the final curvature is
$\vartheta$ up to logarithms. A partial Legendre transform realizes the exact recursive trace while controlling all joint primitive derivatives; see \cref{app:recursive-hierarchy}.

\paragraph{Uniform accuracy.}
Choose a bounded macro width $b_0$, calibrate the middle width $h_0$ to the target gradient, and set
$n_0\asymp(b_0/h_0)^2$.
The macro length satisfies $S\asymp\Delta/(L_0b_0^2)$, hence
$Sn_0\asymp\Delta/(L_0h_0^2)$.
The count therefore retains its polynomial order while $b_0$ controls the primitive gradient budget. The full joint Hessian bounds are of order
$(\epsilon/L)\sqrt\vartheta/\kappa$ after retaining their depth allocations,
which permits $\vartheta\le\kappa^2$. \Cref{app:uniform-accuracy-geometry,waI:sec:information} give the ordered constant choices and the complete proof.

\paragraph{Arbitrary-query information.}
Projection observers retain all scalar projections onto adaptive query spans. Strong convexity defines a virtual dual inverse on the disclosed subspaces. On the projection event it equals every inverse needed by the sampled response: the inactive inverse on every stochastic round, and the mean inverse on an active round. Still-undisclosed coordinates are controlled directly by projections onto those subspaces. A newly disclosed source is instead kept inside the wider predecessor-flat interval by a simultaneous homotopy argument. This avoids comparing the total hidden displacement with the finest blind width, so a successful lower coin reveals only one eligible deepest direction. Final output directions are inserted after the output is fixed, and truncation converts fixed-horizon bounds to expected-cost bounds.

\paragraph{Deterministic transport.}
The deepest chain reads each incoming direction at the first endpoint of a quadratic system and reads predecessor and release gates at the last endpoint. Its trace is unchanged, but a new source needs $\Theta(\sqrt\kappa)$ sequential transport steps. The observer controls simultaneous probes in both edge directions and fixes all dimensions before execution. The geometry and information proofs are in \cref{app:det-transport-geometry,app:det-transport-info}.

\subsection{Upper algorithm}\label{sec:upper-overview}
At $x$, write $y_0=y^*(x)$ and freeze $a_0=f_y(x,y_0)$. The ideal tilted solution and estimator are
\[
 y_t(a_0)=\argmin_y\{g(x,y)+t\langle a_0,y\rangle\},\qquad
 h_t(a_0)=f_x(x,y_0)+\frac{g_x(x,y_t(a_0))-g_x(x,y_0)}t.
\]
Its bias is at most $E_\rho t$. The tilt leaves lower curvature unchanged, so the same accelerated fresh-noise solver applies to both lower problems.

\begin{algorithm}[Frozen tilts with displacement-based refreshes]
\label[algorithm]{alg:frozen}
Use the parameters and conditional confidence levels in \cref{ac:thm:cached-upper}.
\begin{enumerate}
\item At the first iterate, or when the displacement from the anchor exceeds $\epsilon/[CL_f(1+k)^2]$, solve the anchor lower problem and estimate the full upper gradient. Store its outer component and its lower component projected onto the ball of radius $C_f$.
\item At the current $x$, solve the untilted lower problem and the problem with the stored linear tilt to their prescribed distance accuracies. At a refresh, reuse the anchor's untilted solve from step 1.
\item Combine the stored outer upper-gradient component with two fresh lower-gradient means, as in \eqref{cache:eq:estimator}.
\item If adding the current $\|G\|^2$ would make the cumulative sum exceed the public cap $E_{\rm cap}$, return $x_0$. Otherwise take $x\leftarrow x-G/\Lambda$.
\item After the prescribed number of steps, return a uniformly selected pre-update iterate.
\end{enumerate}
When $L_f=0$, one upper-gradient anchor suffices.
\end{algorithm}
On the simultaneous accuracy event the cap is not reached. On every history it bounds the path, initialization radii, number of refreshes, and gradient norm of every possible output. Stored estimates are analyzed through deterministic error propagation on this event. Conditional confidence amplification controls its complement under the original finite-variance oracle.

\section{Conclusion}\label{sec:conclusion}
The main theorem gives matching polynomial orders for deterministic work and both noise sources over a conditioning-independent accuracy range. A recursive Legendre hierarchy supplies the lower-noise variation factor; endpoint transport supplies the deterministic delay; upper-gradient reuse attains the independent upper-noise rate. The lower proofs cover arbitrary randomized queries and outputs, and the upper proof bounds cost on every history. The proved range $1\le\vartheta\le\kappa^2$ allows $\vartheta/\kappa$ to be unbounded. Matching lower bounds for larger variation at fixed $\kappa$, exact logarithmic exponents, and the regimes beyond \eqref{eq:uniform-accuracy-main} remain unresolved.

\paragraph{Author note.}
AI tools, including Colombo, assisted with the construction, analysis, and writing. The authors remain responsible for all claims and attribution; complete formal verification is not asserted. This research preprint is released for public scrutiny and is not intended for journal or conference submission. Reports of independent or concurrent work may be sent to \texttt{colombo.2026@outlook.com}. Citations and priority statements will be corrected when warranted. Unpublished material shared for this purpose will be kept confidential and used only to assess overlap; public disclosure or processing by an automated research system requires its authors' permission.

\clearpage
\appendix

\section{Base hierarchy and parameter-range bridges}\label{app:lower}
This appendix proves \cref{thm:main}. For readability, the construction is first written with fixed numerical slack: lower modulus at least $\mu/2$, $C_f,L_f,L_g\le CL$, and Hessian-block variation at most $C\mu\vartheta$. Every constant is independent of $L,\Delta,\kappa,\vartheta,\epsilon$ and dimension. \Cref{sec:exact-budget-normalization} explicitly changes the quadratic modulus, coupling, normal coefficient, and chain curvature to meet the exactly prescribed budgets. These fixed numerical changes preserve the rate and the accuracy range. The proof is organized as a global geometric certificate, a primitive-and-oracle certificate, an information certificate, and the complete parameter-range assembly.

\begin{theorem}[All-gap conservative lower bound]
\label[theorem]{thm:main}\label[theorem]{thm:main-lower-informal}
There are numerical $c,C>0$ such that, for all $L,\Delta>0$,
$\kappa\ge2$, $1\le\vartheta\le\kappa$, $\sigma_g\ge0$, and
\begin{equation}\label{eq:main-accuracy}
 0<\epsilon\le c\min\{L,\kappa\sqrt{L\Delta}\},
\end{equation}
there are sufficiently large finite dimensions and globally admissible instances in $\mathcal C_{\rm curv}$
with lower modulus at least $\mu$, $C_f,L_f,L_g\le L$, Hessian-block
variation at most $\rho=\mu\vartheta$, $F(0)-\inf F\le\Delta$,
and $y^*(0)=0$ such that the original
SFO minimax complexity required to guarantee
$\mathbb E\|\nabla F(\widehat X)\|\le\epsilon$ is at least
\begin{equation}\label{eq:main-rate}
 \boxed{\ \Omega\!\left(
 \kappa^2L\Delta\epsilon^{-2}
 +\vartheta^2\kappa^8L^3\sigma_g^2\Delta\epsilon^{-6}
 \right).\ }
\end{equation}
The bound holds for randomized adaptive algorithms, arbitrary terminal
outputs, both fixed horizons and worst-case expected numbers of calls,
and even when oracles must be conservative. Exact upper gradients are
returned by the hard oracle, so any nonnegative upper variance budget
is allowed.
\end{theorem}

\subsection{Profiles, chains, selectors, and radial maps}

Use the profiles
\[
\Psi(t)=\begin{cases}0,&t\le1/2,\\
\exp\{1-(2t-1)^{-2}\},&t>1/2,
\end{cases}
\qquad
\Phi(t)=\sqrt e\int_{-\infty}^{t}e^{-s^2/2}\,ds.
\]
They are nonnegative and nondecreasing, $\Psi\le e$, $\Phi$ is bounded, and $\Phi'(t)\ge1$ for $|t|\le1$. Their derivatives of every fixed order are bounded. In particular $|\Psi'|\le C\sqrt\Psi$.

Choose a smooth nondecreasing $\chi_-:\R\to[0,1]$ which is zero on $(-\infty,1/2]$ and one on $[1,\infty)$. Define
\[
 p(t)=\int_t^0\chi_-(-s)\,ds.
\]
Then $p\ge0$, $p=0$ on $[-1/2,\infty)$, $-1\le p'\le0$, $p'=-1$ on $(-\infty,-1]$, and $p(t)\le(-t)_+$. Its derivatives of orders at least one are bounded. The linear, rather than quadratic, growth of $p$ will be used to bound the value of a selected micro core.

For a length $n$, width $h>0$, and curvature parameter $L>0$, define
\begin{align}
 C_{n,h}(z)
 &=Lh^2\left[-\Phi(z_1/h)
 -\sum_{j=2}^n\Psi(z_{j-1}/h)\Phi(z_j/h)
 +\sum_{j=1}^np(z_j/h)\right].\label{eq:chain}
\end{align}
Every coordinate derivative of $C_{n,h}$ is nonpositive. Also,
\begin{align}
 \norm{\nabla C_{n,h}(z)}&\le C Lh\sqrt n,
 &\norm{\nabla^2 C_{n,h}(z)}&\le C L,\label{eq:chain-reg}\\
 |C_{n,h}(z)|&\le C Lh^2n+Lh\sqrt n\norm z.\label{eq:chain-value}
\end{align}
The last estimate follows from $p(t)\le(-t)_+$. If $j$ is the first coordinate with $z_j<h$, then
\begin{equation}
 \partial_j C_{n,h}(z)\le-Lh.\label{eq:first-small}
\end{equation}
Indeed, if $-h\le z_j<h$, the incoming chain term gives this bound because the predecessor is at least $h$ (or $j=1$). If $z_j<-h$, the negative-side penalty gives it. Every other contribution to this coordinate derivative is nonpositive.

Choose an even smooth cutoff $c_0:\R\to[0,1]$ which vanishes on $[-1/4,1/4]$, equals one outside $(-1/2,1/2)$, and is nondecreasing on $[0,\infty)$. Choose a nonincreasing smooth cutoff $e_0:\R\to[0,1]$ which equals one on $(-\infty,1/4]$, vanishes on $[1,\infty)$, and satisfies $|e_0'|\le C\sqrt{e_0}$. One may obtain the last property by squaring a smooth cutoff. For $z\in\R^n$ put
\begin{align}
 h_i(z)&=e_0\left(\sum_{j=i}^n c_0(z_j/h)^2\right),\\
 \psi_1(z)&=h_1(z),
 &\psi_i(z)&=\Psi(z_{i-1}/h)h_i(z)\quad(i\ge2),\label{eq:selectors}\\
 r_i(z)&=-Lh^2\psi_i(z)\Phi(z_i/h).
 \label{eq:negative-residual}
\end{align}
In particular every $r_i$ is nonpositive.

\begin{lemma}[Selector bounds and one active jet]\label[lemma]{lem:selectors}
At every point, at most one selector $\psi_i$ has a nonzero jet through any prescribed fixed derivative order. The same holds for $r_i$. With constants depending only on the fixed profiles,
\begin{align}
 0\le\psi_i&\le e,&
 \norm{\nabla\psi_i}&\le C h^{-1}\sqrt{\psi_i},&
 \norm{D^k\psi_i}&\le C_k h^{-k}\quad(1\le k\le3),\label{eq:selector-bounds}\\
 |r_i|&\le C Lh^2\psi_i,&
 \norm{\nabla r_i}&\le C Lh\sqrt{\psi_i}.
 \label{eq:residual-first}
\end{align}
Consequently the centered residual
\[
 \mathcal R_{n,h}(z)=\sum_{i=1}^n r_i(z)+Lh^2\Phi(0)
\]
satisfies $\mathcal R_{n,h}(0)=0$ and
\[
 \norm{D^k\mathcal R_{n,h}}_\infty\le C_k Lh^{2-k},
 \qquad 0\le k\le3.
\]
If $|z_j|\le h/4$ for every $j>r$, then the gradient of $C_{n,h}-\mathcal R_{n,h}$ is supported on $[r]$, while the gradient of $\mathcal R_{n,h}$ is supported on $[r+1]$.
\end{lemma}
\begin{proof}
Let $k$ be the largest index with $|z_k|\ge h/2$, with $k=0$ if none exists. For $i\le k$, the tail-mask argument is at least one, so $h_i$ and all its jets vanish. For $i\ge k+2$, the predecessor lies in the flat region of $\Psi$, so all selector jets vanish. Thus only $i=k+1$ can remain. If $k=n$, no selector remains.

Where derivatives of $e_0$ can contribute, the inner sum is at most one. Its first three derivative norms are dimension independent: for example its gradient squared is bounded by $C h^{-2}\sum c_0^2\le C h^{-2}$. The Hessian is diagonal and the higher derivative estimates follow similarly, followed by the chain rule. The square-root gradient estimate follows from $|\Psi'|\le C\sqrt\Psi$, $|e_0'|\le C\sqrt{e_0}$, and the bounded ranges of the factors. The residual estimates follow by differentiating $-Lh^2\psi_i\Phi(z_i/h)$ and using the one-active-jet property. In a tail with $|z_j|\le h/4$, all cutoff derivatives vanish, the outgoing chain term is exactly canceled in the carrier, and the added penalty $p$ is flat there. This proves the support statements.
\end{proof}

For $R>0$ define
\[
 B_R(v)=\frac{v}{\sqrt{1+\norm v^2/R^2}},
 \qquad
 P_{R,\nu}(v)=\nu R^2\left(\sqrt{1+\norm v^2/R^2}-1\right).
\]
Write $\alpha=(1+\norm v^2/R^2)^{-1/2}$. Then
\[
 DB_R=\alpha I-\alpha^3vv^\top/R^2,
 \qquad \nabla P_{R,\nu}=\nu B_R(v).
\]
For an orthonormal frame $V$, let $z=V^\top B_R(v)$ and let $p$ denote a gradient in the $z$ coordinates. The ambient gradient of $P_{R,\nu}(v)+J(V^\top B_R(v))$ obeys
\begin{equation}
 \ip{\nabla_v(P+J\circ V^\top B_R)}{V_j}
 =\alpha p_j+c_{\rm rad}z_j,
 \qquad
 c_{\rm rad}=\nu-\frac{\alpha^2}{R^2}\ip v{Vp}.
 \label{eq:coordinate-wrapper}
\end{equation}
Its radial derivative is
\begin{equation}
 \frac{\nu\norm v}{\sqrt{1+\norm v^2/R^2}}
 +\alpha^3\ip{Vp}{v/\norm v}.
 \label{eq:radial-wrapper}
\end{equation}

\subsection{The macro objective}

Fix a macro length $S\ge1$, a width $b>0$, and an orthonormal frame $U$. Put
\[
 a=Lb^2,
 \qquad E=Lb,
 \qquad R_m=K_m b\sqrt S,
 \qquad \nu_m=\nu_0 L,
\]
where $\nu_0>0$ is a sufficiently small fixed number and $K_m$ is a sufficiently large fixed number. Define
\begin{align}
 Z(x)&=U^\top B_{R_m}(x),\\
 H_m(x)&=C_{S,b}(Z(x))+P_{R_m,\nu_m}(x),\\
 \psi_i(x)&=\psi_i(Z(x)),
 &R_i(x)&=r_i(Z(x)),\\
 A_m(x)&=H_m(x)-\sum_{i=1}^S R_i(x).
 \label{eq:macro-defs}
\end{align}
The uncentered choice of the $R_i$ is deliberate: it preserves $R_i\le0$. A constant shift of the carrier is immaterial.

Composition with $B_{R_m}$ preserves
\begin{equation}
 \norm{\nabla_x\psi_i}\le \frac{C}{b}\sqrt{\psi_i},
 \qquad
 \norm{\nabla_x R_i}\le CE\sqrt{\psi_i}.
 \label{eq:wrapped-selectors}
\end{equation}
The one-active-jet property is also preserved.

\begin{lemma}[One-sided macro end-game]\label[lemma]{lem:macro-endgame}
The constants $\nu_0,K_m$ can be fixed numerically so that
\begin{equation}
 Z_S(x)<b\quad\Longrightarrow\quad
 \norm{\nabla H_m(x)}\ge c_m E,
 \label{eq:macro-endgame}
\end{equation}
where $c_m>0$ is numerical. Moreover $H_m$ is coercive and its initial gap is at most $CaS$.
\end{lemma}
\begin{proof}
The gap bound follows because the negative-side penalty and radial penalty are nonnegative and the bounded chain is bounded below by $-CaS$. The radial penalty proves coercivity.

For the gradient bound, first suppose $\norm x\le R_m$, so $\alpha\ge1/\sqrt2$. All coordinates of $p=\nabla C_{S,b}(Z)$ are nonpositive. On a negative coordinate $Z_j$, the derivative of its outgoing predecessor gate is zero. Hence
\[
 \left(\sum_j Z_jp_j\right)_+
 \le C Lb^2S+Lb\sqrt S\norm Z.
\]
Here the first term uses $\sup_t|t|\Phi'(t)<\infty$, and the second is the contribution of $|p'|\le1$. Since $\ip x{Up}=\alpha^{-1}\ip Zp$, this estimate and \eqref{eq:chain-reg} imply
\[
 \nu_m-C L(K_m^{-1}+K_m^{-2})
 \le c_{\rm rad}\le\nu_m+C L/K_m.
\]
Choose $K_m$ after $\nu_0$ so that
\[
 \nu_m/2\le c_{\rm rad}\le2\nu_m.
\]
Let $j$ be the first coordinate with $Z_j<b$. If $Z_j<0$, \eqref{eq:first-small} and positivity of $c_{\rm rad}$ give a component at most $-E/\sqrt2$. If $0\le Z_j<b$, the same component is at most $-E/\sqrt2+2\nu_0E$. Choose $\nu_0$ small enough to leave a fixed margin.

If $\norm x>R_m$, \eqref{eq:radial-wrapper} and \eqref{eq:chain-reg} give a radial derivative at least
\[
 \nu_mR_m/\sqrt2-C E\sqrt S.
\]
Increasing the numerical $K_m$ makes this at least a fixed multiple of $E$. This proves the lemma.
\end{proof}

\subsection{Micro cores, common wrapper, and terminal receiver}

Choose a small numerical parameter $\zeta>0$ later. Let
\[
 \kappa\ge\kappa_0,
 \qquad
 \beta=b/\kappa,
 \qquad
 g_0=L\beta=E/\kappa,
 \qquad
 M=\lfloor\zeta^6\kappa^2\rfloor\ge1.
\]
All micro frames $V_i\in\R^{d_u\times M}$ use a common lower coordinate $u$. They need not have mutually orthogonal spans. Put
\[
 R_u=C_u b/\zeta,
 \qquad \nu=\zeta L,
 \qquad B_u=B_{R_u},
 \qquad P_u=P_{R_u,\nu},
\]
and write
\[
 z^{(i)}(u)=V_i^\top B_u(u),
 \qquad s_i(u)=z^{(i)}_M(u).
\]
Notice that $\beta\sqrt M\le\zeta^3b$ and $R_u/(\beta\sqrt M)$ is a numerical constant of order at least $C_u\zeta^{-4}$; neither radius nor amplitude grows with $S$.

Fix a smooth nondecreasing $q_0:\R\to[0,1]$ with $q_0=0$ on $(-\infty,1]$ and $q_0=1$ on $[2,\infty)$. Set
\[
 q(s)=q_0(s/b),
 \qquad Q_i(u)=q(s_i(u)).
\]
Thus $q'\ge0$ and $\norm{q^{(j)}}_\infty\le C_jb^{-j}$.

Choose smooth cutoff profiles $j_0,k_0$ with
\[
 j_0=0\text{ on }(-\infty,1/4],\quad
 j_0=1\text{ on }[1/2,\infty),
\]
\[
 k_0=1\text{ on }(-\infty,2],\quad
 k_0=0\text{ on }[3,\infty),
 \qquad 0\le j_0,k_0\le1.
\]
For a numerical $A>0$ define $D(0)=0$ and
\begin{equation}
 D'(s)=-A(\nu s+g_0)j_0(s/\beta)k_0(s/b).
 \label{eq:D-profile}
\end{equation}
It follows that $D=0$ on $(-\infty,\beta/4]$, $D'\le0$, and
\begin{equation}
 D'(s)=-A(\nu s+g_0)
 \qquad(\beta/2\le s\le2b).
 \label{eq:D-active}
\end{equation}
The receiver stops changing only after the terminal gate is already one. Direct differentiation gives
\begin{align}
 \norm D_\infty&\le C A(\nu b^2+g_0b),
 &\norm{D'}_\infty&\le CA(\nu b+g_0),
 &\norm{D''}_\infty&\le CA L.
 \label{eq:D-bounds}
\end{align}
The turn-on at width $\beta$ is harmless: on its support $\nu s+g_0=O(\nu\beta+g_0)$, so its derivative cost is $O(\nu+g_0/\beta)=O(L)$.

Define the bounded selected micro core
\begin{equation}
 W_i(u)=C_{M,\beta}(z^{(i)}(u))+D(s_i(u)).
 \label{eq:W-def}
\end{equation}
The common radial penalty is not included in $W_i$ and is never multiplied by a selector. By \eqref{eq:chain-value},
\begin{equation}
 |W_i(u)|\le\delta_W a,
 \qquad
 \delta_W=C\left(\zeta^6+C_u\zeta^2+A\zeta+A/\kappa\right).
 \label{eq:W-value}
\end{equation}
This bound is global: the negative-side penalty costs at most
\[
 L\beta\sqrt M R_u\le C_u\zeta^2a.
\]

For the exact lower lift, define the centered micro residual
\[
 \rho_i(u)=\mathcal R_{M,\beta}(z^{(i)}(u)).
\]
The exact upper carrier for this block is $W_i-\rho_i$. Both the negative-side penalty and the terminal receiver belong to this carrier. The receiver and gate are flat whenever an unexposed terminal coordinate lies in $[-\beta/4,\beta/4]$.

\subsection{Hyperobjective and the global margin}

Define first the function of $(x,u)$
\begin{equation}
 \mathcal F(x,u)
 =A_m(x)+P_u(u)
 +\sum_{i=1}^S\psi_i(x)W_i(u)
 +\sum_{i=1}^S R_i(x)Q_i(u).
 \label{eq:F}
\end{equation}
The intended hyperobjective in the actual outer variables $(x,v)$ is
\begin{equation}
 F(x,v)=\mathcal F(x,\kappa v).
 \label{eq:outer-F}
\end{equation}
Equivalently, an exact upper primitive and residual for the compensated lift are
\begin{align}
 f(x,v,u,\sigma)
 &=A_m(x)+P_u(u)
   +\sum_i\psi_i(x)(W_i(u)-\rho_i(u))
   +\sum_iR_i(x)Q_i(u)+L\sigma,\label{eq:upper-primitive}\\
 r(x,u)&=\sum_i\psi_i(x)\rho_i(u).
 \label{eq:lower-residual}
\end{align}
If the lower compensation gives $u^*(x,v)=\kappa v$ and $\sigma^*(x,v)=r(x,\kappa v)/L$, substitution gives exactly \eqref{eq:outer-F}.

\begin{lemma}[Common-wrapper micro certificate]\label[lemma]{lem:micro}
Fix $w_*>0$. One can choose $A,C_u,\zeta,\kappa_0$ in that order, with $A,C_u,\kappa_0$ large and $\zeta$ small, such that the following holds. At any $(x,u)$ with an active index $i$ satisfying
\[
 \psi_i(x)\ge w_*\quad\hbox{and}\quad Q_i(u)<1,
\]
we have
\begin{equation}
 \norm{\nabla_u\mathcal F(x,u)}\ge c_u g_0,
 \label{eq:micro-margin}
\end{equation}
where $c_u>0$ depends only on $w_*$ and the fixed profiles.
\end{lemma}
\begin{proof}
Only the active index contributes $u$-dependent selected terms. Write $w=\psi_i(x)$, $R=R_i(x)\le0$, $z=V_i^\top B_u(u)$, and $s=z_M$. In chain coordinates put
\[
 p=w\nabla C_{M,\beta}(z)
   +\bigl[wD'(s)+Rq'(s)\bigr]e_M.
\]
Every coordinate of $p$ is nonpositive. Moreover,
\begin{equation}
 \norm p\le C_A E,\label{eq:p-bound}
\end{equation}
where $C_A$ depends on $A$ and the fixed profiles but not on $S,M,\kappa,\zeta$ when $\zeta\le1$, $\kappa\ge4$. This uses $w\le e$, $|R|\le Ca$, $|q'|\le C/b$, \eqref{eq:chain-reg}, and \eqref{eq:D-bounds}.

If $\norm u>R_u$, the radial derivative \eqref{eq:radial-wrapper} is at least
\[
 \nu R_u/\sqrt2-\norm p
 \ge(C_u/\sqrt2-C_A)E.
\]
Choose $C_u\ge8C_A$ to get more than the required margin.

Now let $\norm u\le R_u$, so $\alpha\ge1/\sqrt2$. The receiver and gate contributions to $\ip zp$ are nonpositive, because they have nonzero derivative only for positive $s$, $D'\le0$, $R\le0$, and $q'\ge0$. For the remaining chain and negative-side penalty,
\[
 (\ip zp)_+
 \le C L\beta^2M+C L\beta\sqrt M\norm z.
\]
Consequently
\begin{align*}
 c_{\rm rad}
 &=\nu-\frac{\alpha}{R_u^2}\ip zp\\
 &\ge\nu-CL\left(\frac{\zeta^8}{C_u^2}
                       +\frac{\zeta^4}{C_u}\right)
 \ge\nu/2
\end{align*}
for sufficiently small $\zeta$. In the other direction, \eqref{eq:p-bound} and the choice of $C_u$ give
\[
 c_{\rm rad}\le\nu+\norm p/R_u\le\tfrac54\nu.
\]

If $s<\beta$, choose the first $j$ with $z_j<\beta$. By \eqref{eq:first-small}, $p_j\le-wg_0$. If $z_j<0$, the radial correction is nonpositive. If $0\le z_j<\beta$, the coordinate formula gives
\[
 \ip{\nabla_u\mathcal F}{V_{i,j}}
 \le-\frac{w_*g_0}{\sqrt2}+\frac54\nu\beta
 =-\frac{w_*g_0}{\sqrt2}+\frac54\zeta g_0.
\]
Choose $\zeta\le w_*/10$. This leaves a fixed negative multiple of $w_*g_0$.

If instead $s\ge\beta$ and $Q_i(u)<1$, then $s<2b$ and \eqref{eq:D-active} applies. All other terminal contributions are nonpositive, so
\[
 p_M\le-w_*A(\nu s+g_0).
\]
Hence
\[
 \ip{\nabla_u\mathcal F}{V_{i,M}}
 \le-\frac{w_*A}{\sqrt2}(\nu s+g_0)+\frac54\nu s.
\]
Taking $A\ge8/w_*$ makes this at most a fixed negative multiple of $g_0$. This proves the lemma in every case.
\end{proof}

\begin{theorem}[Global macro--micro gradient margin]\label[theorem]{thm:global}
There are numerical choices of all the constants above, including $\kappa_0$, and a numerical $c_*>0$, such that for every $S\ge1$, every $L,b>0$, every $\kappa\ge\kappa_0$, every admissible choice of frames, and every outer point $(x,v)$,
\begin{equation}
 Z_S(x)<b
 \quad\Longrightarrow\quad
 \norm{\nabla F(x,v)}\ge c_* Lb.
 \label{eq:global-margin}
\end{equation}
The conclusion is global in both outer variables. In particular it applies to arbitrary terminal outputs, not only points on a designated trajectory.
\end{theorem}
\begin{proof}
Differentiate the identity $A_m=H_m-\sum R_i$. Because at most one selector/residual jet is active,
\begin{equation}
 \nabla_x\mathcal F
 =\nabla H_m-(1-Q_i)\nabla R_i+W_i\nabla\psi_i
 \label{eq:macro-error}
\end{equation}
when an index $i$ is active. If none is active, $\nabla_x\mathcal F=\nabla H_m$ and the macro end-game already proves the result.

Assume an active index exists. By \eqref{eq:wrapped-selectors} and \eqref{eq:W-value},
\begin{equation}
 \norm{\nabla_x\mathcal F-\nabla H_m}
 \le CE\bigl[(1-Q_i)+\delta_W\bigr]\sqrt{\psi_i}.
 \label{eq:macro-error-bound}
\end{equation}
Choose a numerical $w_*>0$ so small that $2C\sqrt{w_*}\le c_m/2$. If $\psi_i\le w_*$ and $\delta_W\le1$, then \eqref{eq:macro-endgame} and \eqref{eq:macro-error-bound} give $\norm{\nabla_xF}\ge c_mE/2$.

Choose a numerical $\gamma>0$ sufficiently small that $2C\sqrt e\,\gamma\le c_m/2$. If $Q_i\ge1-\gamma$ and $\delta_W\le\gamma$, the same conclusion follows.

In the remaining case, $\psi_i>w_*$ and $Q_i<1-\gamma<1$. Lemma~\ref{lem:micro} gives
\[
 \norm{\nabla_vF(x,v)}
 =\kappa\norm{\nabla_u\mathcal F(x,\kappa v)}
 \ge\kappa c_u g_0=c_u E.
\]
Thus \eqref{eq:global-margin} holds with $c_*:=\min\{c_m/2,c_u\}$.

For completeness the numerical choices are compatible in the following order. First fix all profiles and the macro constants, obtaining $c_m$ and the constants in \eqref{eq:macro-error-bound}. Next choose $w_*,\gamma$ as above. Set $A\ge8/w_*$. Choose $C_u\ge8C_A$. Then choose $\zeta>0$ small enough for the radial-coefficient inequalities, $\zeta\le w_*/10$, and
\[
 C(\zeta^6+C_u\zeta^2+A\zeta)\le\gamma/2.
\]
Finally choose $\kappa_0$ large enough that $\kappa_0\ge4$, $\zeta^6\kappa_0^2\ge64$, and $CA/\kappa_0\le\gamma/2$. Then $M\ge1$, all receiver intervals are correctly ordered, and $\delta_W\le\gamma$.
\end{proof}

\subsection{Primitive budgets and the exact hierarchical mean}
\label{sec:hierarchy-budgets}

Let the outer variable be $(x,v)\in\mathbb R^{D}\times\mathbb R^d$ and
the lower variable be $(u,s)\in\mathbb R^d\times\mathbb R$.  Fix
$\mu=L/\kappa$, $1\le\vartheta\le\kappa$, and take a width $b>0$. Put
\[
 \beta=b/\kappa,\qquad a=Lb^2.
\]
The rate specialization later chooses $b=C_b\epsilon/L$ with a fixed
numerical $C_b$; the present budget statements are expressed in $b$.
Choose a fixed numerical $\zeta>0$, an integer
$1\le M\le C\zeta^6\kappa^2$, and a common lower wrapper
\[
 B_u(u)=\frac{u}{\sqrt{1+\|u\|^2/R_u^2}},\qquad
 R_u=C_u b/\zeta,
\]
\[
 P_u(u)=\nu R_u^2\bigl(\sqrt{1+\|u\|^2/R_u^2}-1\bigr),
 \qquad \nu=\zeta L.
\]
Here $C_u$ and $\zeta$ are fixed numerical choices, not additional varying
problem parameters.  In particular
\[
 \|DB_u\|\le1,\quad \|D^jB_u\|\le C_jR_u^{1-j}\ (j=2,3),
 \quad \|\nabla P_u\|\le\nu R_u=CLb,
 \quad \|\nabla^2P_u\|\le\nu.
\]

Let $V_i\in\mathbb R^{d\times M}$ be the micro frames, and write
$z_i(u)=V_i^\top B_u(u)$.  Their mutual relative orientations do not
enter any of the following budget bounds.  Let $\psi_i(x)$ be the macro
selectors and $R_i(x)$ the uncentered macro residual summands.
They are the tail-mask functions from the construction, composed with the
macro frame and macro radial map.  At any point at most one index has a
nonzero jet through order three.  Their dimension-independent bounds are
\begin{equation}\label{hier:eq:macro-budgets}
 \|D^j\psi_i\|_\infty\le C_jb^{-j}\quad(0\le j\le3),
 \qquad
 \|D^j R_i\|_\infty\le C_ja b^{-j}\quad(0\le j\le2).
\end{equation}
The macro carrier $A_m$ has
$\operatorname{Lip}(\nabla A_m)\le CL$.  Its full gradient may
grow with the macro gap; no primitive assumption bounds that gradient.

For completeness, the one-active-jet assertion follows by taking the
last macro coordinate of magnitude greater than $b/2$.  Earlier tail
masks have a flat value zero; later predecessor gates have a flat value
zero.  Derivatives of a tail mask are dimension independent because on
their support the sum of the squared scalar cutoffs is at most one.
Composition with the macro radial map preserves
\eqref{hier:eq:macro-budgets}, since its radius is at least a fixed
multiple of $b$.

\subsubsection{The bounded selected micro carrier}

Let $H_i^{\rm mic}$ be the one-sided micro chain at width $\beta$ and
amplitude $L\beta^2$, including the coordinate penalty $p$, where
$p\ge0$, $p(0)=0$, $p'=0$ on $[-1/2,\infty)$, and $p'$ is bounded and
equals $-1$ on $(-\infty,-1]$.  Fixed numerical scale factors in the
chain are absorbed into the constants below.  Let $R^{\rm mic}$ be its
centered tail residual, excluding the coordinate penalty, and define
\[
 r_i(u)=R^{\rm mic}(z_i(u)),\qquad
 r_i(0)=0.
\]
The residual certificate gives
\begin{equation}\label{hier:eq:micro-residual-budgets}
 \|D^j r_i\|_\infty\le C_jL\beta^{2-j},\qquad 0\le j\le3.
\end{equation}
This remains valid with the common radius $R_u$ in place of the original
radius proportional to $\beta\sqrt M$: every wrapper error is dominated
by the displayed expression because $R_u\ge c\beta$.

Write $t_i(u)=V_{i,M}^\top B_u(u)$.  The terminal receiver $D$ and release
gate $Q$ obey
\[
 |D|\le C(\zeta a+a/\kappa),\quad
 |D'|\le C(\zeta Lb+L\beta),\quad |D''|\le CL,
\]
\[
 |Q|\le1,\quad |Q'|\le C/b,\quad |Q''|\le C/b^2.
\]
Only these size estimates, rather than the separate sign and gradient
margin conditions on $D,Q$, are needed for the present certificate.
The release gate varies over scale $b$, not scale $\beta$.
Define the selected carrier
\[
 C_i(u)=H_i^{\rm mic}(z_i(u))-r_i(u)+D(t_i(u)).
\]
Then
\begin{equation}\label{hier:eq:selected-carrier}
 \|C_i\|_\infty\le Ca,\qquad
 \|DC_i\|_\infty\le CLb,\qquad
 \|D^2C_i\|_\infty\le CL.
\end{equation}
Indeed the bounded chain terms have total magnitude $O(ML\beta^2)$.
The coordinate penalty satisfies $p(t)\le C|t|$, and hence, on the
wrapper image,
\[
 L\beta^2\sum_{j=1}^M p(z_j/\beta)
 \le CL\beta\sqrt M\,R_u\le C\zeta^2a.
\]
The core gradient has norm at most $CL\beta\sqrt M\le C\zeta^3Lb$,
while its Hessian has norm at most $CL$ by the bounded degree of the
chain Hessian and the diagonal penalty Hessian.  Wrapper Hessian errors
are bounded by $CL\beta\sqrt M/R_u\le CL$.  Subtracting the residual and
adding the receiver proves \eqref{hier:eq:selected-carrier}.

\subsubsection{Full upper-primitive regularity}

The actual upper mean is
\begin{equation}\label{hier:eq:f-mean}
 f(x,v,u,s)=A_m(x)+P_u(u)
       +\sum_i\psi_i(x)C_i(u)
       +\sum_iR_i(x)Q(t_i(u))+Ls.
\end{equation}
It is independent of the outer variable $v$ except through the lower
response when the hyperobjective is formed.

\begin{lemma}[Upper primitive]\label[lemma]{hier:lem:upper-primitive}
If $b\le c$ for a fixed numerical $c$, then
\[
 \sup\|f_u\|\le CLb,\qquad
 \sup\|f_{(u,s)}\|\le L+CLb\le C'L,
 \qquad \operatorname{Lip}(\nabla f)\le C'L.
\]
All constants are independent of the number of macro stages and the
ambient dimensions.
\end{lemma}
\begin{proof}
At any point only one selector index has a nonzero jet.  For the first
sum in \eqref{hier:eq:f-mean}, its $xx,xu,uu$ Hessian blocks are bounded,
respectively, by
\[
 C a/b^2,\qquad C(Lb)/b,\qquad CL,
\]
all of which are $O(L)$.  For the release sum, the corresponding bounds
are again
\[
 C a/b^2,\qquad C(a/b)(1/b),\qquad C a/b^2.
\]
Here the second derivative of $Q(t_i(u))$ also contains a wrapper term
$O(1/(bR_u))$, dominated by $O(1/b^2)$.
The macro carrier and the common penalty have $O(L)$ Hessians.  These
block bounds imply an $O(L)$ joint Hessian norm, without a factor from
the number of blocks.  The lower $u$ gradient is at most
$\nu R_u+CLb+Ca/b=CLb$.  Finally $f_s=L$ exactly, which explains why
the full lower-gradient budget is $O(L)$ rather than $O(\epsilon)$.
\end{proof}

\subsubsection{Joint residual bounds and multiscale decomposition}

Define the residual actually transmitted through the lower response by
\[
 r(x,u)=\sum_i\psi_i(x)r_i(u).
\]
Derivatives here and below are with respect to the joint variable
$(x,u)$, using Euclidean multilinear operator norms.  The product rule,
one-active-jet property, $\beta\le b$, and
\eqref{hier:eq:micro-residual-budgets} imply
\begin{equation}\label{hier:eq:joint-r}
 \|D^j r\|_\infty\le C_jL\beta^{2-j},\qquad 0\le j\le3.
\end{equation}
For example a term with $p$ macro derivatives and $q$ lower derivatives
is bounded by
\[
 C b^{-p}L\beta^{2-q}\le CL\beta^{2-p-q}.
\]
The same observation applies to every flattened residual and every
multiscale difference.  In particular no additional factor $\kappa$ or
number-of-stages factor is introduced by macro derivatives.

Fix a sufficiently small numerical width constant $c_{\rm ms}\le1/16$, independently of the gradient-margin constant $c_*$. Use the following multiscale widths with $\vartheta$ itself:
\[
 J=\left\lfloor\log_4(\max\{1,\log\vartheta\})\right\rfloor,
 \quad a_j=1-4^{-(j+1)},\quad b_j=1-\tfrac12 4^{-j},
\]
\[
 \delta_j=c_{\rm ms}\beta\vartheta^{-a_j}\quad(0\le j<J),
 \qquad w_j=c_{\rm ms}\beta\vartheta^{-b_j}\quad(0\le j\le J).
\]
Choose an even smooth cutoff $\omega$ that is zero on $[-1,1]$, one outside $(-2,2)$, and takes values in $[0,1]$. For $0<\delta\le\beta/16$, define
\[
 c_\delta(t)=t\omega(t/\delta),\qquad
 R^{\rm mic}_\delta(z)=R^{\rm mic}\bigl(c_\delta(z_1),\ldots,c_\delta(z_M)\bigr),
 \qquad r_{i,\delta}(u)=R^{\rm mic}_\delta(V_i^\top B_u(u)).
\]
Flattening is applied to the latent micro coordinates before composition with the radial map. For $J=0$, set $E_{i,0}=r_i$. For $J\ge1$, define
\[
 E_{i,0}=r_{i,\delta_0},\qquad
 E_{i,j}=r_{i,\delta_j}-r_{i,\delta_{j-1}}\ (1\le j<J),\qquad
 E_{i,J}=r_i-r_{i,\delta_{J-1}},
 \]
In either case, set $e_j(x,u)=\sum_i\psi_i(x)E_{i,j}(u)$ for $0\le j\le J$.

Then $\sum_{j=0}^J e_j=r$, and $w_J\asymp\beta/\vartheta$.
At any point at most one macro jet and at most three resolution jets
are active.  The full joint derivative bounds for $e_j$ are therefore
the multiscale bounds at chain scale $L$ proved below in \cref{fb:lem:multiscale}. The product-rule estimate above shows explicitly that its proof applies to joint $(x,u)$ derivatives here.

Choose a fixed odd compactly supported $C^\infty$ profile $n$ equal to the identity on $[-1/4,1/4]$, write $n_w(s)=wn(s/w)$, and put
\begin{equation}\label{hier:eq:K-mean}
 \mathcal K(x,u,s)=\frac{\mu}{2L^2}r(x,u)^2
       -\frac\mu L\sum_{j=0}^J e_j(x,u)n_{w_j}(s).
\end{equation}
The third-order-block calculation in \cref{fb:lem:multiscale}, with $\alpha=1$ and joint $(x,u)$ derivatives, yields
\begin{equation}\label{hier:eq:K-budgets}
 \|D^2\mathcal K\|_\infty
       \le C\mu(\beta\sqrt\vartheta+\beta^2),\qquad
 \|D^3\mathcal K\|_\infty
       \le C\mu(\vartheta+\beta).
\end{equation}
For clarity, a component's four third-order block bounds are
$(\mu/L)(w_jD_{3,j},D_{2,j},D_{1,j}/w_j,A_j/w_j^2)$.
The exponent identities and bounded overlap give $O(\mu\vartheta)$;
the compensation contributes $O(\mu\beta)$.

\subsubsection{Exact mean response}

Define the actual lower mean by
\begin{equation}\label{hier:eq:g-mean}
 g(x,v,u,s)=\frac\mu2\|u\|^2-L\langle v,u\rangle
            +\frac\mu2s^2+\mathcal K(x,u,s).
\end{equation}

\begin{theorem}[Hierarchical mean certificate]\label[theorem]{hier:thm:mean}
For $\beta\vartheta\le c$ sufficiently small and $b\le c$, the pair
\eqref{hier:eq:f-mean}--\eqref{hier:eq:g-mean} satisfies
\[
 \|f_y\|_\infty\le CL,\quad
 \operatorname{Lip}(\nabla f),\operatorname{Lip}(\nabla g)\le CL,
 \quad \nabla^2_{yy}g\succeq(\mu/2)I,
 \quad \operatorname{Lip}(\nabla^2g)\le C\mu\vartheta.
\]
The exact lower response is
\[
 u^*(x,v)=\kappa v,\qquad
 s^*(x,v)=r(x,\kappa v)/L.
\]
Consequently the actual hyperobjective is
\begin{align*}
 F(x,v)={}&A_m(x)+P_u(\kappa v)\notag\\
 &+\sum_i\psi_i(x)
     [H_i^{\rm mic}(V_i^\top B_u(\kappa v))+D(t_i(\kappa v))]\notag\\
 &+\sum_iR_i(x)Q(t_i(\kappa v)).
\end{align*}
It has $\operatorname{Lip}(\nabla F)\le CL\kappa^2$.
At $(x,v)=(0,0)$ the exact lower solution is $(u,s)=(0,0)$.
With $S\ge1$ macro stages, $F$ is coercive and
$F(0,0)-\inf F\le CSa$.
\end{theorem}
\begin{proof}
The full quadratic Hessian in \eqref{hier:eq:g-mean} has norm $O(L)$,
and its lower Hessian is $\mu I$.  Equation
\eqref{hier:eq:K-budgets} and $\beta\vartheta\le c$ make the lower
perturbation at most $\mu/2$, and give the asserted third-derivative
bound.  This proves global lower strong convexity and all lower budgets.
The upper budgets are \cref{hier:lem:upper-primitive}.

Moreover $|r|/L\le C\beta^2\le w_J/4$ under the same condition.
Every normal feature is therefore linear at $s=r(x,u)/L$.  At that
point
\[
 g_s=\mu s-\mu r/L=0,\qquad
 g_u=\mu u-Lv+\mu r\nabla_u r/L^2
                      -\mu s\nabla_u r/L=\mu u-Lv.
\]
Strong convexity gives the unique response.  The upper term $Ls^*=r$
cancels precisely the selected centered residuals in $C_i$, proving the
displayed expression for $F$.  This expression, before composing
$u=\kappa v$, has joint $(x,u)$ Hessian $O(L)$ by the same estimates as
the upper primitive, now with the full micro core in place of its
carrier.  The affine composition has operator norm at most $\kappa$,
so $\operatorname{Lip}(\nabla F)\le CL\kappa^2$.  Finally each centered
$r_i(0)=0$, and hence $r(x,0)=0$ for all $x$.

For the gap, the macro core is bounded below by $-CSa$: its coordinate
penalty and radial penalty are nonnegative.  Since every uncentered
macro residual is nonpositive, $A_m=H_m-\sum_i
R_i\ge H_m$.  The selected micro core plus receiver
is bounded in absolute value by $Ca$, as is the entire release sum;
only one macro value is active.  Also $P_u\ge0$.  Thus $F\ge-CSa-Ca$.
At the origin $A_m=0$, the gate and receiver vanish, and the
only selected micro core has value $-L\beta^2\Phi(0)$, so
$|F(0,0)|\le Ca$.  This proves the gap bound.  Finally
$A_m$ is coercive in $x$, $P_u(\kappa v)$ is coercive in $v$,
and all remaining selected terms are bounded.  Hence $F$ is coercive
in the joint outer variable.
\end{proof}

The displayed hyperobjective is exactly \eqref{eq:outer-F}.
Thus the global margin in \cref{thm:global} applies to this mean pair.
The information argument below uses the same centered micro residual.

\subsubsection{Conservative full-gradient information envelope}

Put $\tau=w_J/8$.  Flatten each micro residual at this resolution and set
\[
 r_\tau(x,u)=\sum_i\psi_i(x)r_{i,\tau}(u),\qquad d=r-r_\tau.
\]
Then
\[
 \|d\|_\infty\le CL\beta\tau,\qquad
 \|D_{(x,u)}d\|_\infty\le CL\beta,\qquad
 \|r_\tau\|_\infty\le CL\beta^2,\qquad
 \|D_{(x,u)}r_\tau\|_\infty\le CL\beta.
\]
The macro derivatives are included here.  For example
$\|\nabla_xd\|\le CL\beta\tau/b\le CL\beta$.

Define $\mathcal K_0$ by replacing $r^2$ with $r_\tau^2$ and replacing
the final atom $e_J$ with $e_J-d$.  The difference is exactly
\[
 \mathcal D=\mathcal K-\mathcal K_0
 =\frac\mu{2L^2}(r^2-r_\tau^2)-\frac\mu Ld\,n_{w_J}(s).
\]
For $z=(x,u)$, direct differentiation gives
\[
 \nabla_z\mathcal D
 =\frac\mu{L^2}(d\nabla_zr+r_\tau\nabla_zd)
   -\frac\mu Ln_{w_J}(s)\nabla_zd,\qquad
 \partial_s\mathcal D=-\frac\mu Ld n_{w_J}'(s).
\]
It is independent of $v$.  Consequently
\[
 \|\nabla\mathcal D\|_\infty
 \le C\mu(\beta^2\tau+\beta^3+\beta w_J)
 \le C_B\frac{\mu\beta^2}{\vartheta}=:B
\]
under $\beta\vartheta\le c$ and $w_J\asymp\tau\asymp\beta/\vartheta$.
This is a bound on the \emph{full} gradient, including the macro outer
component, not only on the lower-gradient block.

With
\[
 p=\frac{B^2}{\sigma_g^2+B^2},\qquad \xi\sim\operatorname{Ber}(p),
\]
the scalar sample potential
\[
 g_\xi=g_{\rm quad}+\mathcal K_0+(\xi/p)\mathcal D
\]
has globally unbiased gradient and global squared-error variance at most
$\sigma_g^2$.  Returning the exact gradient of $f$ gives a conservative
original SFO.  In particular
\[
 B\asymp\frac{\mu\beta^2}{\vartheta}
       \asymp\frac{\epsilon^2}{\vartheta\kappa^3L}
\]
for the chosen numerical envelope and the rate specialization
$b=C_b\epsilon/L$.  No macro-stage-count factor enters
this calculation.

\subsection{Hierarchical projection coupling for macro--micro blocks}
\label{sec:hierarchy-oracle}

This section proves the information accounting for the hierarchy above.
The global primitive budgets and stationary-point barrier are supplied
by \cref{hier:thm:mean,thm:global}. Under the explicit
locality conditions verified below, the information argument charges $SM$
Bernoulli activations, not merely $S$ or $M$.

\paragraph{Notation for the information argument.}
In this section $w$ denotes the outer micro variable called $v$ above,
$b_x=B_{R_m}$ and $b_u=B_u$. The latent macro carrier is
$C_{S,b}-\sum_i r_i$; its public radial term is written separately.
The release scalar is $q_i(z)=-Lb^2\Phi(z_i/b)$, so
$R_i=\psi_iq_i$. These are not different functions from those already
used in the mean and margin certificates.

\subsubsection{Frames, scalar observations, and locality conditions}

Let $S,M\ge1$ be integers. Independently sample two Haar orthonormal
frames
\[
 U=(U_1,\ldots,U_S)\in\R^{d_x\times S},\qquad
 V=(V_{i,j}:1\le i\le S,\ 1\le j\le M)
       \in\R^{d_u\times SM}.
\]
All micro columns, including columns of different blocks, belong to the
single orthonormal frame $V$. Use outer variables $(x,w)$ and lower variables $(u,s)$. Suppose the public radial maps $b_x,b_u$ satisfy
\[
 b_x(x)=\alpha_x(\norm x)x,\qquad
 b_u(u)=\alpha_u(\norm u)u,\qquad
 \sup\norm{b_x}\le R_x,\quad\sup\norm{b_u}\le R_u.
\]
Their derivatives therefore transform a latent frame-gradient into a
linear combination of frame columns and the public query direction.
Write
\[
 z=U^\top b_x(x),\qquad v_i=V_i^\top b_u(u),
 \quad V_i=(V_{i,1},\ldots,V_{i,M}).
\]
Choose positive locality thresholds $\tau_x,\tau_u$.

The following are the precise first-derivative locality assumptions used
by the argument. All scalar profiles and scalar coefficients are public
functions of the latent coordinates $(z,v_1,\ldots,v_S)$, the query,
and previously revealed information. In particular they have no further
dependence on unrevealed frame orientations.

\begin{enumerate}
\item \emph{Macro carrier.} Whenever $|z_j|\le\tau_x$ for all $j>r$,
the latent gradient of the macro carrier $A_{\rm macro}$ is supported
on $[r]$. Public radial terms may be added to it.
\item \emph{Macro eligibility.} Under the same tail condition, every
selector $\psi_i(z)$ with $i>r+1$ has zero value and zero first
derivative. For $i\le r+1$, the latent gradient $\nabla\psi_i$ is
supported on $[r]$. The release scalar $q_i(z)$ has latent gradient
supported on $[i]$.
\item \emph{Micro carriers.} Whenever $|(v_i)_j|\le\tau_u$ for
$j>k$, the latent gradient of $A_{{\rm micro},i}$ is supported on
$[k]$. The unflattened residual gradient is supported on $[k+1]$;
the fully inactive normal potential has micro-gradient support in
$[k]$. The latter assertion includes the flattened compensation term
and the coarser normal features, not just one summand.
\item \emph{Terminal gates.} For $|t|\le\tau_u$,
\begin{equation}\label{hier:eq:gate-flat}
 Q(t)=0,\qquad Q'(t)=0,\qquad D'(t)=0.
\end{equation}
The constant value of $D$ on this interval need not be zero.
\end{enumerate}

These conditions apply to the upper primitive
\begin{equation}\label{hier:eq:upper-form}
 \begin{split}
 f={}&A_{\rm macro}(z)+P_x(x)+P_u(u)+Ls\\
 &+\sum_{i=1}^S\psi_i(z)
 \left[A_{{\rm micro},i}(v_i)+D((v_i)_M)
                  +Q((v_i)_M)q_i(z)\right],
 \end{split}
\end{equation}
with public radial terms $P_x,P_u$. The lower sample gradient is assumed
to have the conservative form
\begin{equation}\label{hier:eq:sample-form}
 \nabla g_\xi
 =\nabla g_{\rm quad}+\nabla K_0+
\frac{\xi}{p}\nabla(K-K_0),
 \qquad \xi\sim\operatorname{Ber}(p),
\end{equation}
where the quadratic part is public, $K_0$ is inactive in the sense of
condition 3, and $K$ uses the unflattened micro residuals. Products such
as $(\sum_i\psi_iR_i)^2$ are permitted: their derivatives are scalar
linear combinations of the same selector and residual gradients.
The upper gradient is exact. Fresh Bernoulli variables are independent
of the past and of both frames.

\paragraph{Checking macro eligibility.}
For the one-sided selectors
\[
 \psi_1(z)=h_1(z),\qquad
 \psi_i(z)=\Psi(z_{i-1}/b)h_i(z)\quad(i\ge2),
\]
assume $\Psi$ and its derivative vanish on $(-\infty,1/2]$, and
that every tail-mask coordinate cutoff is flat on
$[-1/4,1/4]$. Then condition 2 holds for $\tau_x\le b/4$:
if $i>r+1$, its predecessor gate is flat zero; and any tail-mask
derivative in an unseen coordinate is zero. A two-sided predecessor
gate also works if both signed gates are flat when the coordinate has
absolute value at most $b/4$. These statements concern full zero jets;
small but nonzero derivatives would not suffice for exact-oracle
coupling.

\paragraph{Why the value condition on $Q$ matters.}
The equation $Q'=0$ near zero alone is insufficient. A nonzero constant
$Q$ would leave the term $\psi_iQ\nabla q_i$ in the exact upper
gradient and could reveal the next macro column immediately. This is
why \eqref{hier:eq:gate-flat} explicitly requires $Q=0$.

\subsubsection{A stronger virtual protocol}

Fix an algorithm and expose its private seed at the start of the
analysis, independently of $U,V$. The virtual protocol maintains two
known subspaces, $W_x\subset\R^{d_x}$ and
$W_u\subset\R^{d_u}$. It knows all projections of every macro column
onto $W_x$ and of every micro column onto $W_u$, together with the full
columns that have been revealed. All such information can be given to
the algorithm; this only strengthens its observations.

Let $a$ be the number of activations so far, capped at $SM$. For
$a<SM$, maintain
\begin{equation}\label{hier:eq:prefix-state}
 r=\lfloor a/M\rfloor,\qquad k=a-rM.
\end{equation}
The known macro columns are $U_1,\ldots,U_r$. All micro columns in
blocks $i\le r$ are fully known; the first $k$ columns in block $r+1$
are known; no later block has a fully revealed micro column.

At each query $(x,w,u,s)$ the protocol does the following, in order.
\begin{enumerate}
\item Insert $x$ into $W_x$ and $w,u$ into $W_u$, ignoring zero
directions and directions already in the corresponding subspace.
Disclose every frame column's new projection onto the enlarged subspaces.
\item Draw the fresh coin $\xi$. If $\xi=1$ and $a<SM$, reveal the
next micro column $V_{r+1,k+1}$ in full and add it to $W_u$. This
revelation is made even if the actual query would have been uninformative.
If the block is now complete, also reveal $U_{r+1}$ in full for free and
add it to $W_x$. Increment $a$ and update \eqref{hier:eq:prefix-state}.
\item Compute all scalar coefficients in the original gradient formulas
from the disclosed projections. Retain only frame-vector contributions
whose columns have been fully revealed, together with all public query
direction and quadratic contributions. Return these virtual gradients.
\end{enumerate}

This defines a protocol for every pair of frames, including on a failed
concentration event. There is no conditioning on the eventual good event
in its definition. The free micro and macro revelations can only make
the protocol more informative. After the algorithm fixes its output
$(\widehat x,\widehat w)$, insert these two directions into $W_x,W_u$
for analysis; the output itself is not changed.

For a horizon of at most $N$ original calls, the numbers of direction
insertions are bounded by
\[
 m_x=N+1,\qquad m_u=2N+1.
\]
The maximum numbers of full-column revelations are $K_x=S$ and $K_u=SM$.

\subsubsection{Two-frame projection concentration}

\begin{lemma}[Uniform projection event]
\label[lemma]{hier:lem:projection}
For $\delta\in(0,1)$, sufficiently large numerical $C$ makes the
following dimension conditions sufficient:
\begin{equation}\label{hier:eq:dimension}
 d_j\ge C\left[
 m_j+K_j+2+
 m_j\left(K_j+\frac{R_j^2}{\tau_j^2}\right)
 \log\frac{16m_jK_j}{\delta}
 \right],\qquad j\in\{x,u\}.
\end{equation}
With probability at least $1-\delta$, at every stage, including just
after a query-direction insertion and before its activation, every
unrevealed macro column and micro column satisfy
\[
 \norm{P_{W_x}U_i}\le\frac{\tau_x}{2R_x},\qquad
 \norm{P_{W_u}V_{i,j}}\le\frac{\tau_u}{2R_u}.
\]
The same bounds hold after the terminal output insertions.
\end{lemma}

\begin{proof}
For one frame of $K$ columns, with at most $m$ direction insertions,
set
\[
 b_*^2=\frac{8\log(16mK/\delta)}{d-m-K-2}.
\]
Let $\mathcal G$ contain the algorithm's private seed, previous coins, all inserted query vectors, all disclosed projection vectors, and every fully revealed column. Virtual responses are measurable functions of $\mathcal G$ and add no further conditioning. Initially the two frames have independent Haar laws. Inductively their conditional residual law is invariant under each pair of orthogonal transformations fixing $W_x$ and $W_u$, respectively. Projection disclosure and full-column revelation restrict this invariance to the enlarged stabilizers. Thus every unrevealed column has a rotationally invariant residual of norm at most one in a space of dimension at least $d-m-K$. The next projected query direction
is fixed under this conditioning. Spherical one-coordinate concentration
and a union bound imply that each insertion adds at most $b_*^2$ to
each unrevealed projection energy, except on an event of probability
at most $\delta/2$ for this frame.

If a fully revealed column has energy $e_j$ and another has energy
$e_i$, orthogonality bounds the increase in $e_i$ by
$e_ie_j/(1-e_j)$. Hence a common energy bound $z$ changes under a
revelation by at most $z\mapsto z/(1-z)$. Moving all $m$ additions
before all $K$ revelations only increases the bound, yielding
\[
 \norm{P_W v_i}^2\le
 \frac{mb_*^2}{1-Kmb_*^2}\le2mb_*^2,
\]
provided $Kmb_*^2\le1/2$. Apply this recurrence up to the first violation of its bound; the displayed condition keeps every denominator bounded away from zero, so no such violation occurs on the insertion event. Equation \eqref{hier:eq:dimension} ensures
this condition and $2mb_*^2\le\tau^2/(4R^2)$.

Apply this argument to $U$ and $V$, and union bound their failure
probabilities. Adaptivity mixing the two frames does not invalidate
the conditional symmetry: all scalar observations are already functions
of the disclosed projections, and the conditional law remains invariant
under separate orthogonal transformations fixing $W_x$ and $W_u$.
No conditional independence of the two residual frames is needed.
\end{proof}

\subsubsection{No uncharged vector information}

\begin{lemma}[Actual--virtual coupling]
\label[lemma]{hier:lem:coupling}
Under the stated locality and scalar-dependence conditions, on the good
event of \cref{hier:lem:projection}, the returned-gradient parts of the original and virtual transcripts
agree at every call; genie disclosures are additional observations. If fewer than $SM$ activations have occurred, the
last macro column remains unrevealed and
\[
 |U_S^\top b_x(\widehat x)|\le\tau_x/2.
\]
More generally, at a state $(r,k)$ all unexposed macro coordinates of
the query or output have absolute value at most $\tau_x/2$, and all
unexposed micro coordinates have absolute value at most $\tau_u/2$.
\end{lemma}

\begin{proof}
Induct over queries. Before the current activation, the query directions
have been inserted. For every currently unrevealed column,
\[
 |U_i^\top b_x(x)|\le R_x\norm{P_{W_x}U_i}\le\tau_x/2,
\]
and the analogous bound holds for $V_{i,j}^\top b_u(u)$. The same
argument applies to $b_u(\kappa w)$ when this is needed for the output.

Consider the pre-activation state $(r,k)$. For $i\le r$, every micro
column in block $i$ and every macro column through $U_i$ is already
known. For $i>r+1$, both the value and derivative of $\psi_i$ vanish,
so this block contributes neither an upper nor a lower vector gradient.
The only potentially incomplete eligible block is $i=r+1$.

In this current block, the micro carrier gradient uses only its first
$k$ columns. If $k<M$, its terminal coordinate is unrevealed and lies
in the flat region \eqref{hier:eq:gate-flat}; thus both terminal
micro-vector contributions vanish, and $Q=0$ removes the potentially
new macro vector $\nabla q_{r+1}$. The remaining macro gradients of
the selector and carrier lie in the known macro prefix. This proves
that the entire exact upper gradient uses only the pre-activation
known columns and public query directions.

When $\xi=0$, the inactive lower gradient also uses only those known
columns. Its compensation terms and normal-feature scalar products
do not enlarge vector support. When $\xi=1$, the active residual can
add at most $V_{r+1,k+1}$, which the stronger protocol has just revealed.
The gradients from all older blocks are already known, and all future
blocks still have zero selector jets at this fixed query. Completing a
block and freely revealing $U_{r+1}$ cannot trigger an additional new
block in the same response: the query was selected before this
revelation, and its projection on $U_{r+1}$ is still in the flat region.

Every scalar coefficient, including those in exact upper responses and
in the compensated lower potential, is computed from the disclosed
projections. The radial chain rule introduces only a known query-direction
multiple in addition to these frame-vector contributions. Therefore no
term discarded by the virtual truncation is nonzero on the good event.
The returned gradients agree, completing the induction.

At the terminal output, make the additional projection insertions after
fixing the output. The same projection calculation applies. If $a<SM$,
then $r<S$ by \eqref{hier:eq:prefix-state}, so $U_S$ is still unexposed.
\end{proof}

The argument permits repeated queries to earlier blocks and arbitrary
changes of query magnitude. It does not presume that the algorithm
follows a monotone geometric trajectory. It also avoids an informal
claim that exact scalar observations contain no information: all such
information is explicitly granted by the stronger projection observer.

\subsubsection{Activation counting and expected query budgets}

For completeness, state the additional analytic certificate needed to
turn the information argument into an optimization lower bound.

\paragraph{Prefix barrier certificate.}
For every state $(r,k)$ of \eqref{hier:eq:prefix-state} with $r<S$, and
every outer point whose unexposed macro and micro coordinates obey the
projection thresholds above, suppose that the exact hyperobjective
satisfies
\begin{equation}\label{hier:eq:barrier}
 \norm{\nabla F(x,w)}\ge8\eps.
\end{equation}
This may be implied by a stronger global end-game lemma, but it is not
a consequence of locality or projection concentration alone.

\begin{theorem}[Hierarchical activation lower bound]
\label[theorem]{hier:thm:conditional-lower}
Assume the displayed locality conditions, a globally admissible original
SFO of \eqref{hier:eq:sample-form}, the dimension conditions
\eqref{hier:eq:dimension}, and the prefix barrier certificate
\eqref{hier:eq:barrier}. If $N\le SM/(16p)$, every randomized adaptive
algorithm using at most $N$ original calls satisfies, under the random
frame distribution,
\[
 \PP\{\norm{\nabla F(\widehat x,\widehat w)}\ge8\eps\}
 \ge13/16.
\]
Consequently the fixed-horizon minimax complexity is
$\Omega(SM/p)$. The same order holds for worst-case expected query
complexity under a uniform guarantee
$\E\norm{\nabla F(\widehat x,\widehat w)}\le\eps$.
\end{theorem}

\begin{proof}
Use projection failure probability $1/8$. The number of activations is
bounded by a $\operatorname{Binomial}(N,p)$ variable, so
\[
 \PP\{a\ge SM\}\le\frac{Np}{SM}\le1/16.
\]
With probability at least $13/16$, the coupling event holds and the
prefix is incomplete. Equation \eqref{hier:eq:barrier} gives the
claimed probability bound. Averaging over frames and then selecting
one mean instance gives the usual randomized minimax conclusion;
oracle samples remain random on that fixed mean instance.

For the expected-cost statement, first suppose $SM/p\ge32$ and take
$n=\lfloor SM/(16p)\rfloor\ge SM/(32p)$. If a uniformly correct solver had worst-case expected cost
less than $n/8$, truncate it at $n$ calls. Truncation has probability at
most $1/8$, and Markov's inequality applied to its original output norm
gives probability at most $1/4$ of a norm larger than $4\eps$.
Thus its truncated version has norm at most $4\eps$ with probability
at least $5/8$, contradicting the fixed-horizon bound, which allows
probability at most $3/16$ for that event. Both comparisons can be made
after averaging over the same frame distribution. If $SM/p<32$,
apply the zero-query certificate to the algorithm stopped immediately
before its first call, returning the origin if a call would be made.
It agrees with the original output on every run making no calls.
Since $\Pr\{N_{\mathcal A}\ge1\}\le\mathbb E N_{\mathcal A}$
and the uniform gradient guarantee gives
$\Pr\{\|\nabla F(\widehat x,\widehat w)\|\ge8\epsilon\}\le1/8$,
the zero-query barrier forces
$13/16\le1/8+\mathbb E N_{\mathcal A}$ under the same hard prior.
Thus the worst-case expected cost is at least $11/16$, which is
$\Omega(SM/p)$ in this bounded range.
\end{proof}

\paragraph{Rate substitution.}
The analytic certificates in \cref{hier:thm:mean,thm:global}, whose compatibility with the information conditions is verified for the same functions in the next subsection, have the parameter scales
\[
 S\asymp\frac{L\Delta}{\eps^2},\qquad M\asymp\kappa^2,
 \qquad
 B_*\asymp\frac{\eps^2}{\vartheta\kappa^3L},\qquad
 p=\frac{B_*^2}{\sigma_g^2+B_*^2},
\]
With these scales the preceding theorem gives
\[
 \Omega\!\left(
 \kappa^2L\Delta\eps^{-2}
 +\vartheta^2\kappa^8\sigma_g^2L^3\Delta\eps^{-6}
 \right).
\]
The analytic certificates are \cref{hier:thm:mean,thm:global}; the
following interface verification checks that this information argument
applies to precisely those functions. The final choice of $S,b$ and
the remaining parameter ranges are given next.

\subsubsection{Interface with the one-sided receiver construction}

The explicit construction in \cref{thm:global,hier:thm:mean} has the required locality.  This
verification is useful because the three certificates must concern the
same oracle, including its exact scalar coefficients.

Take $\tau_x=b/4$ and $\tau_u=w_J/8$.  The common normal-width constant
is chosen so that $\tau_u\le\beta/4$.  The macro carrier is the
one-sided chain plus its public radial penalty, with its
\emph{uncentered} tail residual removed.  Removing a constant from a
residual does not change the robust support proof.  The added
negative-side penalty is flat on every unseen coordinate, so the
macro carrier satisfies condition 1.  The explicit predecessor gates
and tail masks give condition 2 exactly as checked above.

The selected micro carrier is the one-sided micro chain minus its
\emph{centered} tail residual, plus the receiver $D$.  Its negative-side
penalty and receiver are both flat on an unseen tail:
\[
 |t|\le\tau_u\quad\Longrightarrow\quad
 p'(t/\beta)=0,\qquad D(t)=D'(t)=0.
\]
The receiver profile has $D'=0$ throughout
$(-\infty,\beta/4]$, and $D(0)=0$; its wide release gate obeys
$Q(t)=0$ throughout $(-\infty,b]$.  Thus all terminal conditions in
\eqref{hier:eq:gate-flat} hold.  In particular there is no hidden
constant value of $Q$ that could multiply a new macro vector.

For the inactive lower potential, every coarser flattening resolution
satisfies $\delta_j\ge w_J>\tau_u$.  If $J\ge1$, its last atom is
\[
 E_J-(R-R_{\tau_u})
   =R_{\tau_u}-R_{\delta_{J-1}},
\]
which has micro-gradient support in the revealed prefix, as do all
earlier atoms.  For $J=0$, the inactive atom is just $R_{\tau_u}$.
The compensation uses the same flattened residual.  Multiplication by
macro selectors and squaring the summed residual introduce only scalar
coefficients and selector gradients already supported on the macro
prefix.  This verifies condition 3 for the complete inactive potential,
including all its outer-gradient blocks.  The unflattened residual
adds at most the next micro column.

Choose all micro frames as disjoint column blocks of the single Haar
frame $V$ used here.  The analytic bounds and the global margin allow
this special choice, since they do not require overlapping micro
spans.  The global-margin certificate states
\[
 U_S^\top B_{R_m}(x)<b
 \quad\Longrightarrow\quad
 \norm{\nabla F(x,v)}\ge c_*Lb.
\]
With $b\ge8\eps/(c_*L)$, the terminal projection bound
$|U_S^\top B_{R_m}(\widehat x)|\le b/8$ therefore supplies
\eqref{hier:eq:barrier}; no micro-coordinate restriction is additionally
needed for this implication.  Consequently the hierarchy theorem
applies once the stated primitive and gap certificates are invoked.
Here $M=\lfloor\zeta^6\kappa^2\rfloor$ counts the actual core
coordinates: the receiver acts on coordinate $M$ and introduces no
additional hidden columns.
\subsection{Bounded-gap and bounded-condition parameter ranges}
\label{sec:hierarchy-fallbacks}

The hierarchy is needed only when both $\Delta/L$ and $\kappa$ exceed
fixed numerical thresholds. This subsection supplies the other cases and the complete amplitude-aware multiscale lemma used earlier.

\subsubsection{A single chain at a separate curvature scale}

Fix $0<\ell\le L$, set $\alpha=\ell/L$, $\mu=L/\kappa$, and
$1\le\vartheta\le\kappa$. Use the fixed profiles
\[
 \Psi(t)=0\ (t\le1/2),\quad
 \Psi(t)=\exp\{1-(2t-1)^{-2}\}\ (t>1/2),\qquad
 \Phi(t)=\sqrt e\int_{-\infty}^t e^{-r^2/2}\,dr.
\]
Let $p_-\ge0$ be smooth, zero on $[-1/2,\infty)$, with
$-1\le p_-'\le0$ and $p_-'=-1$ on $(-\infty,-1]$. For width
$\beta>0$ and length $T\ge1$, define
\[
 H(z)=\eta\ell\beta^2
 \left[-\Phi(z_1/\beta)
 -\sum_{i=2}^T\Psi(z_{i-1}/\beta)\Phi(z_i/\beta)
 +\sum_{i=1}^Tp_-(z_i/\beta)\right],
\]
where $\eta>0$ is fixed numerically. Then
\[
 H(0)-\inf H\le CT\ell\beta^2,\quad
 \|\nabla H\|_\infty\le C\ell\beta\sqrt T,\quad
 \|\nabla^2H\|_\infty\le C\ell.
\]
All coordinate derivatives are nonpositive. If $z_T<\beta$, the first
coordinate $j$ with $z_j<\beta$ has $\partial_jH(z)\le-\eta\ell\beta$:
the incoming term supplies the bound on $[-\beta,\beta)$ and the
negative-side penalty supplies it below $-\beta$.

Choose flat smooth cutoffs $c_0,e_0$ with $c_0=0$ on $[-1/4,1/4]$,
$c_0=1$ outside $(-1/2,1/2)$, $e_0=1$ on $(-\infty,1/4]$, and
$e_0=0$ on $[1,\infty)$. Put
\[
 h_i(z)=e_0\!\left(\sum_{j=i}^Tc_0(z_j/\beta)^2\right),\quad
 \psi_1=h_1,\quad \psi_i=\Psi(z_{i-1}/\beta)h_i\ (i\ge2),
\]
\[
 R(z)=-\eta\ell\beta^2\sum_{i=1}^T\psi_i(z)\Phi(z_i/\beta)
                   +\eta\ell\beta^2\Phi(0),\qquad A=H-R.
\]
Then $R(0)=0$, $\nabla A(0)=0$, and
\begin{equation}\label{fb:eq:residual}
 \|D^jR\|_\infty\le C_j\ell\beta^{2-j},\qquad 0\le j\le3.
\end{equation}
These are Euclidean multilinear operator norms. Take the last index
with coordinate magnitude at least $\beta/2$. Earlier tail masks and
later predecessor gates are flat zero, leaving at most one summand jet.
On the derivative support of a tail mask its inner sum is at most one,
so its derivative norms are $O(\beta^{-j})$ without dimension factors.
For example the squared norm of the first derivative of the inner sum
is at most $C\beta^{-2}\sum c_0^2\le C\beta^{-2}$; higher derivatives
are diagonal forms followed by the chain rule.
If $|z_i|\le\beta/4$ for every $i>r$, cancellation and penalty flatness
give
\begin{equation}\label{fb:eq:locality}
 \operatorname{supp}\nabla A(z)\subseteq[r],\qquad
 \operatorname{supp}\nabla R(z)\subseteq[r+1].
\end{equation}

For an orthonormal frame $U\in\mathbb R^{d\times T}$, set
\[
 {\cal R}=K\beta\sqrt T,\quad
 B(u)=\frac{u}{\sqrt{1+\|u\|^2/{\cal R}^2}},\quad
 P(u)=\nu{\cal R}^2
       \bigl(\sqrt{1+\|u\|^2/{\cal R}^2}-1\bigr),\quad
 \nu=\eta_0\eta\ell,
\]
\[
 {\cal H}(u)=H(U^\top B(u))+P(u),\quad
 r(u)=R(U^\top B(u)),\quad
 {\cal A}(u)=A(U^\top B(u))+P(u).
\]
Here $\eta_0$ is fixed small and $K$ is fixed large afterwards.
The wrapper has derivative bounds $1,C/{\cal R},C/{\cal R}^2$ through
order three. Consequently
\begin{equation}\label{fb:eq:wrapper-reg}
 \operatorname{Lip}(\nabla{\cal A}),
 \operatorname{Lip}(\nabla{\cal H})\le C\ell,\qquad
 \sup\|\nabla{\cal A}\|\le C\ell\beta\sqrt T.
\end{equation}
Equation \eqref{fb:eq:residual} also holds for $r$. Moreover
\begin{equation}\label{fb:eq:gap-margin}
 {\cal H}(0)-\inf{\cal H}\le CT\ell\beta^2,\qquad
 (U^\top B(u))_T<\beta\ \Longrightarrow\
        \|\nabla{\cal H}(u)\|\ge c_0\ell\beta.
\end{equation}
Here is a direct global check. Write $z=U^\top B(u)$ and
$q=(1+\|u\|^2/{\cal R}^2)^{-1/2}$. The component in column $U_j$ is
\[
 q\,\partial_jH(z)+
 \left(\nu-\frac{q}{{\cal R}^2}\langle z,\nabla H(z)\rangle\right)z_j.
\]
For $\|u\|\le{\cal R}$ the scalar in parentheses lies between $\nu/2$
and $2\nu$, after increasing $K$. The positive part of the inner product
comes only from negative coordinates and is bounded by
$C\ell\beta^2T+C\ell\beta\sqrt T\|z\|$; the other bound follows from
$\|\nabla H\|\le C\ell\beta\sqrt T$. For the first $j$ with $z_j<\beta$,
the radial correction helps if $z_j<0$, and is at most $2\nu\beta$ if
$0\le z_j<\beta$. Choose $\eta_0$ small. For $\|u\|>{\cal R}$, the
radial derivative of $P$ is at least $\nu{\cal R}/\sqrt2$, whereas
the chain contribution is at most $C\ell\beta\sqrt T$; increase $K$.
Coercivity and the gap bound follow from $P\ge0$, $P(0)=0$, and the
nonnegative coordinate penalty.

\subsubsection{The amplitude-aware multiscale lemma}

For $0<\delta\le\beta/16$, let $c_\delta(t)=t\omega(t/\delta)$, where
$\omega$ is even, flat zero on $[-1,1]$, flat one outside $(-2,2)$,
and takes values in $[0,1]$. Define
$R_\delta(z)=R(c_\delta(z_1),\ldots,c_\delta(z_T))$.
Only coordinates with magnitude at most $2\delta\le\beta/8$ change.
Predecessor gates and tail masks are unchanged, and only the outgoing
$\Phi$ factor in the unique active summand can change. Thus
\[
 \|R_\delta\|_\infty\le C\ell\beta^2,\quad
 \|DR_\delta\|_\infty\le C\ell\beta,\quad
 \|D^2R_\delta\|_\infty\le C\ell\beta/\delta,\quad
 \|D^3R_\delta\|_\infty\le C\ell\beta/\delta^2,
\]
\[
 \|R-R_\delta\|_\infty\le C\ell\beta\delta.
\]
If the tail after $r$ is within $\delta$, then
$\operatorname{supp}\nabla R_\delta\subseteq[r]$.
For $\delta'<\delta$, all jets of $R_{\delta'}-R_\delta$ vanish
unless the active child belongs to $\delta'<|z_j|<2\delta$.

Set $v=\vartheta/\alpha$, and define
\[
 J=\left\lfloor\log_4(\max\{1,\log v\})\right\rfloor,\quad
 a_j=1-4^{-(j+1)},\quad b_j=1-\tfrac12 4^{-j},
\]
\[
 \delta_j=c_{\rm ms}\beta v^{-a_j}\ (0\le j<J),\qquad
 w_j=c_{\rm ms}\beta v^{-b_j}\ (0\le j\le J).
\]
For $J=0$ set $E_0=R$. Otherwise set
\[
 E_0=R_{\delta_0},\quad
 E_j=R_{\delta_j}-R_{\delta_{j-1}}\ (1\le j<J),\quad
 E_J=R-R_{\delta_{J-1}}.
\]
Then $\sum E_j=R$ and $w_J\asymp\beta/v$, since
$v^{1-b_J}=\exp((\log v)/(2\cdot4^J))$ is between numerical constants.
The assertion also holds for $J=0$. Successive difference scales have
ratios at least $e^3$, so at most two adjacent annuli overlap. Including
the coarse component, at most three component jets are active.
Put $e_j(u)=E_j(U^\top B(u))$.

For a fixed odd compactly supported smooth $n$ equal to the identity
on $[-1/4,1/4]$, write $n_w(s)=wn(s/w)$ and set
\begin{equation}\label{fb:eq:K}
 {\cal K}(u,s)=\frac{\mu}{2L^2}r(u)^2
             -\frac\mu L\sum_{j=0}^J e_j(u)n_{w_j}(s).
\end{equation}

\begin{lemma}[Amplitude-aware lower-lift budgets]\label[lemma]{fb:lem:multiscale}
For $0<\alpha\le1$, $\vartheta\ge1$, and $\beta\vartheta\le c$,
\[
 \|D^2{\cal K}\|_\infty
 \le C\mu(\beta\sqrt{\alpha\vartheta}+\alpha^2\beta^2),\qquad
 \|D^3{\cal K}\|_\infty
 \le C\mu(\vartheta+\alpha^2\beta)\le C'\mu\vartheta.
\]
Every normal feature is linear at $s=r(u)/L$. Neither $v\le\kappa$
nor $\beta v\le c$ is required.
\end{lemma}
\begin{proof}
Write $A_j,D_{1,j},D_{2,j},D_{3,j}$ for value and derivative bounds
of a component. The coarse value bound is $A_0\le C\ell\beta^2$;
intermediate and final value bounds are
$A_j\le C\ell\beta\delta_{j-1}$. Every first derivative is at most
$C\ell\beta$. Higher derivatives are bounded by
$C\ell v^{a_j}$ and $C\ell\beta^{-1}v^{2a_j}$, using $a_{J-1}$ for
the final component. Wrapper errors are dominated since
$ {\cal R}\ge K\beta$. The four third-order blocks of a normal
coupling are bounded by
\[
 \frac\mu L(w_jD_{3,j},D_{2,j},D_{1,j}/w_j,A_j/w_j^2).
\]
Use $2a_j-b_j=1$, $2b_j-a_{j-1}=1$, and $2b_0=1$. Every entry is
at most $C(\mu\ell/L)v=C\mu\vartheta$. For the final piece use
$2a_{J-1}-b_J\le1$ and $2b_J-a_{J-1}=1$. The case $J=0$ follows
directly from the residual bounds, since $1\le v<e^4$ there.
The second-order blocks are bounded by
$(\mu/L)(w_jD_{2,j},D_{1,j},A_j/w_j)$.
Since $a_j-b_j\le1/4$ and $b_j-a_{j-1}\le1/2$, all are at most
$C\mu\alpha\beta\sqrt v=C\mu\beta\sqrt{\alpha\vartheta}$.
Bounded overlap removes any $J$ factor. The compensation product rule
gives $C\mu\alpha^2\beta^2$ and $C\mu\alpha^2\beta$.
Finally
\[
 \frac{\|r\|_\infty/L}{w_J}
 \le C\frac{\alpha\beta^2}{\beta/v}=C\beta\vartheta.
\]
Take this at most $1/4$ and use $w_j\ge w_J$.
\end{proof}

\subsubsection{Two carrier placements and the same information envelope}

Use the lower mean
\[
 g(x,u,s)=\frac\mu2\|u\|^2-L\langle x,u\rangle
                  +\frac\mu2s^2+{\cal K}(u,s).
\]
Its full gradient is $CL$-Lipschitz, its lower Hessian is at least
$(\mu/2)I$, and its full Hessian is $C\mu\vartheta$-Lipschitz under
the lemma's smallness condition. In the linear core,
\[
 g_s=\mu s-\mu r/L,\qquad
 g_u=\mu u-Lx+\mu r\nabla r/L^2-\mu s\nabla r/L.
\]
Thus the unique lower response is $u^*(x)=\kappa x$,
$s^*(x)=r(\kappa x)/L$. Two upper means are useful:
\[
 f_{\rm in}(x,u,s)={\cal A}(u)+Ls,\qquad
 f_{\rm out}(x,u,s)={\cal A}(\kappa x)+Ls.
\]
Both induce exactly $F(x)={\cal H}(\kappa x)$ and have exact lower
initial solution $(u,s)=(0,0)$ at $x=0$. Their budgets differ:
\begin{align*}
 \operatorname{Lip}(\nabla f_{\rm in})&\le C\ell,
 &\|(f_{\rm in})_y\|_\infty&\le L+C\ell\beta\sqrt T,\\
 \operatorname{Lip}(\nabla f_{\rm out})&\le C\kappa^2\ell,
 &\|(f_{\rm out})_y\|_\infty&=L.
\end{align*}
In either case $\operatorname{Lip}(\nabla F)\le C\kappa^2\ell$.

Set $\tau=w_J/8$, let $r_\tau(u)=R_\tau(U^\top B(u))$, and put
$d=r-r_\tau$. Define $ {\cal K}_0$ by replacing $r^2$ with $r_\tau^2$
and replacing the last atom $e_J$ with $e_J-d$. Then
\[
 {\cal D}:={\cal K}-{\cal K}_0
       =\frac\mu{2L^2}(r^2-r_\tau^2)
                  -\frac\mu Ld\,n_{w_J}(s).
\]
Its derivatives are
\[
 \nabla_u{\cal D}
   =\frac\mu{L^2}(d\nabla r+r_\tau\nabla d)
                     -\frac\mu Ln_{w_J}\nabla d,\qquad
 \partial_s{\cal D}=-\frac\mu Ld n_{w_J}'.
\]
The flattening bounds and $\tau\asymp\alpha\beta/\vartheta$ give
\[
 \|d\|_\infty\le C\ell\beta\tau,\quad
 \|\nabla d\|_\infty\le C\ell\beta,\quad
 \|r_\tau\|_\infty\le C\ell\beta^2.
\]
Consequently, using $\beta\vartheta\le c$,
\begin{equation}\label{fb:eq:signal}
 \|\nabla{\cal D}\|_\infty
       \le C_B\frac{\mu\alpha^2\beta^2}{\vartheta}=:B.
\end{equation}
The largest compensation contribution is
$O(\mu\alpha^2\beta^3)$; it is bounded by the envelope because
$\beta\vartheta$ is small. With $p=B^2/(\sigma_g^2+B^2)$ and
$\xi\sim\operatorname{Ber}(p)$, return the exact upper gradient and
the gradient of the scalar sample potential
\[
 g_\xi=g_{\rm quad}+{\cal K}_0+(\xi/p){\cal D}.
\]
This is globally unbiased for $g$, with full-gradient variance at most
$\sigma_g^2$. Both carrier placements thus admit the required
conservative original SFO.

\subsubsection{Single-chain information lower bound}

Draw $U$ as a Haar frame. Give a stronger observer every column's
projection onto the span of queried outer and lower directions and all
fully revealed columns. An inactive query uses only flattened residual
gradients and the exact carrier. An active query can reveal at most
the next full column. Insert the output direction after fixing the
output. An $N$-query horizon has at most $M_q=2N+1$ direction insertions.

With failure probability $\delta$, every unseen column has projection
norm at most $\tau/(2{\cal R})$ whenever
\begin{equation}\label{fb:eq:dimension}
 d\ge C(M_q+T)v^2T\log(4M_qT/\delta).
\end{equation}
Indeed, conditioned on disclosed projections, the residual frame is
rotationally invariant. A new transcript-measurable direction adds at
most $b_*^2=8\log(4M_qT/\delta)/(d-M_q-T-2)$ to each projection energy,
outside a union-bounded failure event. Revealing a full column changes
a common energy bound $z$ by at most $z/(1-z)$, by orthogonality.
Moving additions before revelations bounds the energy by
$M_qb_*^2/(1-TM_qb_*^2)$. Since $ {\cal R}/\tau\le Cv\sqrt T$,
\eqref{fb:eq:dimension} gives the threshold.

Define a virtual transcript for every frame and coin sequence. Use
the disclosed projections for every scalar coefficient. On inactive
calls truncate every hidden vector contribution to the current prefix;
on active calls reveal the next column and truncate to the enlarged
prefix. This is measurable from the genie data even off the good event.
On the good event all unknown effective coordinates have magnitude at
most $\tau/2$. Equation \eqref{fb:eq:locality} and the flattened support
statement show inductively that virtual and actual transcripts agree.
This applies to $f_{\rm in}$ at the queried $u$ and to $f_{\rm out}$
at $\kappa x$, since both directions were inserted. Wrapper derivatives
add only public query-direction components with known scalar coefficients.

Choose
\begin{equation}\label{fb:eq:parameter-choice}
 \beta=C_\beta\frac{\epsilon}{\kappa\ell},\qquad
 T=\left\lfloor c_T\frac{\Delta}{\ell\beta^2}\right\rfloor.
\end{equation}
Fix $C_\beta$ large enough that the margin in \eqref{fb:eq:gap-margin},
after multiplication by $\kappa$, is at least $8\epsilon$; then take
$c_T$ small enough for the gap. Provided the floor argument is at least
a fixed large number and $\beta\vartheta\le c$,
\[
 T\asymp\frac{\kappa^2\ell\Delta}{\epsilon^2},\qquad
 B^2\asymp\frac{\epsilon^4}{\vartheta^2\kappa^6L^2}.
\]
Here $B$ is the chosen numerical envelope. The second identity uses
$\mu=L/\kappa$ and $\alpha=\ell/L$; the chain scale $\ell$ cancels.

For $N\le T/(16p)$, the chance of $T$ activations is at most
$Np/T\le1/16$. Take projection failure probability $1/8$.
With probability at least $13/16$, the output's last column remains
unrevealed and its gradient norm is at least $8\epsilon$.
Averaging over frames and selecting an instance gives the fixed-horizon
minimax lower bound $\Omega(T/p)$.
For worst-case expected cost, truncate a proposed solver at
$n=\lfloor T/(16p)\rfloor$ calls. If its uniform expected cost were
less than $n/8$, truncation would have probability at most $1/8$.
Its expected-gradient guarantee gives probability at least $3/4$ of
a gradient at most $4\epsilon$. The truncated solver would succeed
at this threshold with probability at least $5/8$, contradicting the
fixed-horizon probability bound of at most $3/16$.
The same order therefore holds for worst-case expected calls.

\subsubsection{Bounded gap and all condition numbers}

Suppose $0<\Delta\le L$ and $\kappa\ge2$. Choose the inner carrier,
$\ell=L$, $\alpha=1$, and \eqref{fb:eq:parameter-choice}.
The proposed common accuracy range
\[
 0<\epsilon\le c\min\{L,\kappa\sqrt{L\Delta}\}
\]
ensures $\beta\vartheta\le c'$ because $\vartheta\le\kappa$, and a
sufficiently large floor argument. Furthermore
\[
 \|(f_{\rm in})_y\|_\infty
       \le L+C\sqrt{L\Delta}\le C'L.
\]
This is a normalized admissible family with
\[
 \frac Tp\asymp
 \kappa^2L\Delta\epsilon^{-2}
 +\vartheta^2\kappa^8\sigma_g^2L^3\Delta\epsilon^{-6}.
\]
Its hyperobjective curvature is at most $CL\kappa^2$.

\subsubsection{Bounded condition number and arbitrary larger gap}

Let $\kappa_0\ge2$ be the fixed numerical threshold in the hierarchical
global-margin construction. Suppose $\Delta>L$ and
$2\le\kappa\le\kappa_0$. Choose the outer carrier,
\[
 \ell=L/\kappa^2,\quad
 \alpha=\kappa^{-2},\quad
 v=\vartheta\kappa^2,\quad
 \beta=C_\beta\kappa\epsilon/L.
\]
Then $\|(f_{\rm out})_y\|_\infty=L$ and
$\operatorname{Lip}(\nabla f_{\rm out})\le CL$.
Since $\Delta>L$, the proposed accuracy range reduces to
$\epsilon\le cL$. Choose the same universal numerical $c$ small enough
that $C_\beta\kappa_0^2c$ is below the linear-core threshold.
Then $\beta\vartheta\le c'$ uniformly in this branch.
The floor argument is large because
$L\Delta/\epsilon^2>c^{-2}$.

Now
\[
 T\asymp L\Delta/\epsilon^2
       \ge c_1\kappa_0^{-2}\kappa^2L\Delta/\epsilon^2,\qquad
 B^2\asymp\epsilon^4/(\vartheta^2\kappa^6L^2).
\]
Consequently
\[
 \frac Tp\gtrsim\kappa_0^{-2}
 \left[
 \kappa^2L\Delta\epsilon^{-2}
 +\vartheta^2\kappa^8\sigma_g^2L^3\Delta\epsilon^{-6}
 \right].
\]
The factor $\kappa_0^{-2}$ is a fixed numerical constant, not a
parameter-dependent loss. This family even has
$\operatorname{Lip}(\nabla F)\le CL$.

\paragraph{Assembly of parameter cases.}
The fallbacks cover $\Delta\le L$ for every $\kappa$, and
$\Delta>L$ for $\kappa\le\kappa_0$. Only $\Delta>L$,
$\kappa\ge\kappa_0$ needs the hierarchy. Take the minimum of the finitely
many numerical accuracy constants. Thus the same range
$0<\epsilon\le c\min\{L,\kappa\sqrt{L\Delta}\}$ works throughout.
All families have fixed normalized primitive-budget ratios, exact
initial lower solutions, and
$\operatorname{Lip}(\nabla F)\le CL\kappa^2$.
\subsubsection{Exact prescribed primitive budgets}
\label{sec:exact-budget-normalization}

The fixed numerical slack in the preceding certificates can be removed
without imposing any restriction on $\Delta/L$. We give the coefficient
changes explicitly. Here the prescribed budgets are lower strong
convexity $\mu=L/\kappa$, $C_f,L_f,L_g\le L$, and Hessian-block variation
at most $\rho=\mu\vartheta$, with $\kappa\ge2$ and
$1\le\vartheta\le\kappa$.

Keep all fixed profiles and the hierarchy constants already chosen.
In the lower quadratic and the upper linear term use
\[
 m=\frac32\mu,\qquad c_{xy}=\frac14L,\qquad d=\frac12L,
 \qquad K_{\rm amp}=\frac{c_{xy}}m=\frac\kappa6.
\]
Thus the quadratic lower mean is
$m\|u\|^2/2-c_{xy}\langle v,u\rangle+ms^2/2$, the upper linear term is
$ds$, and the exact affine response amplifier will be $K_{\rm amp}$.
Take chain curvature $\ell=\eta L$, where the fixed numerical
$\eta>0$ will be chosen sufficiently small. For the hierarchy, use
$M\asymp K_{\rm amp}^2$, macro width $b$, and micro width
$\beta=b/K_{\rm amp}$; use curvature $\ell$, rather than $L$, in
both chains and their radial penalties.

Set
\[
 \alpha=\frac{\ell}{d},\qquad
 {v_{\mathrm{ms}}}=\frac{\rho}{C_0m\alpha}
            =\frac{\rho d}{C_0m\ell}.
\]
Choose a fixed numerical $C_0$ above the constants in the derivative
certificate, and then choose
$\eta\le\min\{1/2,1/(3C_0)\}$. Thus $\alpha\le1$ and
\[
 {v_{\mathrm{ms}}}=\frac{\vartheta}{3C_0\eta}\ge1.
\]
Use the same multiscale construction with parameter ${v_{\mathrm{ms}}}$.
No upper bound on ${v_{\mathrm{ms}}}$ relative to $K_{\rm amp}$ is required.
The actual compensation is now
\[
 \mathcal K(x,u,s)=\frac{m}{2d^2}r(x,u)^2
             -\frac md\sum_j e_j(x,u)n_{w_j}(s),
 \qquad w_J\asymp\frac{\beta}{{v_{\mathrm{ms}}}}.
\]
The centered residual has joint derivative bounds
$\|D^jr\|\le C\ell\beta^{2-j}$, including macro derivatives in the
hierarchical branch. This is because $b\ge\beta$ in that branch.

Write
\[
 \delta_{\rm core}=\alpha\beta{v_{\mathrm{ms}}}
                  =\frac{\rho\beta}{C_0m}.
\]
The generalized multiscale calculation gives
\[
 \|D^2\mathcal K\|
 \le Cm(\alpha\beta\sqrt{{v_{\mathrm{ms}}}}+\alpha^2\beta^2)
 \le Cm(\delta_{\rm core}+\delta_{\rm core}^2),
\]
\[
 \|D^3\mathcal K\|
 \le Cm(\alpha{v_{\mathrm{ms}}}+\alpha^2\beta)
 =\frac{C\rho}{C_0}
       \left(1+\frac{\delta_{\rm core}}{{v_{\mathrm{ms}}}^2}\right).
\]
Choose $C_0$ large enough and then the numerical upper bound on
$\delta_{\rm core}$ small enough that
\[
 \|D^2\mathcal K\|\le\mu/8,\qquad
 \|D^3\mathcal K\|\le\rho.
\]
The same condition ensures the exact linear core, since
$(\|r\|_\infty/d)/w_J\le C\delta_{\rm core}$.
Consequently the lower Hessian is at least
$mI-(\mu/8)I=(11\mu/8)I\succeq\mu I$.

There is also enough room in the exact gradient-smoothness budget.
The largest absolute eigenvalue of the quadratic $(v,u)$ Hessian is
\[
 \frac{m+\sqrt{m^2+4c_{xy}^2}}2
 \le \frac{3+\sqrt{13}}8 L
 \qquad(\kappa\ge2).
\]
The scalar $s$ block has norm $m\le3L/4$. Adding the perturbation gives
\[
 \operatorname{Lip}(\nabla g)
 \le\left(\frac{3+\sqrt{13}}8+\frac1{16}\right)L<L.
\]
The third-derivative estimate proves the prescribed bound $\rho$ on
both lower Hessian blocks.

The upper-primitive Hessian estimate is $C_*\ell$ in the hierarchy and
the inner-carrier fallback, where $C_*$ depends only on the already
fixed profiles and hierarchy constants. Choose $\eta\le1/C_*$.
In the hierarchy the nonlinear lower-gradient block has norm at most
$C\ell b=O(\epsilon)$ after setting
$b=C_\epsilon\epsilon/\ell$. Since $f_s=d=L/2$, decreasing the universal
accuracy constant ensures $\|f_y\|\le L$. In the bounded-gap inner
fallback it is enough to use
\[
 \|f_u\|\le C\sqrt{\ell\Delta}\le C\sqrt\eta\,L
 \quad(\Delta\le L)
\]
and choose $\eta$ small enough that this is at most $L/2$.
Thus $C_f,L_f\le L$ as well.

The exact response equations are
\[
 u^*=K_{\rm amp}v,\qquad s^*=r(x,K_{\rm amp}v)/d.
\]
The term $ds^*=r$ still cancels precisely the selected centered
residuals. Thus all previous hyperobjective identities and support
arguments hold after replacing the amplifier by $K_{\rm amp}$.

It remains to check that these adjustments preserve the rate. Choose
$\beta\asymp\epsilon/(K_{\rm amp}\ell)$ and, in the hierarchy,
$b=K_{\rm amp}\beta$. Then
\[
 \delta_{\rm core}
 \asymp\frac{\vartheta\epsilon}
               {K_{\rm amp}\ell}
 \le C\frac{\epsilon}{\eta L},
\]
because $\vartheta\le\kappa$ and $K_{\rm amp}=\kappa/6$.
The original high-accuracy restriction $\epsilon\le cL$, with a smaller
fixed numerical $c$, therefore guarantees the core and perturbation
requirements, and also makes $b$ sufficiently small.
The stage count is
\[
 T_{\rm total}\asymp
 \frac{K_{\rm amp}^2\ell\Delta}{\epsilon^2}
 \asymp\frac{\kappa^2L\Delta}{\epsilon^2}.
\]
Here and below the constants absorb the fixed $\eta$ and the fixed
fraction of $K_{\rm amp}^2$ used for the micro length.
The conservative signal envelope becomes
\[
 B\asymp\frac{m\alpha\beta^2}{{v_{\mathrm{ms}}}}
 \asymp\frac{m^2\alpha^2\beta^2}{\rho}
 \asymp\frac{m^4\epsilon^2}{d^2c_{xy}^2\rho}
 \asymp\frac{\epsilon^2}{\vartheta\kappa^3L}.
\]
The $\ell$ factors cancel. Thus
$T_{\rm total}(1+\sigma_g^2/B^2)$ has exactly the displayed original
deterministic and stochastic parameter exponents. Gap normalization
uses the same fixed-fraction choices of the chain lengths and never
requires $\Delta/L$ to be bounded.

For completeness, increase the fixed condition-number threshold so
that the hierarchy is used only where $K_{\rm amp}$ exceeds its
required numerical threshold. The remaining physical $\kappa$ lie in
a fixed bounded interval. For their outer-carrier fallback use
\[
 \ell_{\rm out}=\frac{\eta L}{1+K_{\rm amp}^2},
\]
rather than $\eta L/K_{\rm amp}^2$ when $K_{\rm amp}<1$.
Then $C K_{\rm amp}^2\ell_{\rm out}\le C\eta L$, while $f_y=(0,d)$.
The same generalized construction uses
${v_{\mathrm{ms}}}=\rho d/(C_0m\ell_{\rm out})\ge1$, and a sufficiently
small universal accuracy constant enforces the core condition on the
entire bounded interval. Its stage count is
\[
 T\asymp
 \frac{\eta K_{\rm amp}^2}{1+K_{\rm amp}^2}
       \frac{L\Delta}{\epsilon^2}
 \ge c_*\frac{\kappa^2L\Delta}{\epsilon^2},
\]
where $c_*>0$ is numerical because the interval is fixed.
Its signal has the same order just computed. The bounded-gap inner
fallback works for every $\kappa\ge2$, including $K_{\rm amp}<1$;
its proof needs only a positive amplifier bounded below by $1/3$.
Finally choose the common accuracy constant small enough for all
branches and for the chain-length floors. This preserves the full range
$0<\epsilon\le c\min\{L,\kappa\sqrt{L\Delta}\}$.

This explicit reparametrization proves membership in the exactly
prescribed primitive class. All auxiliary constants are numerical;
none depends on $\Delta/L$.

\subsection{Completion of the all-gap theorem}
\label{sec:completion}

\begin{proof}[Proof of \cref{thm:main}]
First suppose $\Delta\ge L$ and $\kappa\ge\kappa_0$.
Choose the numerical multiple $C_b$ so that
$b=C_b\epsilon/L$ and the constant in \cref{thm:global} gives
$c_*Lb\ge8\epsilon$. Fix all profile and hierarchy constants as in
that theorem, and let
\[
 S=\left\lfloor c_S\frac{\Delta}{Lb^2}\right\rfloor,
 \qquad M=\lfloor\zeta^6\kappa^2\rfloor.
\]
Choose $c_S$ sufficiently small for the gap estimate in
\cref{hier:thm:mean}, and the universal accuracy constant sufficiently
small so that $S\ge1$, $b\le c$, and the floor quantities are within
fixed factors of their arguments. Since $\vartheta\le\kappa$,
$\beta\vartheta=b\vartheta/\kappa\le b$, so the multiscale core
condition follows without a new parameter restriction.

The primitive budgets, exact response, initial solution, and
$O(L\kappa^2)$ hyperobjective curvature follow from
\cref{hier:thm:mean}. The conservative oracle has
\[
 B=C_B\frac{\mu\beta^2}{\vartheta}
   \asymp\frac{\epsilon^2}{\vartheta\kappa^3L},
 \qquad p=\frac{B^2}{\sigma_g^2+B^2}.
\]
Select finite dimensions satisfying the projection conditions. On the
coupled event an incomplete macro prefix has $Z_S(\widehat x)<b$,
so \cref{thm:global} supplies the required $8\epsilon$ barrier.
The activation theorem therefore gives $\Omega(SM/p)$ calls. Since
\[
 SM\asymp\frac{\kappa^2L\Delta}{\epsilon^2},
 \qquad
 \frac{SM}{p}=SM\left(1+\frac{\sigma_g^2}{B^2}\right),
\]
this is \eqref{eq:main-rate}.

For $\Delta\le L$, use the bounded-gap inner-carrier construction in
the preceding parameter-range section. It gives the same rate and
the stated bounded-gap accuracy condition. For $\Delta>L$ and
$2\le\kappa<\kappa_0$, use the outer-carrier construction in that
section. Its $\kappa^6$ stochastic coefficient and unamplified
deterministic coefficient differ from \eqref{eq:main-rate} by at most
the fixed numerical factor $\kappa_0^2$. Its stronger-looking accuracy
condition is also implied by $\epsilon\le cL$ after choosing the one
universal $c$ sufficiently small. Taking the minimum of the numerical
constants in the three cases proves the theorem for every parameter
tuple in \eqref{eq:main-accuracy} under the convenient numerical slack.
The explicit coefficient changes in
Section~\ref{sec:exact-budget-normalization} give the exactly prescribed
budgets in the statement, with the same rate and the same accuracy
range after one further numerical reduction of $c$.
\end{proof}

\section{Variation-sensitive upper bounds with upper-gradient reuse}
\label{app:upper}
\label{app:accelerated-cached-upper}

We construct an upper algorithm from fresh first-order samples.
Its lower solves use acceleration, and it reuses the entire upper-gradient
estimate until the outer displacement exceeds a prescribed radius.
The proof separates the number of outer steps from the number of
upper-gradient refreshes and bounds the total calls on every history.
Acceleration of strongly convex lower solves has prior precedent in
\citet[Appendix~F]{CJZ}; the finite-variance framework of
\citet{Aybat2019} is also related. We give the required conditional-noise
argument explicitly.

Throughout this appendix put
\[
 k=L_g/\mu\ge1,\qquad L_a=L_f(1+k)^2,
 \qquad \|y_{\rm init}-y^*(x_0)\|\le R_0.
\]
Every gradient sample is fresh at its one queried point.
Returned vectors may be stored and reused in arithmetic; the algorithm
never evaluates a previous random sample at a different point.

\subsection{Response geometry and a frozen linear tilt}

\begin{lemma}[Response, tilt, and curvature bounds]
\label[lemma]{up:lem:estimator-geometry}
Under \cref{ass:primitive}, the response $y^*$ is $k$-Lipschitz.
At a fixed outer point $x$, write $y_0=y^*(x)$ and
$a_0=f_y(x,y_0)$. For $t>0$ and $\|a\|,\|a'\|\le C_f$, define
\begin{equation}\label{up:eq:tilt}
 y_t(a)=\argmin_y\{g(x,y)+t\langle a,y\rangle\},\qquad
 h_t(a)=f_x(x,y_0)+
       \frac{g_x(x,y_t(a))-g_x(x,y_0)}{t}.
\end{equation}
The tilted problem retains the original lower strong-convexity and
smoothness constants. Put
\begin{equation}\label{up:eq:Erho}
 E_\rho=\frac{(1+k)\rho C_f^2}{2\mu^2}.
\end{equation}
Then
\begin{equation}\label{up:eq:tilt-bias}
 \|y_t(a)-y_0\|\le\frac{tC_f}{\mu},\qquad
 \|h_t(a_0)-\nabla F(x)\|\le E_\rho t,
\end{equation}
and
\begin{equation}\label{up:eq:tilt-stability}
 \|y_t(a)-y_t(a')\|\le\frac{t}{\mu}\|a-a'\|,\qquad
 \|h_t(a)-h_t(a')\|\le k\|a-a'\|.
\end{equation}
In particular the ideal estimator is exact for every $t>0$ if $\rho=0$.
A universal hyperobjective curvature bound is
\begin{equation}\label{up:eq:Lambda-universal}
 \operatorname{Lip}(\nabla F)\le
 \Lambda_{\rm univ}:=(1+k)^2\left(L_f+\frac{\rho C_f}{\mu}\right).
\end{equation}
\end{lemma}
\begin{proof}
Strong monotonicity of $g_y(x,\cdot)$, stationarity, and joint gradient
smoothness give
\[
 \mu\|y^*(x)-y^*(x')\|
 \le\|g_y(x,y^*(x'))-g_y(x',y^*(x'))\|
 \le L_g\|x-x'\|.
\]
The same argument applied to the two tilted stationarity equations gives
the solution bounds in
\eqref{up:eq:tilt-bias}--\eqref{up:eq:tilt-stability}.
The bound for $h_t$ follows from the $L_g$-Lipschitz property of $g_x$.

For the bias, let $d=y_t(a_0)-y_0$,
$H=g_{yy}(x,y_0)$, and $B=g_{xy}(x,y_0)$.
Differentiating lower stationarity yields
$\nabla F(x)=f_x(x,y_0)-BH^{-1}a_0$.
Taylor expansion along the lower variable gives
\[
 Hd+e_y=-ta_0,\qquad
 \|e_y\|\le\frac{\rho}{2}\|d\|^2,
\]
\[
 g_x(x,y_0+d)-g_x(x,y_0)=Bd+e_x,\qquad
 \|e_x\|\le\frac{\rho}{2}\|d\|^2.
\]
Thus
\[
 h_t(a_0)-\nabla F(x)=-BH^{-1}e_y/t+e_x/t.
\]
Using $\|BH^{-1}\|\le k$ and $\|d\|\le tC_f/\mu$
proves the bias estimate.

For curvature write $z_x=(x,y^*(x))$ and
$C(z)=g_{xy}(z)g_{yy}(z)^{-1}$. The inverse-difference identity gives
\[
 \|C(z)\|\le k,\qquad
 \|C(z)-C(z')\|\le\frac{\rho(1+k)}{\mu}\|z-z'\|.
\]
The already proved response bound gives
$\|z_x-z_{x'}\|\le(1+k)\|x-x'\|$.
Compare $f_x(z_x)-C(z_x)f_y(z_x)$ at the two points, using
$\|f_y\|\le C_f$ and $\operatorname{Lip}(\nabla f)\le L_f$.
The result is \eqref{up:eq:Lambda-universal}.
\end{proof}

\subsection{An accelerated finite-variance inner solver}

\begin{lemma}[Accelerated distance accuracy with fresh samples]
\label[lemma]{acc:lem:inner}
Let $\phi$ be $\mu$-strongly convex and $M$-smooth, with
$k_\phi=M/\mu\ge1$.  At each adaptive query its gradient estimate
is conditionally unbiased and has conditional variance at most
$\sigma^2$.  From a starting point satisfying
$\mathbb E\|u_0-u^*\|^2\le R^2$, where $u^*=\arg\min\phi$,
a procedure using only these gradients produces $\widehat u$ with
$\mathbb E\|\widehat u-u^*\|^2\le r^2$ using at most
\begin{equation}\label{acc:eq:inner-cost}
 C\left[
 \sqrt{k_\phi}\log\left(2+\frac{k_\phi R^2}{r^2}\right)
 +\frac{\sigma^2}{\mu^2r^2}\right]
\end{equation}
calls.  The guarantee can be conditioned on an arbitrary preceding
history under which the function, the initial distance bound, and the
displayed public parameters are fixed.  The number of calls is fixed
by these parameters before the solve starts.
\end{lemma}
\begin{proof}
Put $\alpha=\sqrt{\mu/M}$ and initialize $v_0=u_0$.  At a step,
form
\[
 w=\frac{u+\alpha v}{1+\alpha},\qquad
 G=\nabla\phi(w)+\xi,
\]
where $G$ is the mean of a fresh batch of size $B$, so that
$\mathbb E[\xi\mid\mathcal H]=0$ and
$\mathbb E[\|\xi\|^2\mid\mathcal H]\le\sigma^2/B$.
Use the same returned vector in the two updates
\begin{equation}\label{acc:eq:iteration}
 u^+=w-G/M,\qquad
 v^+=(1-\alpha)v+\alpha w-(\alpha/\mu)G.
\end{equation}
This uses one queried point per batch and requires no function values.
Consider
\[
 V(u,v)=\phi(u)-\phi(u^*)+\frac\mu2\|v-u^*\|^2.
\]
For the exact-gradient updates $u_e^+,v_e^+$, smoothness,
strong convexity, $\alpha^2/\mu=1/M$, and
$(1+\alpha)w=u+\alpha v$ give
\[
 V(u_e^+,v_e^+)\le(1-\alpha)V(u,v)
 -\frac{\mu(1-\alpha)}2
       \left(\|u-w\|^2+\alpha\|v-w\|^2\right).
\]
For clarity, this follows by expanding the squared norm of
$(1-\alpha)(v-u^*)+\alpha(w-u^*)-\alpha\nabla\phi(w)/\mu$,
using $\phi(u_e^+)\le\phi(w)-\|\nabla\phi(w)\|^2/(2M)$,
and applying strong convexity at $w$ to $u$ and $u^*$ with weights
$1-\alpha$ and $\alpha$.  The gradient-norm terms cancel and the
remaining linear term vanishes by the displayed identity for $w$.

The noisy updates differ from these exact updates by
$-\xi/M$ and $-\alpha\xi/\mu$.  Smoothness for the first
component and expansion of the second squared norm show that the
conditional expected potential increases by at most
\[
 \left(\frac1{2M}+\frac{\alpha^2}{2\mu}\right)
       \mathbb E[\|\xi\|^2\mid\mathcal H]
 \le\frac{\sigma^2}{BM}.
\]
Both linear noise terms have zero conditional expectation.  Hence
\begin{equation}\label{acc:eq:potential}
 \mathbb E[V(u^+,v^+)\mid\mathcal H]
 \le(1-\alpha)V(u,v)+\frac{\sigma^2}{BM}.
\end{equation}
This proof uses smoothness of the mean function only.

If $R\le r$, return the initial point.  Otherwise put
$U_0=MR^2$ and
\[
 S=\left\lceil\log_2\frac{4MR^2}{\mu r^2}\right\rceil,
 \qquad U_s=2^{-s}U_0,\qquad
 K_a=\lceil\alpha^{-1}\log4\rceil.
\]
At stage $s$, run $K_a$ updates with batch size
\[
 B_s=\max\left\{1,
   \left\lceil\frac{2\sigma^2}{M\alpha U_s}\right\rceil
            \right\}.
\]
Keep both states across stage boundaries; only the batch size changes.
Initially $\mathbb EV\le MR^2=U_0$.  Summing
\eqref{acc:eq:potential} within a stage gives
\[
 \mathbb EV_{\rm end}
 \le\tfrac14U_{s-1}+\frac{\sigma^2}{B_sM\alpha}
 \le U_s.
\]
Thus $\mathbb E\|u_{\rm end}-u^*\|^2
\le2U_S/\mu\le r^2/2$.  The geometric sum satisfies
\[
 \sum_{s=1}^S U_s^{-1}\le 2U_S^{-1}
 \le\frac{16}{\mu r^2}.
\]
Since $K_a=O(\sqrt{k_\phi})$ and
$K_a/(M\alpha)=O(1/\mu)$, the total $K_a\sum_sB_s$ is
bounded by \eqref{acc:eq:inner-cost}.  In particular the noise cost
is not multiplied by the initialization logarithm.
\end{proof}

For a fixed outer point and a frozen vector $a$, the tilted objective
\[
 \phi_{x,t,a}(y)=g(x,y)+t\langle a,y\rangle
\]
has exactly the same $\mu,L_g$ and gradient-noise variance as
$g(x,\cdot)$, for every $t>0$.  There is no condition of the form
$tL_f<\mu$: the perturbation here is linear, and $a$ is held fixed
throughout its solve.  The original SFO implements its gradient as
$\widehat g_y(x,y)+ta$.

\subsection{Confidence amplification under conditional finite variance}

The next elementary device allows reuse of an upper estimate without
assuming that it remains conditionally unbiased after reuse.  It is
a finite-dimensional median-distance selector; geometric-median
amplification is another standard option \citep{Minsker2015}.

\begin{lemma}[Conditional confidence amplification]
\label[lemma]{cache:lem:confidence}
Suppose a procedure, when started after any history with a fixed
target $z^*$, returns $Z$ with conditional MSE at most $a^2/144$.
Run it $q$ times with fresh calls, holding the target and public
parameters fixed.  For each returned point compute the smallest
radius containing at least $\lceil q/2\rceil$ returned points,
and return a point with the smallest such radius, breaking ties
by index.  Then
\[
 \mathbb P\{\|\widehat z-z^*\|>a\mid\text{starting history}\}
 \le e^{-q/8}.
\]
This remains true when the original oracle guarantees only the
stated conditional moment bounds; unconditional independence of
the repeated outputs is unnecessary.
\end{lemma}
\begin{proof}
Let $I_i$ indicate $\|Z_i-z^*\|>a/3$.  Conditional Chebyshev
gives $\mathbb E[I_i\mid\text{previous runs}]\le1/16$.
Repeated conditional expectation therefore gives, for $\lambda>0$,
\[
 \mathbb E\exp\left(\lambda\sum_{i=1}^qI_i\right)
 \le\left(1+\frac{e^\lambda-1}{16}\right)^q.
\]
With $\lambda=\log4$, Markov's inequality implies
\[
 \mathbb P\left\{\sum I_i\ge q/4\right\}
 \le\left[(19/16)4^{-1/4}\right]^q\le e^{-q/8}.
\]
If fewer than $q/4$ points are bad, a good point has at least
$\lceil q/2\rceil$ points within distance $2a/3$.  The selected
point therefore has such a ball; this ball contains a good point,
so its distance from $z^*$ is at most $a$.
\end{proof}

For a vector mean with conditional variance at most $\sigma^2$,
one candidate is a mean of
$\lceil144\sigma^2/a^2\rceil\vee1$ fresh observations.
Conditional martingale orthogonality gives its required MSE.
For a lower solution, one candidate is
Lemma~\ref{acc:lem:inner} with target distance $a/12$.
Thus confidence $1-\beta$ costs only a factor
$O(\log(1/\beta))$ and requires neither higher moments nor
mean-square smoothness of sample gradients.

\subsection{Caching the entire upper gradient along the outer path}

Choose
\begin{equation}\label{cache:eq:curvature}
 \Lambda\ge\max\{\operatorname{Lip}(\nabla F),L_a\},\qquad
 \Lambda>0.
\end{equation}
The universal choice
$(1+k)^2(L_f+\rho C_f/\mu)$ satisfies this condition whenever
it is positive; one may always choose a larger positive bound.
At an anchor $a$, compute an accurate lower point $\widehat y_a$
and a high-confidence estimate of the entire vector
$\nabla f(a,\widehat y_a)$.  Store its $x$ component $\widehat b_a$
and project its $y$ component to the radius-$C_f$ ball, obtaining
$\widehat a_a$.  Reuse both stored vectors until the current outer
point leaves the ball of radius
\begin{equation}\label{cache:eq:refresh-radius}
 r_x=\frac{\epsilon}{C L_a}
\end{equation}
around the anchor.  When $L_a=0$, set $r_x=+\infty$ and use a
single anchor.  A refresh takes place at the current point before
its next outer update.

At every current point $x$, solve the untilted lower problem and
the lower problem tilted by the \emph{fixed stored} vector
$\widehat a_a$.  Use fresh lower-gradient means to form
\begin{equation}\label{cache:eq:estimator}
 G(x)=\widehat b_a+
 \frac{\widehat g_x(x,\widehat y_t)
            -\widehat g_x(x,\widehat y_0)}{t}.
\end{equation}
No new upper-gradient query is made between anchor refreshes.

\begin{lemma}[Pointwise error of a cached estimate]
\label[lemma]{cache:lem:error}
Choose a sufficiently large numerical $C$ and put
\[
 t=\min\{1,\epsilon/(CE_\rho)\},\qquad
 r=\frac{\epsilon t}{CL_g},\qquad
 r_A=\min\left\{r,\frac{\epsilon}{C(1+k)L_f}\right\},
\]
where $t=1$ if $E_\rho=0$ and the second entry for $r_A$ is
$+\infty$ when $L_f=0$.  Suppose each of the two current lower
solutions is within $r$ of its exact solution, the anchor lower
solution is within $r_A$, and the anchor full upper-gradient error
is at most $\epsilon/[C(1+k)]$.  Suppose each final lower-gradient
mean has error at most $\epsilon t/C$.  If $\|x-a\|\le r_x$,
then the estimator in \eqref{cache:eq:estimator} satisfies
\[
 \|G(x)-\nabla F(x)\|\le\epsilon/8.
\]
At an anchor the untilted solve itself is performed to radius $r_A$;
no separate lower solve is required for its upper-gradient batch.
\end{lemma}
\begin{proof}
The response graph $\Phi(x)=(x,y^*(x))$ is $(1+k)$-Lipschitz.
Consequently the combined error of the stored $f_x$ and $f_y$,
with the latter multiplied by the tilt-direction sensitivity $k$,
is at most
\[
 (1+k)\left[
   \frac{\epsilon}{C(1+k)}+L_fr_A
                         +L_f(1+k)\|x-a\|\right].
\]
Projection of the stored $f_y$ cannot increase its distance to
$f_y(x,y^*(x))$, which belongs to the same radius-$C_f$ ball.
The last display is $O(\epsilon/C)$ by the choices of $r_A,r_x$.
The ideal frozen-tilt bias is at most $E_\rho t$.
The two lower-solution errors contribute at most $2L_gr/t$;
the final gradient means contribute at most $2\epsilon/C$.
Increasing the fixed constant proves the result.  This is a
deterministic error statement on simultaneous accuracy certificates;
it does not assert conditional unbiasedness of the reused batch.
\end{proof}

\subsection{A path cap, all-history budgets, and the complete upper bound}

\begin{theorem}[Accelerated cached-gradient upper bound]
\label[theorem]{ac:thm:cached-upper}
Under the original primitive, initialization, and fresh-SFO assumptions,
let \eqref{cache:eq:curvature} hold and set
\[
 A=1+\frac{\Lambda\Delta}{\epsilon^2},\qquad
 B=1+\frac{L_a\Delta}{\epsilon^2},\qquad
 t=\min\{1,\epsilon/(CE_\rho)\},
\]
with $t=1$ at $E_\rho=0$.  There is an implementable procedure
with $\mathbb E\|\nabla F(\widehat x)\|\le\epsilon$ and total
query cost
\begin{equation}\label{cache:eq:general-upper}
 \boxed{\widetilde O\!\left[
 A\left(1+\sqrt k+\frac{k^2\sigma_g^2}{\epsilon^2t^2}\right)
 +B\left(1+\frac{(1+k)^2\sigma_f^2}{\epsilon^2}
       +\frac{(1+k)^2L_f^2\sigma_g^2}{\mu^2\epsilon^2}
       \right)\right].}
\end{equation}
All calls, including those on failure histories, obey a deterministic
budget of this order with the logarithms specified below.  In particular
the bound is also an expected-query bound.  If
$\epsilon^2\ge2\Lambda\Delta$, returning $x_0$ requires no calls
and already has the desired guarantee.
\end{theorem}
\begin{proof}
We give the nontrivial case and retain the outer-count floors.

\paragraph{Path and initialization bounds.}
Choose
\[
 N=\max\{1,\lceil32\Lambda\Delta/\epsilon^2\rceil\},\qquad
 E_{\rm cap}=16\Lambda\Delta+N\epsilon^2,
\]
\begin{equation}\label{cache:eq:radius-cap}
 R_{\rm cap}=\frac{\sqrt{NE_{\rm cap}}}{\Lambda},\qquad
 G_{\rm cap}=\sqrt{2\Lambda\Delta}+\Lambda R_{\rm cap},\qquad
 \delta=\min\{1/16,\epsilon/(16G_{\rm cap})\}.
\end{equation}
Before executing a proposed step $x^+=x-G/\Lambda$, check whether
the cumulative sum of squared computed $G$ vectors, including this
one, exceeds $E_{\rm cap}$.  If it does, stop and output $x_0$.
Otherwise execute the step.  If all $N$ steps are completed, output
a uniformly chosen pre-update iterate, with an independent output
index.  This cap uses only computed vectors and public constants.

Every executed path, including a failure path, has total length at most
\begin{equation}\label{cache:eq:path-length}
 \sum_{i\text{ executed}}\|x_{i+1}-x_i\|
 \le R_{\rm cap}
 \le C\left(\frac\Delta\epsilon+\frac\epsilon\Lambda\right).
\end{equation}
Thus every outer point and anchor lies in the deterministic ball
of radius $R_{\rm cap}$ around $x_0$.  All lower solves start from
the same fixed $y_{\rm init}$ and use the certified radius
\begin{equation}\label{cache:eq:inner-radius}
 D=R_0+kR_{\rm cap}+tC_f/\mu.
\end{equation}
This bound applies even after arbitrary earlier estimation failures,
because the cached tilt is always projected to norm at most $C_f$.
The sample counts and accelerated stage counts can therefore be
fixed from $D$ before any lower solve begins.

\paragraph{Simultaneous accuracy and expected stationarity.}
Each iteration has at most two confidence-amplified lower solves
and two confidence-amplified final lower-gradient means.  An anchor
adds one confidence-amplified full upper-gradient mean.  There are
at most $5N$ such operations.  Give each conditional failure
probability $\delta/(10N)$ by
Lemma~\ref{cache:lem:confidence}, using
\[
 q=\left\lceil8\log(10N/\delta)\right\rceil
\]
candidates per operation.  At nonanchors the two lower targets
are $r$; at anchors the untilted target is $r_A$ and the tilted
target remains $r$.  The oracle mean targets and tilt direction
are fixed throughout each operation.  The conditional failure
bound and a union bound apply also to the adaptively selected
anchor times and imply simultaneous validity of all certificates
with probability at least $1-\delta$.

On this good event Lemma~\ref{cache:lem:error} gives
$\|G_i-\nabla F(x_i)\|\le\epsilon/8$.  Smoothness yields
\[
 F(x_i-G_i/\Lambda)
 \le F(x_i)-\frac{\|\nabla F(x_i)\|^2}{2\Lambda}
           +\frac{\|G_i-\nabla F(x_i)\|^2}{2\Lambda}.
\]
For every prefix, including a proposed last step, this implies
\[
 \sum_{i\le t}\|G_i\|^2
 \le4\Lambda\Delta+4(t+1)(\epsilon/8)^2<E_{\rm cap}.
\]
If that last step had been the first rejected one, apply the same
inequality to its hypothetical next point; $F\ge\inf F$ there too.
Thus no cap rejection occurs on the good event.  The good-event
uniform output has conditional expected norm at most
\[
 \left(2\Lambda\Delta/N+\epsilon^2/64\right)^{1/2}
 \le\epsilon/2.
\]

On every failure history the returned point, including an abort
output, is still in the ball of \eqref{cache:eq:radius-cap}.
The initial gap and $\Lambda$-smoothness imply
$\|\nabla F(x_0)\|\le\sqrt{2\Lambda\Delta}$, so its gradient
norm is at most $G_{\rm cap}$.  The bad event therefore contributes
at most $\delta G_{\rm cap}\le\epsilon/16$ to the expected norm.
This proves the desired expectation guarantee without imposing
bounded oracle samples or an a priori bounded-iterate assumption.

\paragraph{A deterministic query budget.}
We count calls on every history.  Successive anchors
are separated by more than $r_x$, so their disjoint intervening
path segments and \eqref{cache:eq:path-length} give
\begin{equation}\label{cache:eq:anchor-count}
 M\le1+R_{\rm cap}/r_x
 \le C\left[1+L_a\left(\frac\Delta{\epsilon^2}
                                  +\frac1\Lambda\right)\right]
 \le CB.
\end{equation}
Also $M\le N$.  This count remains valid if a proposed update
is rejected: the anchor has been taken at an already visited point.

Define the explicit initialization logarithm
\[
 \mathcal L_{\rm in}=\log\left(2+\frac{kD^2}{r_A^2}\right),
 \qquad \mathcal L_{\rm conf}=1+\log(10N/\delta).
\]
Lemma~\ref{acc:lem:inner} and the confidence repetition show that
the two per-iteration lower solves cost at most a numerical multiple of
\[
 \mathcal L_{\rm conf}\left[
 N\sqrt k\,\mathcal L_{\rm in}
 +\frac{N\sigma_g^2}{\mu^2r^2}
 +\frac{M\sigma_g^2}{\mu^2}
           \left(\frac{C(1+k)L_f}{\epsilon}\right)^2\right].
\]
Here $r_A^{-2}\le r^{-2}+[C(1+k)L_f/\epsilon]^2$;
the second part is paid only when a new upper anchor is formed.
The final lower means cost
$C\mathcal L_{\rm conf}N[1+\sigma_g^2/(t^2\epsilon^2)]$,
and the anchor upper means cost
$C\mathcal L_{\rm conf}M[1+(1+k)^2\sigma_f^2/\epsilon^2]$.
Since $k\ge1$, $r^{-2}=C L_g^2/(t^2\epsilon^2)$, and
$N=O(A),M=O(B)$, these estimates give
\eqref{cache:eq:general-upper}.  They also give an explicit
deterministic bound before suppressing logarithms:
\begin{equation}\label{cache:eq:explicit-budget}
 C\mathcal L_{\rm conf}\left\{
 N\left[1+\sqrt k\,\mathcal L_{\rm in}
              +\frac{k^2\sigma_g^2}{t^2\epsilon^2}\right]
 +CB\left[1+\frac{(1+k)^2\sigma_f^2}{\epsilon^2}
       +\frac{(1+k)^2L_f^2\sigma_g^2}{\mu^2\epsilon^2}
       \right]\right\}.
\end{equation}
Every quantity in this expression is a public deterministic function
of the budgets, the requested accuracy, and the initialization
certificate.  All internal oracle points may be unbounded; the
global oracle assumptions apply there, and their number is fixed.
\end{proof}

\subsection{Normalized rates and the variation range}

Assume $C_f,L_f,L_g\le L$, $\kappa=L/\mu\ge2$, and put
$\vartheta=\rho/\mu\ge0$ and $\chi=1+\vartheta$.  Choose the
universal curvature bound in \eqref{up:eq:Lambda-universal}, taking
$\Lambda=L$ if that expression vanishes.  Then
$\Lambda=O(L\chi\kappa^2)$ and \eqref{cache:eq:curvature} holds.
Write
\[
 A_*=\frac{L\chi\kappa^2\Delta}{\epsilon^2},\qquad
 B_*=\frac{L\kappa^2\Delta}{\epsilon^2}.
\]
For every $\vartheta\ge0$ and every $\epsilon>0$,
Theorem~\ref{ac:thm:cached-upper} gives the complete bound
\begin{equation}\label{range:eq:all-variation-upper}
 \boxed{\widetilde O\!\left[\begin{aligned}
 &(1+A_*)\left(\sqrt\kappa+
       \frac{\kappa^2\sigma_g^2}{\epsilon^2}
       +\frac{L^2\vartheta^2\kappa^6\sigma_g^2}{\epsilon^4}\right)\\
 &\quad +(1+B_*)\left(1+
       \frac{\kappa^2\sigma_f^2}{\epsilon^2}
       +\frac{\kappa^4\sigma_g^2}{\epsilon^2}\right)
                         \end{aligned}\right].}
\end{equation}
Indeed, $k\le\kappa$, $E_\rho=O(L\vartheta\kappa^2)$, and
$t^{-2}\le C(1+E_\rho^2/\epsilon^2)$.  The second term inside the
first parentheses retains the $t=1$ branch.  The final term inside
the second parentheses is the additional lower-solve accuracy paid
at upper-gradient refreshes.  The all-history deterministic budget
and initialization logarithms remain those in
\eqref{cache:eq:explicit-budget}.

For $0\le\vartheta\le\kappa^2$, a useful shorter expression is
\begin{equation}\label{cache:eq:normalized-upper}
 \boxed{\widetilde O\!\left[\begin{aligned}
 &(1+A_*)\left(\sqrt\kappa+
       \frac{L^2\vartheta^2\kappa^6\sigma_g^2}{\epsilon^4}\right)\\
 &\quad +(1+B_*)\left(1+
       \frac{\kappa^2\sigma_f^2}{\epsilon^2}
       +\frac{\kappa^4\sigma_g^2}{\epsilon^2}\right)
                         \end{aligned}\right].}
\end{equation}
To obtain it, use $1+A_*\le\chi(1+B_*)$ and
$\chi\le1+\kappa^2\le(5/4)\kappa^2$.  Thus the retained
per-iteration $\kappa^2\sigma_g^2/\epsilon^2$ term in
\eqref{range:eq:all-variation-upper} is dominated by the refresh
term $(1+B_*)\kappa^4\sigma_g^2/\epsilon^2$, with a numerical
constant.  This simplification does not require
$\vartheta\ge1$.

Now suppose
\begin{equation}\label{range:eq:upper-accuracy}
 0<\epsilon\le c\min\{L,\sqrt{L\Delta}\}
\end{equation}
for a fixed numerical $c>0$.  The floors in
\eqref{range:eq:all-variation-upper} are then absorbed.  With
$D,U,G$ given by the algebraic expressions in
\eqref{eq:joint-scales}, define
\begin{equation}\label{range:eq:fourth-power-cost}
 V=L\kappa^4(\chi+\kappa^2)\sigma_g^2\Delta\epsilon^{-4}.
\end{equation}
For every $\vartheta\ge0$ the resulting bound is
\begin{equation}\label{range:eq:four-term-upper}
 \mathsf T_{\rm SFO}(\epsilon)
 \le\widetilde O(D+U+G+V).
\end{equation}
The two summands in $V$ come respectively from the per-iteration
and refresh costs just identified.  This is an upper bound for the
present algorithm; no necessity or matching claim is made for $V$.

For every $\vartheta\ge1$, the additional term is dominated by $G$
throughout \eqref{range:eq:upper-accuracy}.  In fact, when
$\sigma_g>0$,
\[
 \frac{V}{G}
 =\frac{\epsilon^2}{L^2\vartheta^2\kappa^4}
       \left(1+\frac{\kappa^2}{\chi}\right)
 \le C\frac{\epsilon^2}{L^2\kappa^2},
\]
using $\vartheta\ge1$, $\chi\ge2$, and $\kappa\ge2$.
When $\sigma_g=0$ both terms vanish.  Hence, without any upper
restriction on $\vartheta$,
\begin{equation}\label{cache:eq:three-terms}
 \mathsf T_{\rm SFO}(\epsilon)
 \le\widetilde O\!\left(
 L\chi\kappa^{5/2}\Delta\epsilon^{-2}
 +L\kappa^4\sigma_f^2\Delta\epsilon^{-4}
 +L^3\vartheta^2\chi\kappa^8\sigma_g^2\Delta\epsilon^{-6}
                 \right),\qquad \vartheta\ge1.
\end{equation}
All displayed bounds are uniform in the nonnegative noise budgets.
The three-term matching assertion in
\cref{thm:joint-noise-comparison} retains the separately proved
range of its lower-bound families.

\section{Recursive geometry and exact lower primitives}
\label{app:recursive-hierarchy}
This appendix proves the structural geometry, derivative bounds, and
Legendre identities used by the recursive lower family. We first record
them at the convenient base calibration $h_0=b/K$. The independent-base
argument in \cref{app:uniform-accuracy-geometry} supplies the final
parameters; \cref{waI:sec:information} proves the corresponding query count.
The auxiliary extension lemma below is a sufficient calibration of the
same formulas and does not restrict the final theorem. For this auxiliary
numerical calibration we retain $1\le\vartheta\le\kappa$; the
structural identities and raw derivative bounds are reused with the
new calibration in \cref{app:uniform-accuracy-geometry}.

\subsection{A bounded chain with nonpositive frontier residuals}

Choose a fixed $C^\infty$ nondecreasing cutoff $\Psi$ such that
$\Psi=0$ on $(-\infty,1/2]$ and $\Psi=1$ on $[1,\infty)$.
It can be chosen to obey the same square-root derivative estimates as
the original cutoff.  Keep the Gaussian profile $\Phi$ of the original
construction.  In particular $\Phi,\Phi'>0$, $\Phi'\ge1$ on
$[-1,1]$, and
\[
 \Phi(t)+\Phi(-t)=C_\Phi\quad\hbox{for all }t.
\]
For $e=\ell h^2$, define the bounded two-sided chain
\[
 C_{n,\ell,h}(z)=e\left[-\Phi(z_1/h)
 +\sum_{k=2}^n\{
 \Psi(-z_{k-1}/h)\Phi(-z_k/h)
 -\Psi(z_{k-1}/h)\Phi(z_k/h)\}\right].
\]
All coordinate derivatives are nonpositive.  If $|z_n|<h$, the
first coordinate of magnitude less than $h$ has derivative at most
$-g$, where $g=\ell h$.  The derivative and radial bounds are
\begin{equation}\label{rec:eq:chain-jets}
 \begin{gathered}
 \|C\|_\infty\le Cne,\qquad
 \|DC\|\le Cg\sqrt n,\qquad
 \|D^2C\|\le C\ell,\\
 \|D^3C\|\le C\ell/h,\qquad
 \langle z,DC(z)\rangle_+\le Cne.
 \end{gathered}
\end{equation}
These follow term by term from the bounded-degree chain.  For the last
bound a potentially positive term has a negative coordinate multiplying
a nonpositive derivative, and the Gaussian derivative or the compact
activation derivative makes its magnitude bounded by a numerical
multiple of $e$.

Let $H_k$ be the original smooth tail masks at width $h$ and put
\[
 \begin{gathered}
 A(t)=\Psi(t)+\Psi(-t)\in[0,1],\qquad \psi_1=H_1,\\
 \psi_k=A(z_{k-1}/h)H_k\quad(k\ge2),\qquad
 R_k=-e\psi_k\Phi(z_k/h)\le0.
 \end{gathered}
\]
At most one selector has a nonzero jet at a point.  Furthermore
\begin{equation}\label{rec:eq:selector-jets}
 0\le\psi_k\le1,\quad
 \|D^r\psi_k\|\le Ch^{-r}\ (r\le3),\quad
 \|D\psi_k\|\le Ch^{-1}\sqrt{\psi_k},\quad
 \|DR_k\|\le Cg\sqrt{\psi_k}.
\end{equation}
If a selector has a nonzero jet, then $|z_n|<h/2$.
The last fact follows because a nonzero tail mask forces every coordinate
in its tail into $(-h/2,h/2)$.

The carrier $C-\sum_kR_k$ retains the required robust prefix locality.
Indeed the incoming term can be rewritten exactly as
\[
 \Psi(-z_{k-1}/h)\Phi(-z_k/h)
 -\Psi(z_{k-1}/h)\Phi(z_k/h)
 =-A(z_{k-1}/h)\Phi(z_k/h)
   +C_\Phi\Psi(-z_{k-1}/h).
\]
The second summand uses only the predecessor.  On a blind tail the
mask is identically one and removing $R_k$ removes the entire unknown
frontier dependence.  The even predecessor activation is flat whenever
the predecessor has magnitude at most $h/2$.

\subsection{Uniform local geometry with a possibly small weight}

Here $\ell_p,h_p,g_p=\ell_ph_p$ are parent scales.  Set
\[
 r=\ell/\ell_p\ge1,\quad g=\eta g_p,\quad
 h=\eta h_p/r,\quad
 n=\lfloor c_n r/\eta^2\rfloor.
\]
The fixed numbers $\eta,c_n>0$ will be sufficiently small; take
$c_n/\eta^2\ge128$.  The special base middle chain uses
$r=1$, $\eta=1/K$, and $n=\lfloor c_nK^2\rfloor$ instead.

For a prescribed lower weight $\underline w\in(0,1]$, define
\[
 \widetilde\nu=\zeta\ell_p,\qquad
 \nu=\underline w\widetilde\nu,\qquad
 R=C_Rg_p\sqrt r/\nu,
\]
\[
 B(u)=u(1+\|u\|^2/R^2)^{-1/2},\qquad
 P(u)=\nu R^2(\sqrt{1+\|u\|^2/R^2}-1).
\]
For a chain frame $V$, let $z=V^\top B(u)$ and $t=z_n$.
The release gate is $Q=q_0(t/h_p)$, with $q_0=0$ up to $1$,
$q_0=1$ from $2$ onward, and $q_0'\ge0$.
Define two bounded receivers, both zero at zero, by
\[
 D_+'(t)=-A_{\rm rec}(\widetilde\nu t+g)
                  j_0(t/h)k_+(t/h_p),
\]
\[
 T_-=K_-g/\widetilde\nu,\qquad
 D_-'(t)=-A_{\rm rec}g\,j_0(-t/h)k_-(-t/T_-).
\]
Here $j_0=0$ up to $1/4$ and one from $1/2$ onward,
$k_+=1$ up to $2$ and zero from $3$ onward, and $k_-=1$ up to $1$
and zero from $2$ onward.  Both receiver derivatives are nonpositive.
Put $H(u)=C_{n,\ell,h}(V^\top B(u))+D_+(t)+D_-(t)$.
The chain itself is uncentered at intermediate levels.

\begin{lemma}[Uniform weighted local certificate]
\label[lemma]{rec:lem:weighted-local}
Fixed numerical choices of $A_{\rm rec},K_-,\zeta,c_n,\eta,C_R$
have the following properties, uniformly in $r\ge1$ and
$0<\underline w\le1$.  For $w\in[\underline w,1]$,
$R_p\in[-Ce_p,0]$, $e_p=\ell_ph_p^2$, and every $u$,
\begin{align}
 Q<1&\ \Longrightarrow\
 \|D[P+wH+R_pQ]\|\ge c\underline w g,
 \label{rec:eq:local-global}\\
 |t|<h/2&\ \Longrightarrow\
 \|D[P+wH+R_pQ]\|\ge cwg.
 \label{rec:eq:local-active}
\end{align}
Also $\|H\|_\infty\le\delta e_p$, for any prescribed fixed small
$\delta>0$, and
\[
\|H\|\le Cn\ell h^2,\quad
 \|DH\|\le C\ell h\sqrt n,\quad
 \|D^2H\|\le C\ell,\quad \|D^3H\|\le C\ell/h,
\]
The constants are independent of $\underline w$.  The base middle
chain has the same properties for all sufficiently large numerical $K$.
\end{lemma}
\begin{proof}
The value estimate is
\[
 \|H\|_\infty\le
 C[c_n+A_{\rm rec}(\zeta+\eta+K_-\eta^2/\zeta)]e_p.
\]
First choose the receiver constants for the inequalities below, then
$\zeta,c_n$ small, then $\eta$ smaller with
$c_n/\eta^2\ge128$, and finally $C_R$ large.  The same bound for
the base level replaces $\eta$ by $1/K$, so a fixed lower threshold
on $K$ suffices.  Differentiation of the receivers uses
$g/h=\ell$ and $T_-/h=K_-r/\zeta\ge1$.  The radius in fact
obeys the stronger relation
\[
 \frac{R}{h\sqrt n}\ge
 \frac{C_Rr}{\underline w\zeta\sqrt{c_n}}\ge C.
\]
For the base chain the same calculation gives
$R/(h\sqrt n)\ge C_R/(\underline w\zeta\sqrt{c_n})$.
Since $\|DB\|\le1$, $\|D^2B\|\le C/R$, and
$\|D^3B\|\le C/R^2$, this stronger radius bound absorbs
the full $\ell h\sqrt n$ gradient when differentiating the radial
composition.  It preserves both the $O(\ell)$ Hessian and the
$O(\ell/h)$ third derivative bounds.

The latent gradient is
$p=wDC+[wD_+'+wD_-'+R_pQ']e_n$.
Every coordinate of $p$ is nonpositive, and
$\|p\|\le Cg_p\sqrt r$.
Writing $q=(1+\|u\|^2/R^2)^{-1/2}$, the gradient is
\[
 D[P+wH+R_pQ]=c_{\rm rad}B(u)+qVp,
 \qquad c_{\rm rad}=\nu-q\langle z,p\rangle/R^2.
\]
Its radial coefficient lies in $[\nu/2,5\nu/4]$ after increasing
$C_R$.  For the lower estimate use
$\langle z,p\rangle_+\le C(ne+gT_-)$; the positive receiver and
the nonpositive release contribute no positive scalar product.  For
the upper estimate use $|\langle z,p\rangle|\le R\|p\|$.
The exterior radial derivative, when $\|u\|>R$, is at least
$\nu R/\sqrt2-\|p\|$, which dominates $g$.

Inside the radius $q\ge1/\sqrt2$.  If $|t|<h$, the first coordinate
with magnitude below $h$ has chain derivative at most $-g$.
Its receiver contribution is nonpositive; a negative coordinate has
a helpful radial contribution, and a positive one has radial contribution
at most $5\nu h/4\le(5\zeta/(4r))wg$.
This proves a margin $cwg$ in this case.  If $h\le t<2h_p$, the
positive receiver is fully active and
$w\widetilde\nu\ge\nu$; a fixed $A_{\rm rec}$ makes its
terminal component at most $-cwg$.  If $-T_-\le t\le-h$, the
negative receiver is fully active and gives the same margin.
If $t\le-T_-$, the terminal radial component is at most
$-\nu T_-/2=-K_-\underline w g/2$, and all other terminal
terms are nonpositive.  These cases cover $Q<1$.
When $|t|<h/2$, the first case applies and also $Q=Q'=0$, proving
the stronger assertion \eqref{rec:eq:local-active}.
\end{proof}

\subsection{The recursive bounded payload and global margin}

Let the outer macro have length $S$, width $b$ and curvature $L_0$.
Use its original one-sided construction with the capped activation
above, and denote its functions by $H_{\rm out},A_{\rm out},\psi_i^o,R_i^o$.
Thus $A_{\rm out}=H_{\rm out}-\sum_iR_i^o$ and
$R_i^o\le0$.  Its terminal condition implies a gradient margin
$cL_0b$.  The original macro proof uses only nonnegativity and
monotonicity of the activation, its flat part up to $1/2$, its value
at least one from $1$ onward, and fixed derivative bounds; all these
properties hold for the capped activation used here.  Thus the old
macro end-game, radial and prefix-locality proofs apply without change.

There is one common lower vector $u_j$ for each level $0\le j\le J$.
All chain frames at a given level are disjoint groups of columns of
one Haar frame.  A chain at level $j$ is indexed by its prefix
$I=(i_0,\ldots,i_j)$, where $i_0$ is the outer index and each later
index is a transition at the preceding level.  Its children are denoted
by $I,k$.  All chains at level $j$ use the same public $B_j,P_j$.

Put $a_0=1$, $g_0=L_0b/K$, $h_0=b/K$, and
$n_0=\lfloor c_nK^2\rfloor$.  For $j\ge1$ choose $a_j\ge a_{j-1}$
and put
\begin{equation}\label{rec:eq:geometric-scales}
 \varrho_j=a_j/a_{j-1},\quad \ell_j=a_jL_0,\quad
 g_j=\eta^jg_0,\quad h_j=\eta^jh_0/a_j,\quad
 n_j=\lfloor c_n\varrho_j/\eta^2\rfloor.
\end{equation}
The parent scales for level zero are $L_0,b,L_0b$.
Fix a threshold $w_*\in(0,1)$ below and use
$\underline w_j=w_*^{j+1}$ in each level's local certificate.

For an intermediate chain define $A_I=H_I-\sum_kR_{I,k}$.
Starting at level $J$, define bounded payloads backwards by
\[
 W_I=H_I-H_I(0)\quad(j=J),
\]
\begin{equation}\label{rec:eq:payload}
 W_I=H_I-\sum_kR_{I,k}
          +\sum_kR_{I,k}Q_{I,k}
          +\sum_k\psi_{I,k}W_{I,k}\quad(0\le j<J).
\end{equation}
No public radial penalty occurs inside $W_I$.
For small fixed constants in the preceding lemma, backward induction
gives
\begin{equation}\label{rec:eq:payload-value}
 |W_I|\le\delta e_{j-1},\qquad W_I(0)=0,
 \qquad e_j=\ell_jh_j^2,\quad e_{-1}=L_0b^2.
\end{equation}
Indeed $|H_I|\le\delta_0e_{j-1}$,
$|R_{I,k}|\le Ce_j$, at most one selector contributes, and
$e_j/e_{j-1}=\eta^2/\varrho_j$ (or $K^{-2}$ at level zero).
Choose $\delta_0<\delta/2$ and make the latter ratio small enough;
the induction is then uniform in depth.  At zero,
$H_I(0)=R_{I,1}(0)$ at every intermediate level, the release gates
are zero, and the last level was centered, proving the zero assertion.

Define
\begin{equation}\label{rec:eq:hyperobjective}
 \begin{gathered}
 \mathcal F=A_{\rm out}(x)+\sum_{j=0}^JP_j(u_j)
 +\sum_i\psi_i^o(x)W_i(u_0,\ldots,u_J)
 +\sum_iR_i^o(x)Q_i(u_0),\\
 F(x,v)=\mathcal F(x,Kv_0,\ldots,Kv_J).
 \end{gathered}
\end{equation}

\begin{lemma}[Depth-uniform recursive barrier]
\label[lemma]{rec:lem:global}
There are fixed numerical constants $c>0$, $w_*\in(0,1)$ and
geometry constants, independent of $J$, such that an incomplete outer
terminal implies
\begin{equation}\label{rec:eq:margin}
 \|\nabla F\|\ge c\,w_*^{J+1}\eta^JL_0b.
\end{equation}
The function is coercive, $F(0)=0$, and its initial gap is at most
$CSL_0b^2$.
\end{lemma}
\begin{proof}
The outer gradient differs from $DH_{\rm out}$ by at most
$CL_0b[(1-Q_i)+\delta]\sqrt{\psi_i^o}$.
A fixed small selector threshold $w_*$ and a fixed near-release
threshold $1-\gamma$ preserve half the outer margin unless
$\psi_i^o>w_*$ and $Q_i<1$.  This uses only the small bounded
payload \eqref{rec:eq:payload-value}.

Suppose the proof has descended to a level-$j$ chain, whose actual
ancestor weight is $w\in[\underline w_j,1]$ and whose release gate
is less than one.  Its $u_j$ gradient equals the weighted local base
plus, for at most one child index $k$, the perturbation
\[
 w[-(1-Q_{I,k})DR_{I,k}+W_{I,k}D\psi_{I,k}].
\]
If no child selector has a nonzero jet, this perturbation is zero,
and \eqref{rec:eq:local-global} applies.  Otherwise the tail mask
forces $|t_I|<h_j/2$.  The stronger base bound is now
$cwg_j$, by \eqref{rec:eq:local-active}, whereas the perturbation
has norm at most
\[
 Cwg_j[(1-Q_{I,k})+\delta]\sqrt{\psi_{I,k}}.
\]
The same fixed $w_*,\gamma,\delta$ therefore preserve a margin
$cwg_j/2$ if $\psi_{I,k}\le w_*$ or $Q_{I,k}\ge1-\gamma$.
In the remaining case one descends with new actual weight
$w\psi_{I,k}>\underline w_jw_*=\underline w_{j+1}$ and
child release gate below one.  Nonpositive frontier residuals ensure
that the parent release coefficient in each local base is nonpositive.

At the last level there is no perturbation, so the local certificate
gives $c\underline w_Jg_J$.  Each lower block is amplified by $K$.
The smallest bound among all possible stopping levels is a numerical
multiple of $K\underline w_Jg_J$, proving \eqref{rec:eq:margin}.
Notice that a far negative terminal gives only the smaller
$\underline w_j$ margin, but then every child selector jet is zero.
This is why the payload value constant need not shrink with depth.

The payload and release values are bounded, while every public radial
penalty is outside every selector.  The outer carrier and all public
penalties are coercive in their respective variables.  The gap and
the value at zero follow from \eqref{rec:eq:payload-value} and the
outer macro bound.
\end{proof}

\subsection{Selected carrier jets and the upper primitive}

For a level-$j$ prefix $I$, let $w_I$ be the product of its ancestor
selectors, including the outer selector and excluding the selectors
of its own chain.  At most one such product has a nonzero joint jet
at each level, and $0\le w_I\le1$.  Define, for $1\le j<J$,
\[
 r_j=\sum_{I:\,|I|=j+1}w_I
       \left[A_I+\sum_kR_{I,k}Q_{I,k}\right],
 \qquad
 r_J=\sum_{I:\,|I|=J+1}w_IW_I.
\]
The notation $|I|=j+1$ counts the outer index in the prefix.
All these functions vanish when every lower vector is zero.
They contain neither the base middle carrier at level zero nor any
public radial penalty.

At the deepest level use the centered frontier
\[
 R_{\rm fr}=\sum_{I,k}w_I
   \{-\ell_Jh_J^2\psi_{I,k}
            [\Phi(z_{I,k}/h_J)-\Phi(0)]\},
 \qquad A_J=r_J-R_{\rm fr}.
\]
Subtracting this centered version retains prefix locality because
the added centering term depends only on the flat selector on a blind
tail.  Both $A_J$ and $R_{\rm fr}$ vanish at zero.

\begin{lemma}[Complete selected-carrier derivative ledger]
\label[lemma]{rec:lem:carrier-jets}
There is a numerical $C\ge1$ independent of depth, dimensions and
stage counts such that, jointly in $(x,u_0,\ldots,u_J)$,
\begin{align*}
 \|r_j\|&\le C^{J+1}n_j\ell_jh_j^2,&
 \|Dr_j\|&\le C^{J+1}\ell_jh_j\sqrt{n_j},\\
 \|D^2r_j\|&\le C^{J+1}\ell_j,&
 \|D^3r_j\|&\le C^{J+1}\ell_j/h_j.
\end{align*}
The same bounds hold for $A_J$.  The deepest centered frontier obeys
$\|D^kR_{\rm fr}\|\le C^{J+1}\ell_Jh_J^{2-k}$ for $0\le k\le3$.
Its standard latent-coordinate flattening at $\tau\le h_J/16$ has
\[
 |R_{\rm fr}-R_{{\rm fr},\tau}|
       \le C^{J+1}\ell_Jh_J\tau,\qquad
 \|D(R_{\rm fr}-R_{{\rm fr},\tau})\|
       \le C^{J+1}\ell_Jh_J.
\]
No hidden carrier has length $n_0\asymp K^2$.
\end{lemma}
\begin{proof}
The unselected chain and receiver bounds are those of
Lemma~\ref{rec:lem:weighted-local}.  A release term $R_{I,k}Q_{I,k}$
has value $O(\ell_jh_j^2)$ and first three joint derivative norms
$O(\ell_jh_j)$, $O(\ell_j)$, $O(\ell_j/h_j)$: its child gate varies
over the parent width $h_j$, and the child radial radius is at least
that width.  Thus a release term has no finer derivative scale than
its own carrier.

Multiplication by a parent selector introduces scale $h_{j-1}$, not
$h_j$.  For instance its value-to-gradient ratio is controlled by
$h_j\sqrt{n_j}/h_{j-1}\le C\sqrt{c_n/\varrho_j}$, its contribution to
the Hessian is controlled by
$n_j\ell_jh_j^2/h_{j-1}^2\le Cc_n\ell_{j-1}\le C\ell_j$,
and the mixed gradient contribution is at most
$C\ell_jh_j\sqrt{n_j}/h_{j-1}\le C\ell_j$.
The corresponding third derivative estimates follow by one more
factor $h_{j-1}^{-1}\le h_j^{-1}$.
Earlier ancestor widths are still larger.  Iterating the product
rule loses at most a fixed factor per level.  The one-active-selector
property prevents a sum over stages or dimensions at each iteration.
For the centered frontier, the unselected estimates use one frontier
summand rather than a whole chain.  The flattening estimates follow
from the same fixed latent flattening profiles; ancestor selectors
are not flattened.  This proves all joint, including mixed, estimates.
\end{proof}

The part kept in the upper primitive is
\begin{equation}\label{rec:eq:upper-base}
 f_0=A_{\rm out}+P_0+
 \sum_i\psi_i^o\left[A_i+\sum_kR_{i,k}Q_{i,k}\right]
       +\sum_iR_i^oQ_i,
\end{equation}
where $i$ denotes a level-zero chain in the bracket.  Its full Hessian
is at most $CL_0$, and its lower-vector gradient is at most $CL_0b$.
Indeed the base carrier has value $O(L_0b^2)$, gradient $O(L_0b)$,
and Hessian $O(L_0)$; each middle release has value $O(L_0h_0^2)$
and gate width $h_0$; outer selectors have the larger width $b$.
Only $u_0,u_1$ occur in this part.  The public $P_0$ is not selected,
and $\|DP_0\|\le C L_0b$.
Expanding the recursion gives the exact identity
\begin{equation}\label{rec:eq:trace-decomposition}
 \mathcal F=f_0+\sum_{j=1}^Jr_j+\sum_{j=1}^JP_j.
\end{equation}

\subsection{Fine normal coupling}

We record the multiscale lemma used to transmit the deepest frontier.
Its support property is part of the construction, in addition to its
derivative bounds. Choose an even smooth cutoff $\omega$ with
$\omega=0$ on $[-1,1]$ and $\omega=1$ outside $(-2,2)$, and set
\[
 c_\delta(t)=t\omega(t/\delta).
\]
For $0<\delta\le\beta/128$, replace the incoming coordinate of each
latent frontier term by $c_\delta(t)$, leaving its selector and all
ancestor variables unchanged; write the result as $R_\delta$.
This agrees with flattening the latent coordinates before the radial
composition: only coordinates with magnitude below $2\delta$ change,
and every predecessor gate and tail mask is already flat there.

For the centered frontier of \cref{rec:lem:carrier-jets}, with its
curvature and width denoted here by $\ell,\beta$, the resulting joint
bounds are
\[
 \|R-R_\delta\|_\infty\le C\ell\beta\delta,
 \qquad \|D R_\delta\|_\infty\le C\ell\beta,
 \qquad
 \|D^jR_\delta\|_\infty\le C_j\ell\beta\delta^{1-j}
 \quad(j=2,3).
\]
A common multiplicative constant in the selected-frontier ledger
multiplies these bounds and the estimates below. In particular it is
independent of the resolution parameter $v$.
For $\delta'<\delta$, every jet of $R_{\delta'}-R_\delta$ vanishes
unless the active incoming coordinate has magnitude between
$\delta'$ and $2\delta$. At most one frontier selector has a nonzero
jet, so these are joint bounds without a factor equal to the chain length.
This annular support is used explicitly when summing resolutions.

Choose a fixed odd compactly supported smooth profile $q$, equal to
the identity on $[-1/4,1/4]$, and put $q_w(s)=wq(s/w)$.

\begin{lemma}[Fine coupling for arbitrary positive curvature]
\label[lemma]{chi:lem:fine}
Let $v\ge1$, $\lambda_{\rm ext}>0$, and let the selected centered frontier have the joint derivative,
flattening, and annular-support properties just described, at curvature
$\ell>0$ and width $\beta$. There is a telescoping decomposition
$R_{\rm fr}=\sum_{j=0}^J e_j$ and normal widths $w_j$ such that
$w_J\asymp\beta/v$, $w_j\ge w_J$, and
\[
 \left\|D^2\left(\lambda_{\rm ext}\sum_j e_jq_{w_j}\right)\right\|_\infty
 \le C\lambda_{\rm ext}\ell\beta\sqrt v,\qquad
 \left\|D^3\left(\lambda_{\rm ext}\sum_j e_jq_{w_j}\right)\right\|_\infty
 \le C\lambda_{\rm ext}\ell v.
\]
The constants are dimension independent. A fixed choice of width
constant makes $w_j\le\beta/128$ for every $j$.
\end{lemma}
\begin{proof}
Set
\[
 J=\lfloor\log_4(\max\{1,\log v\})\rfloor,\quad
 \gamma_j=1-4^{-(j+1)},\quad
 \lambda_j=1-\tfrac12 4^{-j},
\]
\[
 \delta_j=c_{\rm ms}\beta v^{-\gamma_j}\ (j<J),\qquad
 w_j=c_{\rm ms}\beta v^{-\lambda_j}\ (j\le J),
 \qquad 0<c_{\rm ms}\le1/128.
\]
For $J=0$ take $e_0=R_{\rm fr}$. Otherwise take the selected sums of
the latent components $R_{\delta_0}$,
$R_{\delta_j}-R_{\delta_{j-1}}$ for $1\le j<J$, and
$R-R_{\delta_{J-1}}$ for $j=J$. The sum is exactly $R_{\rm fr}$.
The value bound for the first component is $C\ell\beta^2$ and for
subsequent components is $C\ell\beta\delta_{j-1}$. Every first
derivative is at most $C\ell\beta$. Higher derivatives are bounded
by $C\ell v^{\gamma_j}$ and $C\ell\beta^{-1}v^{2\gamma_j}$,
using $\gamma_{J-1}$ for the last component. These are joint bounds by the preceding derivative ledger. In the
recursive application $\beta=h_J$; every ancestor width and the
relevant wrapper radius are at least $\beta$.

If $A_j,D_{1,j},D_{2,j},D_{3,j}$ denote these four bounds, the
third-order blocks are
\[
 \lambda_{\rm ext}(w_jD_{3,j},\ D_{2,j},\ D_{1,j}/w_j,\ A_j/w_j^2).
\]
The identities $2\gamma_j-\lambda_j=1$,
$2\lambda_j-\gamma_{j-1}=1$, and $2\lambda_0=1$ bound each
entry by $C\lambda_{\rm ext}\ell v$. For the final component use
$2\gamma_{J-1}-\lambda_J\le1$. The second-order blocks are
$\lambda_{\rm ext}(w_jD_{2,j},D_{1,j},A_j/w_j)$; the inequalities
$\gamma_j-\lambda_j\le1/4$ and
$\lambda_j-\gamma_{j-1}\le1/2$ give the asserted bound.
Successive flattening scales differ by factors at least $e^3$.
At any point at most two difference annuli overlap; including the
first component, at most three component jets contribute. Thus summing
introduces no $J$ factor. When $J=0$, $1\le v<e^4$, and the original
residual bounds prove both estimates directly. Finally
$v^{1-\lambda_J}=\exp((\log v)/(2\cdot4^J))$ stays between
fixed numerical constants, proving the smallest-width comparison.
\end{proof}

\subsection{The Legendre lift and exactly prescribed budgets}

The following formulas use the exact public budgets.  Put
\begin{equation}\label{rec:eq:exact-constants}
 m_0=3\mu/2,\quad c_{xy}=L/4,\quad
 K=c_{xy}/m_0=\kappa/6,\quad L_0=\eta_0L,
 \quad d=L/(4\sqrt{2J}),
\end{equation}
where $\eta_0>0$ is a sufficiently small fixed numerical constant.
This branch has $J\ge1$.  There are $J$ hidden normal coordinates
$q_j$ and $J$ radial normal coordinates $t_j$; these are dual
coordinates in the Legendre representation, not the lower normal
coordinates of the eventual primitive.

Choose a fixed odd smooth compact profile $q$, equal to the identity
near zero, and write $q_w(s)=wq(s/w)$.  For $j<J$ set
$W_j=h_jn_j^{1/3}$.  At the deepest level split
$R_{\rm fr}=\sum_le_l$ by Lemma~\ref{chi:lem:fine}, the full
positive-curvature fine-coupling construction, with parameter $v\ge1$ and widths
$w_{\rm fine}\asymp h_J/v$.  Define
\begin{equation}\label{rec:eq:normal-scales}
 W_J=h_Jn_J^{1/3},\qquad
 v=\frac{\rho}
 {C_{\rm fine}^{J+1}(J+1)m_0(\ell_J/d)},
\end{equation}
where $C_{\rm fine}$ is a sufficiently large fixed numerical construction
constant.
The hidden extension is
\begin{equation}\label{rec:eq:hidden-extension}
 E_h=\sum_{j=1}^{J-1}\frac{r_j(x,p)}d q_{W_j}(q_j)
       +\frac{A_J(x,p)}d q_{W_J}(q_J)
       +\sum_l\frac{e_l(x,p)}d q_{w_l}(q_J).
\end{equation}
All flattening acts only in latent deepest-chain coordinates.  The
intermediate selector variables are untouched.
For the radial part use
\begin{equation}\label{rec:eq:radial-extension}
 E_{r,j}(p_j,t_j)=\frac{\nu_jR_j}{d}\,t_j
      \left(\sqrt{R_j^2+\|p_j\|^2+t_j^2}-R_j\right),
 \qquad E=E_h+\sum_{j=1}^JE_{r,j}.
\end{equation}
For $z=(p_0,\ldots,p_J,q_1,\ldots,q_J,t_1,\ldots,t_J)$ put
\[
 \Psi_x(z)=\tfrac12\|z\|^2+E(x,z),\qquad
 g(x,v;y)=m_0\Psi_x^*(y)-c_{xy}\sum_{j=0}^J\langle v_j,u_j\rangle,
\]
\begin{equation}\label{rec:eq:primitives}
 y=(u_0,\ldots,u_J,s_1,\ldots,s_J,\widetilde s_1,\ldots,
 \widetilde s_J),\qquad
 f=f_0(x,u)+d\sum_{j=1}^J(s_j+\widetilde s_j).
\end{equation}
The conjugate is only in $z$, with $x$ held fixed.

At zero dual normal coordinates,
$D_pE=0$, $\partial_{q_j}E=r_j/d$, and
$\partial_{t_j}E=P_j/d$.  Thus the exact lower optimum is
\begin{equation}\label{rec:eq:response}
 u_j^*=Kv_j,\qquad s_j^*=r_j(x,Kv)/d,
 \qquad\widetilde s_j^*=P_j(Kv_j)/d.
\end{equation}
This follows directly by setting the dual coordinate to $(Kv,0,0)$
in $y=D_z\Psi_x(z)$.  There is no graph-linearity condition involving
the possibly nonzero primal normal coordinates.
Substituting into \eqref{rec:eq:primitives} gives exactly
\eqref{rec:eq:hyperobjective}, by \eqref{rec:eq:trace-decomposition}.
At the origin every $r_j$ and every $P_j$ vanishes, so $y^*(0)=0$.

\subsection{The nonuniform curvature recurrence}

Fix a sufficiently large numerical $C_T$ and integer $p\ge6$.
For large enough $\vartheta$, choose
\begin{equation}\label{rec:eq:depth}
 J=\left\lfloor\log_4\log(2+\vartheta)\right\rfloor\ge1,
 \quad T=\frac{\vartheta}{C_T^{J+1}(J+1)^p},\qquad
 a_j=T^{1-4^{-j}}\quad(1\le j\le J).
\end{equation}
The branch is only used once $T\ge2$ and the physical middle floor
condition $K\ge K_0$ holds.  A fixed large threshold on $\vartheta$
ensures both, since $\vartheta\le\kappa$.  The remaining bounded
parameter interval will be handled by the original lower theorem.
The decisive identity is
\begin{equation}\label{rec:eq:curvature-recurrence}
 a_j(a_j/a_{j-1})^{1/3}=T,
 \qquad a_j^4=T^3a_{j-1}.
\end{equation}
Consequently $a_jn_j^{1/3}\le CT$ at every high-curvature level.
The final curvature obeys
\begin{equation}\label{rec:eq:final-curvature}
 a_J\ge e^{-4}T.
\end{equation}
Indeed $4^J\le\log(2+\vartheta)<4^{J+1}$, so
$4^{-J}\log T\le4$.

\begin{lemma}[Global extension budget]
\label[lemma]{rec:lem:extension-budget}
Choose $C_T$ and the fixed normal constants sufficiently large.
There is a numerical $C_{\rm acc}\ge1$ such that, when
\begin{equation}\label{rec:eq:accuracy}
 b=\frac{C_b\epsilon}{L_0(w_*\eta)^J},\qquad
 0<\epsilon\le cC_{\rm acc}^{-J}(J+1)^{-q_{\rm acc}}
              \min\{L/(\kappa\vartheta),\sqrt{L\Delta}\},
\end{equation}
with a sufficiently large fixed numerical integer $q_{\rm acc}$,
the global joint extension bounds satisfy
\[
 \|D^2E\|\le\delta_E,\qquad
 \|D^3E\|\le\frac{\rho}{C_{\rm Leg}m_0},
\]
where $\delta_E>0$ is any sufficiently small fixed numerical constant
and $C_{\rm Leg}$ is the constant in the Legendre regularity lemma.
The fine parameter in \eqref{rec:eq:normal-scales} is at least one.
Moreover the bounded hidden gradient obeys
\begin{equation}\label{rec:eq:hidden-small}
 \|DE_h\|_\infty\le C^{J+1}\sqrt{J+1}\,b^2/K^2,
\end{equation}
and is smaller than the required fixed multiple of every blind
threshold, including the deepest threshold $w_{\rm fine}/8$.
\end{lemma}
\begin{proof}
For a coarse term, Lemma~\ref{rec:lem:carrier-jets} and
$W_j=h_jn_j^{1/3}$ give
\[
 \|D^2 E_{h,j}\|\le C^{J+1}(\ell_j/d)h_jn_j^{2/3},
 \qquad
 \|D^3 E_{h,j}\|\le C^{J+1}(\ell_j/d)n_j^{1/3}.
\]
The fine contribution has corresponding bounds
$C^{J+1}(\ell_J/d)h_J\sqrt v$ and
$C^{J+1}(\ell_J/d)v$.
Here the full positive-curvature fine lemma is used, including all
mixed selector derivatives; its amplitude parameter is not restricted
to at most one.
Using $d\asymp L/\sqrt J$ and
\eqref{rec:eq:curvature-recurrence}, the sum of coarse third derivative
costs is at most $C^{J+1}(J+1)^{3/2}T$.
The fine cost is at most $\rho/(C m_0)$ by its definition after
increasing $C_{\rm fine}$.  The choice of $T$, with $C_T$ larger
than these fixed constants, pays both costs and ensures $v\ge1$.

The radial extension satisfies
\[
 \|D^2E_{r,j}\|\le C\nu_jR_j/d
     \le C\sqrt{J+1}\,(b/K)\eta^{j-1}\sqrt{\varrho_j},
 \qquad
 \|D^3E_{r,j}\|\le C\nu_j/d
     \le C\sqrt{J+1}\,a_{j-1}.
\]
The first expression does not contain $\underline w_j^{-1}$ because
$\nu_jR_j=C_Rg_{j-1}\sqrt{\varrho_j}$.  Summing its third derivative
bound costs at most $C(J+1)^{3/2}T$, also paid by the chosen $T$.
These radial estimates are global and dimension independent.

For the Hessian bounds note $n_j\le C\varrho_j\le C\vartheta^{3/4}$,
$\vartheta\le\kappa$, and $g_j=\eta^jL_0b/K$.  The coarse
Hessian sum is at most $C^{J+1}\sqrt{J+1}\,b/\sqrt\kappa$;
the fine sum has the same or a smaller bound because
$v\le C\vartheta/a_J$.  The radial Hessian sum is at most
$C(J+1)^{3/2}b$.  Enlarging $C_{\rm acc}$ in
\eqref{rec:eq:accuracy} makes these sums at most $\delta_E$.

For the bounded hidden gradient, both the normal derivative and the
tangential derivative of a coarse term are at most
$C^{J+1}n_j\ell_jh_j^2/d$; the tangential estimate uses
$n_j^{5/6}\le n_j$.  The fine sum has the same bound at its level.
Since
$n_j\ell_jh_j^2\le C\ell_{j-1}h_{j-1}^2$
and the latter values decrease geometrically, their sum is at most
$CL_0b^2/K^2$.  This proves \eqref{rec:eq:hidden-small}.
All nondeep blind widths are at least a fixed multiple of $h_J$.
The deepest one obeys
\[
 w_{\rm fine}/8\ge c h_J a_J/\vartheta
       =c\eta^J b/(K\vartheta).
\]
The ratio of \eqref{rec:eq:hidden-small} to this quantity is at most
$C^{J+1}\sqrt{J+1}\,b\vartheta/(K\eta^J)$.
It is arbitrarily small under \eqref{rec:eq:accuracy} after one more
fixed enlargement of $C_{\rm acc}$.  This is the no-cascade budget
needed by the inverse-locality lemma.
\end{proof}

Lemma~\ref{rec:lem:legendre} now gives global strong
convexity and the exact variation budget for the lower primitive.
For completeness, $g_{yy}\succeq m_0/(1+\delta_E)I\succeq\mu I$.
Before the Legendre perturbation, each $(v_j,u_j)$ quadratic block has
norm
\[
 \frac{m_0+\sqrt{m_0^2+4c_{xy}^2}}2
       \le\frac{3+\sqrt{13}}8L<L.
\]
All such blocks and all normal blocks are diagonal copies; their
operator norms do not add with depth.  The perturbation has full joint
Hessian at most $C m_0\delta_E$, which leaves $L_g\le L$.
The third derivative budget above leaves every required Hessian-block
variation at most $\rho$.

For the upper primitive, choose $\eta_0$ so that $CL_0\le L$ in
\eqref{rec:eq:upper-base}; hence $L_f\le L$.  Its geometric lower
gradient is at most $CL_0b\le L/2$ under the accuracy condition.
The $2J$ linear normal coefficients have joint norm
$d\sqrt{2J}=L/4$.  Thus $\|f_y\|<L$, globally, with the exact
public budget.  No unbounded public radial value was ever multiplied
by a selector.

\subsection{A partial Legendre lift and its joint primitive budgets}
\label{rec:sec:legendre}

The conjugates below are taken only in the displayed dual lower
variables. All derivative norms are joint Euclidean multilinear
operator norms. In particular, estimates involving the macro variable
$x$ include its mixed derivatives with every lower variable.

\begin{lemma}[Joint partial-Legendre estimates and exact response]
\label[lemma]{rec:lem:legendre}
Let
\[
 \Psi(x,z)=\tfrac12\|z\|^2+E(x,z),\qquad
 \|D^2E\|_\infty\le\delta\le\tfrac14,
 \qquad \operatorname{Lip}(D^2E)\le H.
\]
Define $\mathcal L(x,y)=\Psi_x^*(y)$. This is a globally defined
$C^{2,1}$ function. Its lower Hessian satisfies
\[
 (1+\delta)^{-1}I\preceq\mathcal L_{yy}
 \preceq(1-\delta)^{-1}I,
\]
and
\[
 \|D^2_{(x,y)}\mathcal L\|_\infty\le C,
 \qquad \operatorname{Lip}(D^2_{(x,y)}\mathcal L)\le CH.
\]
More precisely,
$\|D^2\mathcal L-\operatorname{diag}(0_x,I_y)\|\le C\delta$.
Write $z=(p,q)$, with $q$ containing all normal coordinates. If
\[
 E(x,p,0)=0,\qquad E_q(x,p,0)=t(x,p),
\]
then the lower objective
\[
 g(x,v,u,s)=m\mathcal L(x,u,s)-c\langle v,u\rangle
\]
has the exact response
\[
 u^*(x,v)=K v,\qquad s^*(x,v)=t(x,Kv),\qquad K=c/m.
\]
Its lower modulus is at least $m/(1+\delta)$, its joint gradient
smoothness is $O(m+c)$, and its joint Hessian variation is $O(mH)$.
\end{lemma}
\begin{proof}
For fixed $x$, $\Psi_x$ is $(1-\delta)$-strongly convex. Thus
\[
 y=z+E_z(x,z)
\]
has one solution $z=z(x,y)$ for every $y$. At this point put
$A=I+E_{zz}$, $B=E_{xz}$, and $C_0=E_{xx}$. Differentiating the
envelope identity gives
\[
 \mathcal L_y=z,\quad \mathcal L_x=-E_x,
 \quad \mathcal L_{yy}=A^{-1},\quad
 \mathcal L_{xy}=-BA^{-1},\quad
 \mathcal L_{xx}=-C_0+BA^{-1}B^\top.
\]
These formulas prove the Hessian estimates. For a direction
$h=(a,b)$ define
\[
 T h=(a,A^{-1}(b-B^\top a)).
\]
Three derivatives of the envelope identity give
\[
 D^3\mathcal L[h_1,h_2,h_3]
   =-D^3E[Th_1,Th_2,Th_3].
\]
The terms containing a second derivative of the optimizing $z$ cancel
by its first-order equation. Since $\|T\|\le C$, this proves the
third-derivative estimate. For $C^{2,1}$ data the same formula holds
almost everywhere and gives the Lipschitz estimate by integration.

The lower first-order equation is
$\nabla_y\mathcal L=(Kv,0)$. Equivalently,
$y=\nabla_z\Psi(x,Kv,0)$. Since $E(x,p,0)=0$ identically in $p$,
$E_p(x,p,0)=0$, and the response follows. In particular exactness
uses the trace at dual normal coordinate zero. It does not require a
normal feature to remain linear at the generally nonzero primal
response $s^*$.
\end{proof}

\begin{lemma}[A global public radial extension]
\label[lemma]{rec:lem:radial}
For $R,\nu,d>0$, put $a=\nu R/d$ and
\[
 \mathcal P(p,t)=a t\bigl(\sqrt{R^2+\|p\|^2+t^2}-R\bigr).
\]
Then
\[
 \mathcal P(p,0)=0,\qquad
 \mathcal P_t(p,0)=\frac{\nu R^2}{d}
       \left(\sqrt{1+\|p\|^2/R^2}-1\right),
\]
and, independently of the dimension,
\[
 \|D^2\mathcal P\|_\infty\le3a,
 \qquad \|D^3\mathcal P\|_\infty\le9\nu/d.
\]
Moreover $\mathcal P_p=\gamma(p,t)p$, where $|\gamma(p,t)|\le a$.
\end{lemma}
\begin{proof}
Write $z=(p,t)$ and $r=(R^2+\|z\|^2)^{1/2}$. The Hessian is
\[
 a\left[e_t\otimes(z/r)+(z/r)\otimes e_t+tD^2r\right].
\]
Use $\|D^2r\|\le1/r$, $\|D^3r\|\le6/r^2$, and $|t|\le r$.
Finally $\mathcal P_p=(at/r)p$. These estimates hold at every
normal query, with no restriction on $t$.
\end{proof}

For the recursive construction take one radial extension for each
physical lower block $p_j$, $1\le j\le J$. Their normal coordinates
are denoted $t_j$. The block $p_0$ retains its public radial term in
the upper primitive. The hidden normal coordinates may be denoted
$q_1,\ldots,q_J$, one per high-curvature level; the final normal
$q_J$ carries both its coarse component and the finest frontier.
The Euclidean norm of all $2J$ upper linear coefficients must obey
the stipulated upper-gradient budget.

\subsection{A fixed conservative law after conjugation}

\begin{lemma}[The full informative-gradient envelope]
\label[lemma]{rec:lem:conjugate-envelope}
Suppose $E$ and $E^0$ both obey the Hessian smallness condition in
Lemma~\ref{rec:lem:legendre}, and let
\[
 \eta_E=\|\nabla_{(x,z)}(E-E^0)\|_\infty,
 \qquad \mathcal L^0=(\tfrac12\|z\|^2+E^0)^*.
\]
Then
\[
 \|\nabla_{(x,y)}m(\mathcal L-\mathcal L^0)\|_\infty
 \le \frac{\sqrt2m}{1-\delta}\eta_E.
\]
Consequently, if $B$ is this envelope or any larger bound, the fixed
sample functions
\[
 g_\xi=m\mathcal L^0-c\langle v,u\rangle
       +\frac\xi p m(\mathcal L-\mathcal L^0),\qquad
 p=\frac{B^2}{\sigma_g^2+B^2},\quad
 \xi\sim\operatorname{Ber}(p),
\]
give globally unbiased conservative full gradients with conditional
variance at most $\sigma_g^2$. Here each coin is fresh and independent
of the frames and previous observations.
\end{lemma}
\begin{proof}
Let $z,z^0$ be the two inverse points at the same $(x,y)$. Strong
monotonicity of the second dual gradient gives
\[
 \|z-z^0\|\le\eta_E/(1-\delta).
\]
The $y$-gradient difference equals $m(z-z^0)$. For the macro block,
\begin{align*}
 \|\mathcal L_x-\mathcal L_x^0\|
 &\le\|(E-E^0)_x(x,z)\|
       +\|E_x^0(x,z)-E_x^0(x,z^0)\|\\
 &\le\eta_E+\delta\|z-z^0\|
 \le\eta_E/(1-\delta).
\end{align*}
This includes every mixed dependence on the macro variables. The
difference is independent of $v$, so it is also the full-gradient
bound for the original variables $(x,v,u,s)$. A common radial term
can have an unbounded gradient: the proof uses its bounded Hessian,
not a bound on that common gradient. The sample gradients are those
of the displayed global scalar functions, and their variance is
$(p^{-1}-1)\|m\nabla(\mathcal L-\mathcal L^0)\|^2\le\sigma_g^2$.
\end{proof}

The inactive potential is fixed globally. Only the deepest frontier
is flattened. If $R$ denotes that selected, centered frontier, write
its multiscale decomposition as $R=\sum_{k=0}^{K_f} e_k$, with
smallest normal width $w_{K_f}$, and set
\[
 \tau=w_{K_f}/8,\qquad \Delta R=R-R_\tau.
\]
Keep every coarse hidden carrier, every earlier fine atom, and every
public radial extension unchanged, and replace $e_{K_f}$ by
$e_{K_f}-\Delta R$. Thus
\begin{equation}\label{rec:eq:fixed-dual-difference}
 E-E^0=\frac{\Delta R}{d}\,q_{w_{K_f}}(q_J).
\end{equation}
If $K_f=0$, the unique inactive fine atom is $R_\tau$. All
flattening acts in the latent deepest coordinates, before the radial
wrapper; it never acts in the ancestor variables. The functions
$E,E^0$, their conjugates, and the sample family are fixed by the
sampled frames and public parameters before an algorithm executes.
There is no revealed-prefix or transcript dependence in this law.

\subsection{Exact locality of the recursive latent functions}

The signed bounded chain admits the nonpositive frontier used in the
recursive construction. Namely, if
$\Phi(t)+\Phi(-t)=C_\Phi$ and
$s(t)=\Psi(t)+\Psi(-t)$, then each edge satisfies
\[
 \Psi(-z_{i-1}/\beta)\Phi(-z_i/\beta)
 -\Psi(z_{i-1}/\beta)\Phi(z_i/\beta)
 =-s(z_{i-1}/\beta)\Phi(z_i/\beta)
   +C_\Phi\Psi(-z_{i-1}/\beta).
\]
With $\psi_1=h_1$ and $\psi_i=s(z_{i-1}/\beta)h_i$, the
uncentered residual terms are
$R_i=-\ell\beta^2\psi_i\Phi(z_i/\beta)\le0$.
The capped gates ensure $0\le\psi_i\le1$. If every coordinate
after $k$ has magnitude at most $\beta/4$, the latent gradient of
$C-\sum_iR_i$ is supported on $[k]$, and that of $\sum_iR_i$
on $[k+1]$. The incoming edge at $k+1$ is cancelled exactly because
its tail mask is flat one. The extra signed predecessor term in
the display uses only coordinate $k$. All later gates are flat zero.

For completeness, only one residual summand can have a nonzero jet:
take the last coordinate of magnitude at least $\beta/2$, with
the index zero if none exists. All earlier masks have flat zero
jets and all later predecessor gates have flat zero jets. Equality
at $\beta/2$ is covered by the flat endpoint values of the cutoffs.
The deepest centered residual replaces each summand by
$-\ell\beta^2\psi_i[\Phi(z_i/\beta)-\Phi(0)]$.
Its additional term $\ell\beta^2\Phi(0)\sum_i\psi_i$ is not
globally constant. On a blind tail its gradient uses only known
predecessors and flat tail selectors, so the same prefix support
holds. Flattening its latent child at $\tau\le\beta/1024$ leaves
every predecessor gate and mask unchanged and removes the gradient
in every unseen child.

These observations iterate down the tree. A selector for a future
node has zero value and zero first derivative; the selector for the
current child has derivatives only in the known parent prefix. A
release across an unfinished subtree has $Q=Q'=0$, and its terminal
receivers have zero derivative, on the unseen terminal interval.
Thus, under the pre-activation tail conditions, the full inactive
hidden gradient uses known frame columns and public radial directions;
the active hidden gradient can add only the next deepest column.
Products of ancestor selectors introduce scalar coefficients and
known ancestor gradients. These statements hold for every value of
all normal coordinates. The exact upper gradient has the same
property: its remaining base carrier and base-to-first-level release
obey these support conditions at the actual query variables, while
the linear normal terms and the base radial term are public.

\subsection{Restricted inverses and arbitrary queries}

\begin{lemma}[A globally defined virtual inverse with no cascade]
\label[lemma]{rec:lem:inverse-locality}
Suppose the preceding latent locality conditions hold, both dual
potentials have Hessian smallness at most $1/4$, and their common
public radial extensions have $|\gamma_j|\le1/2$. Set
$\tau_j^{\rm loc}=h_j/4$ for $j<J$ and
$\tau_J^{\rm loc}=\tau=w_{K_f}/8$. Define
\[
 H_{\rm hid}=\max\{\|D E_h\|_\infty,\|D E_h^0\|_\infty\},
\]
where the derivatives include the tangent, normal, and macro variables,
and assume $H_{\rm hid}\le\min_j\tau_j^{\rm loc}/16$.
There is a projection-observer protocol
defined on every frame realization whose returned gradients equal
the original gradients on its projection event. At most one new
deepest column is needed per successful coin; completing its subtree
may reveal the corresponding parent columns for free.
\end{lemma}
\begin{proof}
For each lower block maintain a known subspace $W_j$, containing all
inserted query directions and fully revealed columns. At a query
insert its outer $v_j$ and lower $u_j$ directions before the coin;
disclose every frame column's projection on the enlarged space.
The macro observer inserts $x$ in $W_x$ in the same way. On a
successful coin reveal the next deepest column and each parent
column whose subtree is now complete, and enlarge these spaces.

If the coin is inactive ($\xi=0$), compute only the virtual inverse
for $E^0$. If it is active ($\xi=1$), compute the virtual inverses for
both $E$ and $E^0$ after the eligible source column has been included.
For each potential required in that branch, define its virtual inverse
as the unique optimizer of
\[
 \min_{p_j\in W_j,\ q,t\ {\text{unrestricted}}}
 \Psi(x,p,q,t)-\langle y,(p,q,t)\rangle.
\]
All normal coordinates are public. Strong convexity gives a unique
optimizer, even off the projection event. Every latent coordinate
and every projected derivative in this restricted problem is a
function of disclosed projections, so its solution is measurable
with respect to the observer's information. Return $m$ times this inverse in
the lower gradient, the $W_x$ projection of $-mE_x$ in the macro
gradient, and the public bilinear contributions exactly; combine
the two virtual gradients by the fixed sample coefficients. This
defines the virtual responses without conditioning on a good event.

It remains to prove agreement. Let $p_j$ be a restricted inverse
point. Its first-order equation and Lemma~\ref{rec:lem:radial} give
\[
 p_j=\lambda_j
 \bigl(u_j-P_{W_j}(E_{\rm hid})_{p_j}\bigr),
 \qquad \tfrac12\le\lambda_j\le2.
\]
The same formula holds for $E^0$. The bounded wrapper $B_R$ is
one-Lipschitz, and
\[
 |w^\top B_R(\lambda u)|
 \le\max\{1,\lambda\}|w^\top B_R(u)|\qquad(\lambda>0).
\]
Therefore a column $w$ that was unknown before the current coin,
including one just revealed by that coin, satisfies
\begin{equation}\label{rec:eq:no-cascade}
 |w^\top B_{R_j}(p_j)|
 \le2|w^\top B_{R_j}(u_j)|+2H_{\rm hid}.
\end{equation}
Use a projection event that bounds the pre-coin query coordinate
by $\tau_j^{\rm loc}/16$. The right side is at most $\tau_j^{\rm loc}/4$. A still-unknown
column can alternatively be controlled directly by its projection
onto $W_j$, since $p_j\in W_j$. Thus every coordinate unknown
before the coin remains in its pre-activation flat interval at each
inverse point required in the current branch. This includes newly
exposed middle columns on an active call; inversion
cannot trigger a further child gate in the same sample.

The latent locality verified above now shows that every component
of the full gradient of each required restricted inverse problem
perpendicular to its restricted spaces is zero. For $E^0$ only old
columns are needed, so its restricted inverse equals its full inverse
on every call. The mean potential $E$ is needed only when $\xi=1$;
on that branch its one additional deepest column has just been
revealed, so its restricted inverse also equals its full inverse.
Public radial derivatives are parallel to $p_j\in W_j$, and the
macro projection discards nothing in each of these computations.
The exact upper gradient already has its pre-coin prefix support at
the actual query. Consequently the virtual lower response equals
\[
 \nabla g^0+\frac{\xi}{p}(\nabla g-\nabla g^0),
\]
and the exact upper response agrees as well. No equality with a
restricted mean inverse is asserted on an inactive call.

No step limits the magnitude of a normal query. The radial scalar
$\gamma_j$ is bounded for every normal coordinate, and the hidden
gradient envelope is global. Scalar observations are fully retained:
the restricted solutions, their norms, and all normal outputs are
computable from the disclosed projections. The argument thus does
not presume that exact scalar observations carry no information.
\end{proof}

\section{Uniform accuracy through an independent base-chain length}
\label{app:uniform-accuracy-geometry}

The earlier sufficient accuracy restriction coupled two quantities that
need not coincide: the affine response gain $K$ and the ratio between
the outer macro width and the base middle width.  We now separate them.
The recursive profiles, higher-level carriers, Legendre extensions, and
endpoint transport are unchanged.  Only the macro width and the number
of coordinates in its base middle chain are recalibrated.

\begin{theorem}[Uniform accuracy and the complete three-term comparison]
\label[theorem]{waG:thm:uniform-matching}
\label[theorem]{waG:thm:uniform-lower}
There are numerical constants $c,c'>0$ and $q<\infty$ such that, for
every $L,\Delta>0$, $\kappa\ge2$, $1\le\vartheta\le\kappa^2$, and
\begin{equation}\label{waG:eq:public-accuracy}
 0<\epsilon\le c\min\{L,\sqrt{L\Delta}\},
\end{equation}
the exactly prescribed primitive class, with $\mu=L/\kappa$,
$C_f,L_f,L_g\le L$, Hessian-block variation at most
$\rho=\mu\vartheta$, and a finite lower-initialization certificate,
satisfies, uniformly for $\sigma_f,\sigma_g\ge0$,
\begin{equation}\label{waG:eq:comparison}
 c'\left[U+\frac{D+G}{[\log(2+\vartheta)]^q}\right]
 \le \mathsf T_{\rm SFO}(\epsilon)
 \le \widetilde O(D+U+G),
\end{equation}
where
\[
 D=L(1+\vartheta)\kappa^{5/2}\Delta\epsilon^{-2},\qquad
 U=L\kappa^4\sigma_f^2\Delta\epsilon^{-4},\qquad
 G=L^3\vartheta^2(1+\vartheta)\kappa^8
              \sigma_g^2\Delta\epsilon^{-6}.
\]
The lower bounds use three separate conservative oracle families, all
with $y^*(0)=0$.  They cover arbitrary randomized adaptive queries,
arbitrary measurable terminal outputs, and worst-case expected costs.
All dimensions are finite and are selected before execution.
The constant in \eqref{waG:eq:public-accuracy} is independent of
$\kappa,\vartheta$, and of the recursion depth.  The rates still lose
fixed powers of $\log(2+\vartheta)$; their exponents are not optimized.
\end{theorem}

\subsection{Parameters and the order of their choice}

Fix the profiles and numerical constants in the weighted local
certificate, including $w_*,\eta,c_n,\zeta,C_R$, with
$0<w_*,\eta<1$.  Write
\[
 \omega=w_*\eta.
\]
Let $C_g>1$ absorb the dimension-independent carrier and split-endpoint
derivative constants.  Fix a sufficiently small numerical
$\eta_{\rm curv}>0$ and put $L_0=\eta_{\rm curv}L$.
This symbol is distinct from the variable base width ratio below.
The public upper normal coefficients remain
\[
 d=\frac{L}{4\sqrt{2J}},\qquad m_0=\frac{3\mu}{2}
 \quad(J\ge1).
\]
For the ordinary recursive lift use $K=\kappa/6$.
For the deterministic endpoint lift use precisely the $K$ and matrix
normalization in \eqref{tr:eq:normalization}.
In either recursive branch there are fixed numerical constants with
\begin{equation}\label{waG:eq:K}
 c_K\kappa\le K\le C_K\kappa.
\end{equation}
The different constant choices do not affect the following estimates.

Choose $C_{\rm fine}$ after the preceding constants.  In particular it
can and will satisfy
\begin{equation}\label{waG:eq:fine-choice}
 C_{\rm fine}>
 C\max\{C_g/(w_*\eta^2),(C_g/w_*)^2,1\}.
\end{equation}
The fixed numerical $C$ here also absorbs the constants of the
information and Legendre lemmas.  Next choose $C_T$ sufficiently large,
in particular larger than the constants required for third-derivative
allocation and
\begin{equation}\label{waG:eq:T-choice}
 C_T>
 C\max\{C_{\rm fine},(C_g/\omega)^2,\omega^{-8/3},1\}.
\end{equation}
Choose a fixed integer $p\ge6$, increasing it for the fixed polynomial
factors in depth if needed.  In the recursive branch set
\begin{equation}\label{waG:eq:depth}
 \begin{gathered}
 J=\left\lfloor\log_4\log(2+\vartheta)\right\rfloor,\qquad
 D_0=J+1,\qquad
 T=\frac{\vartheta}{C_T^{D_0}D_0^p},\\
 a_0=1,\qquad a_j=T^{1-4^{-j}},\qquad
 \varrho_j=a_j/a_{j-1}.
 \end{gathered}
\end{equation}
Use this branch only for $\vartheta>\Theta_*$, where the fixed numerical
threshold $\Theta_*$ is chosen after $C_T$ so that $J\ge1$ and $T\ge2$.
This is possible because $C_T^{D_0}D_0^p$ is at most a fixed power of
$\log(2+\vartheta)$. Increase $\Theta_*$ to be at least
$\kappa_{\rm tr}^2$ and $(K_0/c_K)^2$, where $K_0$ is any fixed
gain threshold used by the local certificates. Since
$\kappa\ge\sqrt\vartheta$, every required transport and local
normalization is then available. The bounded branches are handled below.

Fix $C_b$ sufficiently large for the final gradient margin.  Fix
$c_S>0$ sufficiently small for the outer gap certificate, then fix
$b_*>0$ sufficiently small that
\[
 b_*^2\le c_S/128
\]
and the upper-gradient bound below has numerical slack.  Define
\begin{equation}\label{waG:eq:base-parameters}
 \begin{gathered}
 b_0=b_*\min\{1,\sqrt{\Delta/L_0}\},\qquad
 S=\left\lfloor\frac{c_S\Delta}{L_0b_0^2}\right\rfloor,\\
 h_0=\frac{C_b\epsilon}{K L_0\omega^J},\qquad
 \xi_0=\frac{h_0}{b_0},\qquad
 n_0=\left\lfloor c_n\xi_0^{-2}\right\rfloor .
 \end{gathered}
\end{equation}
Thus $h_0$ is the same small middle width used by the previous
accuracy-calibrated construction, but $b_0$ is now independent of it.
The identity $b_0=Kh_0$ is not imposed or used.  For $1\le j\le J$ set
\begin{equation}\label{waG:eq:higher-scales}
 \ell_j=a_jL_0,\qquad
 h_j=\frac{\eta^j h_0}{a_j},\qquad
 g_j=\eta^jL_0h_0,\qquad
 n_j=\left\lfloor c_n\varrho_j\eta^{-2}\right\rfloor .
\end{equation}
At the base level $\ell_0=L_0$ and $g_0=L_0h_0$.
Only after all these fixed constants and the finite branch thresholds
have been chosen do we decrease the numerical accuracy constant $c$
in \eqref{waG:eq:public-accuracy}.

\subsection{The base certificate does not identify geometry with response gain}

\begin{lemma}[Base chain at an independent width ratio]
\label[lemma]{waG:lem:base}
There is a numerical $\xi_*>0$ such that the base weighted local
certificate remains valid for every $0<\xi_0\le\xi_*$ with
$h_0=\xi_0b_0$ and
$n_0=\lfloor c_n\xi_0^{-2}\rfloor\ge64$.
The parent scales are $L_0,b_0,L_0b_0$ and the child scales are
$L_0,h_0,L_0h_0$.  In particular, with the original receiver and
release formulas at these scales,
\[
 \|H_0\|_\infty\le\delta L_0b_0^2,\quad
 \|DH_0\|\le CL_0b_0,\quad
 \|D^2H_0\|\le CL_0,\quad
 \|D^3H_0\|\le CL_0/h_0.
\]
The complete upper base has
\begin{equation}\label{waG:eq:upper-base}
 \|D^2f_0\|\le CL_0,\qquad
 \|D_uf_0\|\le CL_0b_0.
\end{equation}
Every higher selected-carrier and frontier estimate in
Lemma~\ref{rec:lem:carrier-jets} continues to hold with the new
$h_0,n_0$, without a factor depending on $n_0$.
\end{lemma}
\begin{proof}
Apply the local construction with curvature ratio $r=1$ and gradient
ratio $\eta_{\rm base}=\xi_0$, rather than the special notation
$1/K$.  Its proof requires this ratio to be sufficiently small,
and $n_0=\lfloor c_n/\xi_0^2\rfloor$; it never uses a lower response
map.  The receiver value estimate becomes
\[
 C[c_n+A_{\rm rec}(\zeta+\xi_0+
                 K_-\xi_0^2/\zeta)]L_0b_0^2,
\]
which is uniformly small after the existing fixed choices and
$\xi_0\le\xi_*$.  The positive receiver turns on at width $h_0$
and stops after the release at width $b_0$.  Its turn-on derivative
cost is $g_0/h_0=L_0$.  The negative receiver has
$T_-/h_0=K_-/\zeta$.  All derivative constants are therefore uniform.
The base common radius is
\[
 R_0^{\rm rad}=\frac{C_R L_0b_0}{\nu_0},\qquad
 \nu_0=\underline w_0\zeta L_0,
\]
and
$R_0^{\rm rad}/(h_0\sqrt{n_0})
 \ge C_R/(\underline w_0\zeta\sqrt{c_n})$.
The radial-composition proof of the local margin and derivative bounds
thus remains valid.  Also
$\nu_0R_0^{\rm rad}=C_RL_0b_0$, so its unselected radial gradient is
bounded by the same scale.

For \eqref{waG:eq:upper-base}, the base payload has value
$O(L_0b_0^2)$, gradient $O(L_0b_0)$, and Hessian $O(L_0)$.
The macro selectors have width $b_0$.  The base-to-first-level
release has value $O(L_0h_0^2)$ and child gate width $h_0$, so its
joint Hessian is also $O(L_0)$.  At most one selected jet is active.
These facts prove every joint and mixed estimate, without
identifying $b_0/h_0$ with $K$.

For a higher carrier, its immediate selector width is $h_{j-1}$.
The ratios
$h_j\sqrt{n_j}/h_{j-1}$ and
$n_j\ell_jh_j^2/h_{j-1}^2$ are exactly the earlier ratios.
The outermost selector width is $b_0\ge h_0$, so replacing the
outer ancestor width cannot make any of these estimates larger.
Tail-mask derivative norms are dimension independent and do not
sum over the $n_0$ coordinates.  The original product-rule proof,
with a fixed factor per depth, therefore applies to all higher
carriers and to the source-only split transport carriers.
\end{proof}

\begin{lemma}[Floors and a uniform accuracy constant]
\label[lemma]{waG:lem:floors}
The numerical accuracy constant in \eqref{waG:eq:public-accuracy} can be
chosen so that, throughout the recursive branch,
\[
 S\ge64,\qquad n_0\ge64,\qquad \xi_0\le\xi_*,
\]
and both floor quantities are comparable to their arguments.
Moreover $\epsilon\le c_1L_0b_0$ for any prescribed sufficiently
small numerical $c_1$, by decreasing the same accuracy constant.
\end{lemma}
\begin{proof}
The macro floor argument is
\[
 \frac{c_S\Delta}{L_0b_0^2}
   =\frac{c_S}{b_*^2}\max\{1,\Delta/L_0\}\ge128.
\]
Since $L_0=\eta_{\rm curv}L$, the ratio
\[
 \frac{\min\{L,\sqrt{L\Delta}\}}{L_0b_0}
\]
is bounded by a numerical constant depending only on the fixed
$\eta_{\rm curv},b_*$.  Consequently
\[
 \xi_0\le Cc\,\frac{\omega^{-J}}{K}.
\]
For every fixed $A>1$ and fixed $r\ge0$,
\begin{equation}\label{waG:eq:polylog-envelope}
 \sup_{\substack{\kappa\ge2\\1\le\vartheta\le\kappa^2}}
 \frac{A^{J(\vartheta)}(J(\vartheta)+1)^r}{\kappa}<\infty .
\end{equation}
Indeed $\log(2+\vartheta)\le C\log(2+\kappa)$ on this range.
The numerator is bounded by a fixed power of $\log(2+\kappa)$,
which divided by $\kappa$ is uniformly bounded.
Apply this fact with $A=\omega^{-1}$ and
\eqref{waG:eq:K}.  Decreasing $c$ makes $\xi_0$ at most both
$\xi_*$ and $\sqrt{c_n/128}$.  This proves the base floor
and its comparison.  The same displayed ratio proves the last
assertion.  No parameter-dependent accuracy constant was used.
\end{proof}

\begin{lemma}[Global margin after decoupling the widths]
\label[lemma]{waG:lem:margin}
The recursive hyperobjective has $F(0)=0$, is coercive, and has gap
at most $CSL_0b_0^2$.  Before the final macro coordinate is complete,
its gradient obeys
\begin{equation}\label{waG:eq:margin}
 \|\nabla F\|\ge
 c\min\{L_0b_0,\ K w_*^{J+1}\eta^J L_0h_0\}.
\end{equation}
Under the preceding choices this is at least $8\epsilon$.
The same statement holds for the endpoint-transport trace.
\end{lemma}
\begin{proof}
The backward payload induction uses only the parent value bounds
and their local width ratios.  Lemma~\ref{waG:lem:base} supplies its
base case, and the higher-level recursion is unchanged.  The gap
therefore remains $CSL_0b_0^2$.
The macro alternative in the global-margin proof gives
$cL_0b_0$.  If that alternative does not hold, the same descent
through the eligible path either stops at level $j$ with an amplified
margin $cK w_*^{j+1}g_j$, or reaches the deepest level.
The minimum of these quantities is the second term in
\eqref{waG:eq:margin}.  This argument uses the actual weights while
a child selector is active, exactly as in the earlier weighted proof.
Substituting $h_0$ gives
$K w_*^{J+1}\eta^J L_0h_0=C_bw_*\epsilon$.
Fix $C_b$ first and use Lemma~\ref{waG:lem:floors} for the macro
term.  Both terms can thus exceed the desired constant times
$\epsilon$.
The transport graph identifies its endpoint readouts with the same
$Kv_J$ readout and preserves the hyperobjective as a function of
the outer variables.  Its gap and margin are therefore identical.
\end{proof}

\subsection{The complete extension ledger at uniform accuracy}

\begin{lemma}[Uniform extension and primitive budgets]
\label[lemma]{waG:lem:budgets}
For the preceding parameters, the active and inactive dual extensions
have $\|D^2E\|,\|D^2E^0\|\le\delta_E$, and the mean extension has
\[
 m_0\|D^3E\|\le \rho/C_{\rm Leg}.
\]
The resulting ordinary and hybrid transport primitives satisfy
$C_f,L_f,L_g\le L$, lower modulus at least $\mu$, and the prescribed
Hessian-block variation.  Their exact responses and initial values
are the same graph identities as before.
\end{lemma}
\begin{proof}
Write $D_0=J+1$, $\lambda=L_0/d$, and retain the exact fine parameter
\begin{equation}\label{waG:eq:Gamma}
 \Gamma=C_{\rm fine}^{D_0}D_0(m_0/\mu)\lambda,\qquad
 v=\frac{\vartheta}{\Gamma a_J}.
\end{equation}
The choice of $C_T$ makes $v\ge1$.
For $j\ge1$, the coarse terms cost
\[
 C_g^{D_0}\lambda\,\eta^j h_0 n_j^{2/3}
       \quad\hbox{in }D^2,\qquad
 C_g^{D_0}\lambda\,a_j n_j^{1/3}
       \quad\hbox{in }D^3.
 \tag{W1}
\]
The fine sum costs
\[
 C C_g^{D_0}\lambda\,\eta^Jh_0\sqrt v
       \quad\hbox{in }D^2,\qquad
 C C_g^{D_0}\lambda\,a_Jv
       \quad\hbox{in }D^3.
 \tag{W2}
\]
The public radial sums cost at most
\[
 C\lambda h_0\sum_{j=1}^J\eta^{j-1}\sqrt{\varrho_j}
       \quad\hbox{in }D^2,\qquad
 C\lambda\sum_{j=1}^J a_{j-1}
       \quad\hbox{in }D^3 .
 \tag{W3}
\]
These are the original full joint estimates.  They involve neither
the macro width $b_0$ nor the base length $n_0$.

The identity $a_j\varrho_j^{1/3}=T$ and the choice of
$T=\vartheta/(C_T^{D_0}D_0^p)$ allocate the coarse and radial
third derivatives.  Equation \eqref{waG:eq:Gamma} allocates the
fine third derivatives; increasing $C_{\rm fine}$ relative to
$C_g$ leaves fixed slack.  These choices are independent of
$\epsilon$.

For the Hessians, $\varrho_j=T^{3\cdot4^{-j}}$ and therefore
\[
 \sum_{j=1}^J\eta^j n_j^{2/3}\le C T^{1/2},\qquad
 \sum_{j=1}^J\eta^{j-1}\sqrt{\varrho_j}\le C T^{3/8}.
\]
Substitute $h_0$ and $d\asymp L/\sqrt{D_0}$.  The coarse Hessian
sum is bounded by
\[
 C\frac{\epsilon}{L}\frac{\sqrt\vartheta}{K}\,
 \frac{\sqrt{D_0}}{D_0^{p/2}}\,
 C_g^{D_0}\omega^{-J}C_T^{-D_0/2}
 \le C\frac{\epsilon}{L}\frac{\sqrt\vartheta}{K}.
 \tag{W4}
\]
Likewise the radial Hessian sum is at most
\[
 C\frac{\epsilon}{L}\frac{\vartheta^{3/8}}{K}\,
 D_0^{1/2-3p/8}\omega^{-J}C_T^{-3D_0/8}
 \le C\frac{\epsilon}{L}\frac{\vartheta^{3/8}}K .
 \tag{W5}
\]
Here \eqref{waG:eq:T-choice} absorbs the indicated depth factors.
For the fine sum use $a_J\ge1$ and
$v\le\vartheta/\Gamma$:
\[
 C C_g^{D_0}\lambda\eta^Jh_0\sqrt v
 \le
 C\frac{\epsilon\sqrt\vartheta}{K\sqrt{dL_0D_0}}\,
 (C_g/\sqrt{C_{\rm fine}})^{D_0}w_*^{-J}
 \le C\frac{\epsilon}{L}\frac{\sqrt\vartheta}{K}.
 \tag{W6}
\]
The last inequality uses \eqref{waG:eq:fine-choice} and
$L_0=\eta_{\rm curv}L$; its constant is numerical.
For $\vartheta\le\kappa^2$ and $K\ge c_K\kappa$,
\[
 \frac{\sqrt\vartheta}{K}\le c_K^{-1},\qquad
 \frac{\vartheta^{3/8}}{K}\le c_K^{-1}\kappa^{-1/4}.
\]
Thus all three right sides are bounded by $C\epsilon/L$.
The explicit depth cancellations in (W4)--(W6) are essential here;
replacing them first by an unspecified logarithmic loss would not
justify a uniform constant at $\vartheta=\kappa^2$.
The inactive replacement adds a Hessian bounded by
$C C_g^{D_0}\lambda\eta^Jh_0$, which is no larger than the
fine bound because $v\ge1$.  Thus \eqref{waG:eq:public-accuracy},
with a sufficiently small fixed $c$, makes both full Hessians
at most $\delta_E$.

The Legendre regularity lemma gives the exact lower budgets with
the previous quadratic slack.  The hybrid lift adds the same
explicit $H-m_0I$ quadratic; its matrix bounds and response are
unchanged.  Its endpoint readout norms and physical radial scale
change (W1)--(W6) by numerical constants only, already included
in $C_g$.

For the upper primitive, \eqref{waG:eq:upper-base} gives
$L_f\le L$ after fixing $\eta_{\rm curv}$.
It also gives a geometric lower-gradient bound $CL_0b_0\le L/2$
after fixing $b_*$.  The $2J$ normal coefficients still have joint
norm $d\sqrt{2J}=L/4$.  Hence $C_f\le L$ globally.
No derivative of $f_x$ is subjected to a global gradient-norm
constraint by the primitive class.
The graph identities use zero dual normal coordinates, not a
linear-core condition at nonzero primal normals, so increasing
the permissible accuracy introduces no further graph restriction.
\end{proof}

\subsection{Blind widths, the full signal, and information}

The fine blind width controls a source that remains undisclosed.
A source disclosed in the current call only needs to stay in its
predecessor-flat interval of width $h_J/4$.
Lemma~\ref{waI:lem:uniform-locality} proves the resulting
restricted-inverse coupling for these parameters.
The unchanged first-derivative ledger gives the finite global envelope
\[
 H_{\rm hid}\le C C_g^{D_0}(L_0/d)h_0^2.
\]
The finest blind threshold satisfies
\begin{equation}\label{waG:eq:tau}
 \tau\asymp h_J/v=\Gamma\eta^Jh_0/\vartheta,
 \qquad \tau\le h_J/1024.
\end{equation}
All earlier local thresholds are larger. The new inverse proof uses
$H_{\rm hid}$ only to select the finite projection dimensions; it
does not require $H_{\rm hid}/\tau$ to be small.
Its source-force condition is
\begin{equation}\label{waG:eq:source-budget}
 C_{\rm src}C_g^{D_0}\frac{g_J}{d\sqrt v}\le\frac1{128}.
\end{equation}
Since $g_J=C_b\epsilon/(K w_*^J)$, $v\ge1$, and
$d=L/(4\sqrt{2J})$, its left side is at most
\[
 C\frac{\epsilon}{L}
 \frac{C_g^{D_0}\sqrt J\,w_*^{-J}}{\kappa}.
\]
Equation~\eqref{waG:eq:polylog-envelope} bounds the last factor
uniformly on $\vartheta\le\kappa^2$. Reducing the same numerical
accuracy constant therefore pays the source budget. No condition
depending on $\epsilon\vartheta/(\kappa L)$ is used.

Flatten the deepest source at $\tau$ and use the same globally defined
active and inactive conjugates.  Their full-gradient difference has
the envelope
\begin{align}
 B
 &\le C C_g^{D_0}m_0\lambda\,\eta^Jh_0\tau \notag\\
 &\le C_B^{D_0}D_0^2
              \frac{\epsilon^2}{L\vartheta\kappa^3}.
 \label{waG:eq:signal}
\end{align}
Indeed substitution of
$\Gamma=C_{\rm fine}^{D_0}D_0(m_0/\mu)(L_0/d)$
and $h_0$ cancels $\eta^{2J}$ and $L_0^2$; the remaining
$m_0^2/(\mu d^2K^2)$ is at most $CD_0/(L\kappa^3)$.
All fixed powers of $w_*^{-1},C_g,C_{\rm fine}$ enter only
$C_B^{D_0}$.  In particular the envelope has no factor $n_0$.
The same Bernoulli potential with
$p=B^2/(\sigma_g^2+B^2)$ gives a globally unbiased conservative
SFO with full-gradient variance at most $\sigma_g^2$.

Theorem~\ref{waI:thm:information-count} shows that the ordinary recursive observer charges
$\mathcal N=S\prod_{j=0}^Jn_j$ successful activations.
For the exact deterministic transport family, the frozen one-hop
observer charges $\Omega(n\mathcal N)$ calls with
$n=\Theta(\sqrt\kappa)$.
The nonlinear support and restricted-inverse certificates are
Lemma~\ref{waI:lem:uniform-locality}; the transport angle and
arbitrary-query accounting are unchanged.
The change of $n_0$ only changes the finite number of columns and
the dimensions chosen by these observers.  It introduces no
transcript-dependent function or extra oracle access.

\subsection{Rate calculation and the bounded parameter branches}

By the floor comparisons,
\[
 Sn_0\ge c\frac{\Delta}{L_0h_0^2},\qquad
 \prod_{j=1}^Jn_j
 \ge(c_n/(2\eta^2))^Ja_J.
\]
Consequently
\begin{equation}\label{waG:eq:count}
 \begin{split}
 \mathcal N
 &\ge c a_J K^2L_0\Delta\epsilon^{-2}
                         (c_nw_*^2/2)^J\\
 &\ge
 \frac{c}{C_N^{D_0}D_0^p}\,
               L\vartheta\kappa^2\Delta\epsilon^{-2}.
 \end{split}
\end{equation}
The $\eta^{-2J}$ from the child lengths cancels the $\eta^{2J}$
in $h_0^{-2}$.  The macro scale $b_0$ cancels between $S$ and
$n_0$.  This is why fixing a small macro width does not lose a
power of $\epsilon$, $\Delta$, or $\kappa$.
Here $a_J\ge e^{-4}T$ and $K\asymp\kappa$ were used.

Equations \eqref{waG:eq:signal}--\eqref{waG:eq:count}, together
with Theorem~\ref{waI:thm:information-count}, give
\[
 \frac{c}{C^{D_0}D_0^r}
 \left(L\vartheta\kappa^2\Delta\epsilon^{-2}
       +L^3\vartheta^3\kappa^8\sigma_g^2
                                      \Delta\epsilon^{-6}\right)
\]
for a fixed $r$.  The separate exact transport family gives
\[
 \frac{c}{C^{D_0}D_0^r}
                L\vartheta\kappa^{5/2}\Delta\epsilon^{-2}.
\]
All the displayed losses are fixed powers of
$\log(2+\vartheta)$.

For $1\le\vartheta\le\Theta_*$ set
$\vartheta_b=\min\{\vartheta,\kappa\}$ and apply the all-gap base
lower theorem with $\vartheta_b$. This belongs to its proved range,
and its Hessian-variation budget $\mu\vartheta_b$ is at most the
prescribed $\mu\vartheta$. Its accuracy range contains
\eqref{waG:eq:public-accuracy}. Its lower-noise coefficient
$\vartheta_b^2$ is within a numerical factor of the desired
$\vartheta^2(1+\vartheta)$, because
\[
 \frac{\vartheta^2(1+\vartheta)}{\vartheta_b^2}
 \le\Theta_*^2(1+\Theta_*).
\]
The exact-quadratic theorem \ref{tr:thm:quadratic} supplies the
deterministic term, with its missing $1+\vartheta$ factor bounded
by $1+\Theta_*$. If $\kappa<\kappa_{\rm tr}$, then
$\vartheta\le\kappa^2<\kappa_{\rm tr}^2\le\Theta_*$, so this
same branch covers the unavailable transport normalization.
Thus there is no missing bounded-condition or bounded-variation branch.

The independent upper-noise lower theorem
\ref{uf:thm:lower} already has the full range
\eqref{waG:eq:public-accuracy}, with no logarithmic loss.

\begin{proof}[Proof of Theorem~\ref{waG:thm:uniform-matching}]
The preceding geometry, exact primitive budgets, full informative
envelope, and companion locality lemma verify the structural
prerequisites used in the ordinary and transport observer proofs.
Lemma~\ref{waI:lem:uniform-locality} reruns their inverse and support
invariants, and Theorem~\ref{waI:thm:information-count} proves the full
query-count conclusions with the new parameters. We do not invoke an earlier
theorem whose statement assumes $h_0=b/K$ or the earlier sufficient
accuracy restriction.
The global output margin is $8\epsilon$ before the final macro
column is disclosed.  The same predetermined-horizon dimension
choices, terminal-output insertions, and expected-cost truncations
therefore apply.  They give the two lower bounds just computed,
including the bounded parameter bridges.
Combine these two witnesses and the independent upper-noise witness
by the minimax maximum, which is at least one third of their sum.
Since $\vartheta$ and $1+\vartheta$ are within a factor two, this is
the lower side of \eqref{waG:eq:comparison}.

For the upper side, substitute the normalized primitive bounds into
Theorem~\ref{ac:thm:cached-upper}. The inequalities
$\epsilon^2\le c^2L^2$ and $\epsilon^2\le c^2L\Delta$
absorb the outer floors and the additional fourth-power lower-noise
term into the three displayed terms, uniformly in both variances.
\end{proof}

\section{Uniform accuracy in the recursive information certificate}
\label{waI:sec:information}

The information proof separates the finest inactive-source width from
the wider predecessor-flat interval. Still-undisclosed columns are
controlled by their projections on the restricted inverse spaces;
a simultaneous homotopy controls columns disclosed during the current
call. The total hidden displacement is only required to be finite.
The mean and inactive potentials and their signal envelope remain fixed,
and the independently chosen base length affects only the dimensions.

\subsection{Scales and an order for the constants}

Fix the geometric profile constants, including $0<\eta,w_*<1$ and
$c_n>0$, before choosing any constants below.  Let $J\ge1$, $D=J+1$,
$L_0=\eta_0L$, $m=3\mu/2$, $\mu=L/\kappa$, and
\[
 d=\frac{L}{4\sqrt{2J}},\qquad
 c_K\kappa\le K\le C_K\kappa,\qquad 1\le\vartheta\le\kappa^2.
\]
The positive numerical $c_K,C_K$ cover both the isotropic amplification
and the normalized endpoint-transport amplification.  For a sufficiently
large fixed $C_T$ and a fixed integer $p_T\ge6$, put
\begin{equation}\label{waI:eq:scales}
 \begin{gathered}
 J=\left\lfloor\log_4\log(2+\vartheta)\right\rfloor,\qquad
 T=\frac{\vartheta}{C_T^D D^{p_T}},\qquad T\ge2,\\
 a_0=1,\quad a_j=T^{1-4^{-j}},\quad
 \varrho_j=a_j/a_{j-1},\quad
 \ell_j=L_0a_j,\quad h_j=\frac{\eta^jh_0}{a_j},\\
 n_j=\left\lfloor c_n\varrho_j/\eta^2\right\rfloor
 \quad(1\le j\le J),\qquad W_j=h_jn_j^{1/3}.
 \end{gathered}
\end{equation}
The finitely bounded parameter branch in which $J=0$ or $T<2$ uses
the existing base construction; it is not part of this recursive lemma.
All physical high-level floors are assumed at least one, as ensured
by the already fixed choice $c_n/\eta^2\ge128$.

The macro width is denoted $b_0$.  The base middle chain may have any
length $n_0\ge1$ satisfying
\begin{equation}\label{waI:eq:free-middle}
 h_0\le b_0,\qquad n_0h_0^2\le c_n b_0^2.
\end{equation}
Its complete carrier and public radial term remain in the upper
primitive.  In particular no normal term below carries this
$n_0$-term chain.  The high-level payloads, ancestor selectors,
receivers, release gates, and centered deepest frontier are exactly
those in \cref{rec:lem:carrier-jets}.

We first choose a numerical $C_g>1$ dominating the fixed chain,
selector, wrapper, receiver, and scalar-profile estimates.  It also
dominates the fixed constants in the multiscale estimates proved below.
The constant $C_g$ is chosen independently of $C_{\rm fine}$, $C_T$,
the fine parameter, the depths, and every chain length.  We will choose
$C_{\rm fine}$ after $C_g,w_*,\eta$, and then choose $C_T$ after
$C_{\rm fine}$.  The final accuracy constant is chosen last.

Use the exact fine allocation
\begin{equation}\label{waI:eq:Gamma}
 \Gamma_D=C_{\rm fine}^{D}D\frac m\mu\frac{L_0}{d},
 \qquad v_f=\frac{\vartheta}{\Gamma_Da_J}.
\end{equation}
Since $e^{-4}T\le a_J\le T$ and $L_0/d=4\eta_0\sqrt{2J}$,
choosing $C_T$ sufficiently large ensures
\begin{equation}\label{waI:eq:v-bounds}
 1\le v_f
 \le \frac{e^4 C_T^D D^{p_T}}{\Gamma_D}
 \le A_v^D D^{p_v}
\end{equation}
for fixed numerical $A_v,p_v$.  Neither choice depends on
$\epsilon$, $\Delta$, or $n_0$.

\subsection{A hidden-displacement bound with constants independent of
the fine allocation}

Write $p=(p_0,\ldots,p_J)$ for the dual vector blocks.  The mean hidden
extension is
\begin{equation}\label{waI:eq:Eh}
 E_h=\sum_{j=1}^{J-1}\frac{r_j(x,p)}d q_{W_j}(q_j)
       +\frac{A_J(x,p)}d q_{W_J}(q_J)
       +\frac1d\sum_{\ell=0}^{K_f}e_\ell(x,p)q_{w_\ell}(q_J).
\end{equation}
Here $q_w(s)=wq(s/w)$ for the one fixed compact identity profile,
$R_{\rm fr}=\sum_{\ell=0}^{K_f}e_\ell$, and the decomposition is
the one in \cref{chi:lem:fine} with parameter $v_f$.
Let
\begin{equation}\label{waI:eq:inactive}
 \tau=\frac{w_{K_f}}8,\quad
 \Delta R=R_{\rm fr}-R_{{\rm fr},\tau},\quad
 E_h^0=E_h-\frac{\Delta R}{d}q_{w_{K_f}}(q_J).
\end{equation}
This is the same fixed inactive potential as before.  It does not
depend on the transcript or on a currently revealed prefix.

\begin{lemma}[Uniform bounded hidden displacement]
\label[lemma]{waI:lem:hidden}
For the functions \eqref{waI:eq:Eh}--\eqref{waI:eq:inactive}, there is
a choice of the preceding $C_g$, independent of $v_f$ and
$C_{\rm fine}$, such that
\begin{equation}\label{waI:eq:H}
 \max\{\|D_{p,q}E_h\|_\infty,\|D_{p,q}E_h^0\|_\infty\}
       \le C_g^D\frac{L_0}{d}h_0^2=:H.
\end{equation}
The same bound holds with the macro derivative included whenever
\eqref{waI:eq:free-middle} holds.  In particular it bounds each
individual tangential block $\|(E_h)_{p_j}\|$ and
$\|(E_h^0)_{p_j}\|$.  The bound holds for every value of all normal
coordinates and has no dependence on $n_0$.
\end{lemma}
\begin{proof}
The selected-carrier ledger gives the bounds
\[
 \|r_j\|_\infty\le C_g^D n_j\ell_jh_j^2,\qquad
 \|D r_j\|_\infty\le C_g^D\ell_jh_j\sqrt{n_j},
\]
also for $A_J$.  For a coarse term, the tangential derivative costs
at most
$C_g^D(\ell_j/d)h_j^2n_j^{5/6}$, because
$\|q_{W_j}\|_\infty\le C W_j$.
Its normal derivative costs at most
$C_g^D(\ell_j/d)n_jh_j^2$, because
$\|q_{W_j}'\|_\infty\le C$.
The inequality $n_j^{5/6}\le n_j$ is valid as $n_j\ge1$.
Moreover
\begin{equation}\label{waI:eq:geometric-value-sum}
 n_ja_jh_j^2
 \le c_n\eta^{2j-2}\frac{h_0^2}{a_{j-1}},\qquad
 \sum_{j=1}^J n_ja_jh_j^2
 \le \frac{c_n}{1-\eta^2}h_0^2.
\end{equation}
This bounds the sum of every coarse first derivative.

For completeness, let $v=v_f$ and use the explicit fine scales
\[
 K_f=\lfloor\log_4\max\{1,\log v\}\rfloor,\quad
 \gamma_\ell=1-4^{-(\ell+1)},\quad
 \lambda_\ell=1-\tfrac12 4^{-\ell},
\]
\[
 \delta_\ell=c_{\rm ms}h_Jv^{-\gamma_\ell}\ (\ell<K_f),\qquad
 w_\ell=c_{\rm ms}h_Jv^{-\lambda_\ell},
 \qquad 0<c_{\rm ms}\le1/128.
\]
If $K_f=0$ there is one atom.  Otherwise, successive $w_\ell$ have
ratio at most $e^{-3/2}$, and successive $\delta_\ell$ have ratio at
most $e^{-3}$: these follow by inserting their exponent differences
and $4^{K_f}\le\log v$.  Consequently
\[
 \sum_{\ell=0}^{K_f}w_\ell\le C h_J,\qquad
 \sum_{\ell=0}^{K_f-1}\delta_\ell\le C h_J.
\]
The fine construction has
\[
 \|e_0\|_\infty\le C_g^D\ell_Jh_J^2,\qquad
 \|e_\ell\|_\infty\le C_g^D\ell_Jh_J\delta_{\ell-1}
 \quad(\ell\ge1),\qquad
 \|D e_\ell\|_\infty\le C_g^D\ell_Jh_J.
\]
Thus summing $|q_{w_\ell}|\|D e_\ell\|/d$ and
$|q_{w_\ell}'|\|e_\ell\|/d$ gives at most
$C_g^D(\ell_J/d)h_J^2$.  This argument uses no constant depending
on $v$, no factor $K_f$, and no unspecified overlap assertion.
It is valid for all $q_J$.

Finally the inactive change has
\[
 \|\Delta R\|_\infty\le C_g^D\ell_Jh_J\tau,\qquad
 \|D\Delta R\|_\infty\le C_g^D\ell_Jh_J.
\]
Since $w_{K_f}=8\tau\le h_J/128$, its full first derivative costs
at most $C_g^D(\ell_J/d)h_J\tau$ and is absorbed by the same sum.
This proves \eqref{waI:eq:H}, increasing the fixed $C_g$ if necessary.

Only high-level carriers were used.  Their ancestor weights have
one nonzero jet per level; their derivative scales are $h_k^{-1}$
on vector ancestors and $b_0^{-1}$ on the macro ancestor.
The inequalities $h_j\le h_k$ for $k<j$ and $h_j\le h_0\le b_0$
give the same joint first-derivative ledger.  The number of base
middle coordinates does not occur in a selector's operator-norm
bound.  This explains both the macro assertion and the claimed
independence from $n_0$.
\end{proof}

\subsection{Fine widths and the final margin calibration}

The calibration is
\begin{equation}\label{waI:eq:h0-calibration}
 h_0=\frac{C_b\epsilon}{K L_0(w_*\eta)^J},
 \qquad 0<\epsilon\le cL.
\end{equation}
For the explicit fine scales, including the case $K_f=0$,
\begin{equation}\label{waI:eq:tau-exact}
 \frac{c_{\rm ms}}8\frac{h_J}{v_f}
 \le\tau\le\frac{e^2c_{\rm ms}}8\frac{h_J}{v_f},
 \qquad \tau\le h_J/1024.
\end{equation}
Indeed $v_f^{1-\lambda_{K_f}}
=\exp((\log v_f)/(2\cdot4^{K_f}))$ lies between $1$ and $e^2$.
Since $h_J/v_f=\Gamma_D\eta^Jh_0/\vartheta$, these constants
are independent of the fine allocation and depth.
Set
\[
 \tau_j^{\rm loc}=h_j/4\quad(0\le j<J),\qquad
 \tau_J^{\rm loc}=\tau.
\]
Their minimum is $\tau$. We use these narrow thresholds for columns
still undisclosed after a call; newly disclosed columns will instead
be controlled at the wider thresholds $h_j/4$.

\subsection{Joint near-identity and third-derivative estimates}

\begin{lemma}[Uniform joint extension budgets]
\label[lemma]{waI:lem:extension}
Under \eqref{waI:eq:free-middle}, the constants can be chosen in the
stated order so that \eqref{waI:eq:h0-calibration}, with a sufficiently
small numerical $c$, gives
\[
 \|D^2E\|_\infty,\ \|D^2E^0\|_\infty\le\delta_E,\qquad
 m\|D^3E\|_\infty\le\rho/C_{\rm Leg}.
\]
All norms include macro, tangent, and normal variables. The public
radial coefficient in each tangent equation has magnitude at most
$1/8$. The same statements hold for the full-vector endpoint lift.
\end{lemma}
\begin{proof}
The full joint ledger in Lemma~\ref{waG:lem:budgets} applies to
these identical potentials. That proof uses the base variables only
through \eqref{waI:eq:free-middle}, so no equality $h_0=b_0/K$ is
required. In particular its estimates (W4)--(W6), with their explicit
depth allocations retained, give respectively
\[
 C\frac{\epsilon}{L}\frac{\sqrt\vartheta}{K},\qquad
 C\frac{\epsilon}{L}\frac{\vartheta^{3/8}}{K},\qquad
 C\frac{\epsilon}{L}\frac{\sqrt\vartheta}{K}.
\]
They are bounded by $C\epsilon/L$ on $\vartheta\le\kappa^2$.
The inactive change costs at most
$C C_g^D(L_0/d)\eta^Jh_0$, already absorbed by the fine estimate
since $v_f\ge1$. The third-derivative allocation uses
$a_j\varrho_j^{1/3}=T$ and
$v_f=\vartheta/(\Gamma_Da_J)$; it is independent of accuracy.

For every normal coordinate the public radial part has
\[
 (E_{r,j})_{p_j}=\gamma_{r,j}p_j,\qquad
 \gamma_{r,j}=\frac{\nu_jR_j}{d}
 \frac{t_j}{\sqrt{R_j^2+\|p_j\|^2+t_j^2}},\qquad
 |\gamma_{r,j}|\le\nu_jR_j/d.
\]
Its bound is controlled by the radial Hessian sum. The base radial
term is in $f$ and gives no dual perturbation. The fixed endpoint
readout norm and $\gamma_n$ scaling change only numerical constants.
Choosing $c$ last proves all assertions.
\end{proof}

\subsection{Prefix locality without a small total displacement}
\label{ni:sec:locality}

This section concerns the recursive branch $J\ge1$. The exact-quadratic
branch has no implicit nonlinear inverse and retains its existing proof.
The very small inactive-source width and the predecessor-gate width
serve different purposes.  A column that remains undisclosed is
controlled directly by its projection on the restricted inverse space.
A source column disclosed in the current call need only remain in the
wider predecessor-flat interval.  The following argument keeps this
distinction and permits the total hidden displacement to affect the
ambient dimension without restricting the accuracy.

\subsubsection{Radial decomposition and a finite-dimensional observation}

For one physical vector block write
\[
 B_R(p)=\alpha(p)p,\qquad
 \alpha(p)=(1+\|p\|^2/R^2)^{-1/2},\qquad z=V^\top B_R(p),
\]
where $V$ is the full orthonormal hard frame in that block.  All hidden
terms in the construction depend on this block through $z$.  The public
normal extension depends on $p$ only through $\|p\|^2$.  Thus the full
dual perturbation has the exact decomposition
\begin{equation}\label{ni:eq:radial-decomposition}
 D_pE=\alpha Va+\gamma p,
 \qquad a=D_zE_h,
\end{equation}
where $a$ denotes the derivative in latent coordinates, before the
radial wrapper. In this subsection $D=J+1$ denotes depth.  In particular the hidden wrapper contribution to
$\gamma$ must be included; it is not generally zero.

\begin{lemma}[The total radial coefficient]
\label[lemma]{ni:lem:radial}
Suppose $\|D^2E\|\le\delta<1$ globally and the ambient block dimension
is at least $\operatorname{rank}(V)+2$. Then $|\gamma|\le\delta$
in \eqref{ni:eq:radial-decomposition}, at every macro, vector, and
normal argument.
\end{lemma}
\begin{proof}
Choose a unit vector $e$ orthogonal to both $V$ and $p$.
Differentiating the scalar function $E$ twice in this direction gives
$D^2_{pp}E[e,e]=\gamma$: the first derivatives of the framed
coordinates and of the radial norm vanish in that direction, and
the second derivative of the radial part is exactly its coefficient
in \eqref{ni:eq:radial-decomposition}. Equivalently, differentiate
that displayed gradient and take its inner product with $e$.
The Hessian bound proves the claim.
\end{proof}

The requisite ambient dimensions are freely available in the existing
projection observer.  This is a bound on the total radial coefficient,
including all selected carriers and their ancestor weights, so no sum
over known columns or completed subtrees appears.

Write $E=E_h+E_r$, where $E_r$ is the common public radial extension.
Suppose, as in the preceding lemmas, that
\begin{equation}\label{ni:eq:standing-budgets}
 \|D^2E\|\le\delta\le\tfrac14,\quad
 \|(E_h)_p\|\le H<\infty,\quad
 (E_r)_p=\gamma_r p,\quad |\gamma_r|\le\tfrac14.
\end{equation}
The bound $H$ is allowed to be large. The same assumptions will be used
for $E^0$, with one common $H$. It is enough that these bounds hold on
all lower variables and all normal queries, at the queried macro point;
Lemma~\ref{waI:lem:extension} supplies global bounds.

Let $W^-$ be the known space after inserting the actual vector query $u$
but before revealing any column in this call. Let $v_s$ be a column
unknown at that time, and suppose $\|P_{W^-}v_s\|\le\sqrt e$.
For a restricted homotopy inverse, with the original query fixed, write
\[
 z_t=\argmin_{z\in\mathcal W}
 \{\tfrac12\|z\|^2+tE(x,z)-\langle y,z\rangle\},
 \qquad 0\le t\le1,
\]
where $\mathcal W$ is the direct sum of the fixed enlarged tangent
spaces and all normal coordinates. If $p_t$ is its block in the
space $W$, the public-radial form gives
\begin{equation}\label{ni:eq:query-norm}
 \|u\|\le\tfrac54\|p_t\|+H.
\end{equation}
Indeed $u=(1+t\gamma_r)p_t+tP_W(E_h)_p$.
Since $u\in W^-$ and $\alpha\|p_t\|\le R$, this implies
\begin{equation}\label{ni:eq:old-query-component}
 \alpha(p_t)|\langle v_s,u\rangle|
 \le\sqrt e\{\tfrac54R+H\}
 \le\tfrac54\sqrt e(R+H).
\end{equation}
This controls arbitrarily large query norms, including cancellation
of a large known-direction query by the hidden gradient.  It is the
step for which a bound using only $R\sqrt e$ would be insufficient.

\subsubsection{The exact source coefficient}

We use the definitions immediately preceding
\cref{rec:lem:carrier-jets}: for $j<J$,
\[
 r_j=\sum_I w_I\left[A_I+\sum_k R_{I,k}Q_{I,k}\right],\qquad
 A_I=H_I-\sum_kR_{I,k},
\]
and $A_J=r_J-R_{\rm fr}$, with the centered deepest frontier specified
there.  The signed incoming-edge identity in the recursive construction
shows that $A_I$ has no incoming-source derivative on a blind tail.
The added centering compensation has only known-predecessor derivatives.

Call a column \emph{new} if it is disclosed on the present successful
coin, including the parent columns released when a subtree completes.
The following support facts hold whenever every new level-$j$ column
has latent coordinate of magnitude at most $h_j/4$, and all still
undisclosed columns lie in their prescribed flat intervals.
\begin{enumerate}
\item A new ancestor column has no nonradial hidden derivative. Its
closed carrier cancels the incoming source; its own tail-mask derivative
is zero. A term $R_{I,k}Q_{I,k}$ cannot restore that derivative: if
this parent column was just released, the corresponding child terminal
is also new and is still in its flat release interval. Future child
selectors have zero value and derivative. Old known coordinates are
unrestricted throughout this argument.
\item At the deepest level the only possible new nonradial derivative
is the eligible source. The coarse $A_J$ has none. All its nonradial
force therefore comes from the fine normal sum on $q_J$.
\item For $E^0$, when the source is still undisclosed, the fine source
is flat on $|z_s|\le\tau$. When it has been disclosed on an active
call, it may leave that narrow interval but must stay in $|z_s|<h_J/4$;
its nonradial coefficient is bounded exactly as for the mean.
\end{enumerate}
These are consequences of the explicit formulas, not a statement that
all hidden gradients are small. In particular, the nonradial derivative
of a previously known column may be arbitrarily larger than the fine
blind width.

\begin{lemma}[Fine source force on the wider blind interval]
\label[lemma]{ni:lem:source}
For the fine normal sum in \cref{chi:lem:fine}, and its fixed inactive
last-atom replacement, the eligible source coefficient satisfies
\begin{equation}\label{ni:eq:source-force}
 |a_s|\le C_{\rm src}C_g^D\frac{g_Jh_J}{d\sqrt{v_f}},
 \qquad g_J=\ell_Jh_J.
\end{equation}
Here $D=J+1$, $C_g$ is the fixed selected-carrier ledger constant, and
$C_{\rm src}$ depends only on the fixed scalar profiles. There is no
dependence on the stage count or ambient dimensions beyond the displayed
ledger factor. The estimate
holds for every normal query.
\end{lemma}
\begin{proof}
On the wider source interval the relevant tail mask is flat in its
source coordinate. Its predecessor and all ancestor selectors do not
depend on that coordinate, and their product lies in $[0,1]$.
A latent frontier summand is
\[
 -w_I\ell_Jh_J^2\psi_{I,k}
       [\Phi(z_s/h_J)-\Phi(0)].
\]
Differentiating its incoming profile, or any of its flattened incoming
profiles $\Phi(c_\delta(z_s)/h_J)$, costs at most $C g_J$.
Indeed $c_\delta'$ is bounded by a fixed profile constant. At most one
frontier selector contributes. The source coefficient of each fine atom,
and of the inactive difference, is consequently bounded by $C g_J$.
These source partial estimates are in particular bounded by the
conservative selected-frontier bound $C_g^D g_J$; we retain that
inherited ledger factor throughout.

The explicit normal widths satisfy
\[
 w_0=c_{\rm ms}h_Jv_f^{-1/2},\qquad
 \sum_{l=0}^{K_f}w_l\le Cw_0.
\]
For $K_f\ge1$ this follows from $w_{l+1}/w_l\le e^{-3/2}$;
for $K_f=0$ it is immediate. Since $|q_w(s)|\le Cw$ for every $s$,
summing the fine source coefficients gives
$C C_g^D g_J\sum_lw_l/d$. The inactive change adds one more term with the
smallest normal width, and is absorbed by the same sum.
\end{proof}

\subsubsection{A homotopy proof of actual restricted-inverse locality}

\begin{lemma}[No cascade with a dimension-dependent projection target]
\label[lemma]{ni:lem:no-cascade}
Use the same fixed mean and inactive potentials, source profiles, and
prefix protocol as above. Assume the global budgets
\eqref{ni:eq:standing-budgets} for both potentials and the source budget
\begin{equation}\label{ni:eq:source-budget}
 C_{\rm src}C_g^D\frac{g_J}{d\sqrt{v_f}}\le\frac1{128}.
\end{equation}
For each frame choose the projection event, both before and after the
legal full-column revelations, so that
\begin{equation}\label{ni:eq:projection-target}
 \sqrt{e_j}\le
 \min\left\{\frac{\tau_j^{\rm loc}}{128R_j},
              \frac{h_j}{128(R_j+H)}\right\}.
\end{equation}
There is no requirement $H\ll\tau_j^{\rm loc}$.
Then the same restricted-inverse protocol returns the actual sampled
lower gradients and the exact upper gradients on this projection event,
with at most one deepest disclosure per successful coin. The conclusion
holds for queries of arbitrary norm and arbitrary normal coordinates.
\end{lemma}
\begin{proof}
Fix a query, its coin, and the legal columns revealed on that coin.
The known tangent spaces $W_j$ after that revelation are fixed for the
following analytic argument. On an inactive call compute only the
inverse for $E^0$. On an active call compute both required inverses
for $E$ and $E^0$, exactly as in the branchwise sample protocol.

For either required potential, let $z_t$ be the unique minimizer on
these tangent spaces and all normal coordinates of
\[
 \tfrac12\|z\|^2+tE(x,z)-\langle y,z\rangle,
 \qquad 0\le t\le1.
\]
Its Hessian is at least $(1-\delta)I$. The inverse therefore exists
uniquely and is continuous in $t$ (indeed differentiable under the
smooth profiles). All vector queries are in the restricted spaces,
so at $t=0$ its tangent blocks equal the original queries. The normal
coordinates also vary along the homotopy; every estimate used below
holds globally in those coordinates.

A column still undisclosed after the coin has
\[
 |\langle v,B_{R_j}(p_{j,t})\rangle|
 \le R_j\|P_{W_j}v\|
 \le \tau_j^{\rm loc}/128
\]
for the whole homotopy, since $p_{j,t}\in W_j$. No hidden-displacement
bound enters this assertion.

Now consider the finitely many new columns. At $t=0$ they are in their
wider intervals. Suppose some new coordinate first reaches magnitude
$h_j/4$. Up to and including that first exit all the exact support
facts above hold. Since each new frame column itself belongs to $W_j$,
the projected tangent equation can be tested against it without any
projection error. Write $\mathfrak a_j$ for the latent-force vector
$a$ of block $j$, to distinguish it from the curvature scale $a_j$.
By \eqref{ni:eq:radial-decomposition},
\[
 (1+t\gamma_j)\langle v_s,p_{j,t}\rangle
 =\langle v_s,u_j\rangle-t\alpha_j(\mathfrak a_j)_s.
\]
For a new ancestor $(\mathfrak a_j)_s=0$. For the deepest eligible source,
\eqref{ni:eq:source-force}--\eqref{ni:eq:source-budget} give
$|(\mathfrak a_J)_s|\le h_J/128$. Multiply by $\alpha_j$ and apply
\eqref{ni:eq:old-query-component}. In either case,
\[
 |\langle v_s,B_{R_j}(p_{j,t})\rangle|
 \le\frac{1}{1-\delta}
       \left[\tfrac54\sqrt{e_j}(R_j+H)+\frac{h_j}{128}\right]
 \le\frac{3h_j}{128}<\frac{h_j}{4}.
\]
(The source-force term is zero for ancestors.) This contradicts first
exit. Thus every new column stays in its wider flat interval throughout
the homotopy. Inactive calls have no new columns and need no such step.

At $t=1$, exact prefix locality shows that the full nonradial derivative
uses only columns in the restricted spaces: old columns on an inactive
call, and old columns plus the eligible source on an active call.
The radial derivatives are parallel to $p_j\in W_j$. Consequently the
restricted minimizer solves the full first-order equation and equals
the unrestricted inverse by uniqueness. The macro point is not inverted;
its pre-coin projection event and prefix identities handle its full
gradient exactly. The upper primitive is evaluated at the original query,
where all previously unknown coordinates have their pre-coin bounds.
Combining the two required inverse gradients with the fixed Bernoulli
coefficients therefore gives the actual sample gradient.

Every restricted problem is measurable from the disclosed frame
projections, including on histories outside the projection event.
Induction with the same frames, coins and seed gives actual--virtual
coupling and identical terminal outputs. No source is revealed on an
inactive call merely to compute an unnecessary mean inverse.
\end{proof}

The smaller targets in \eqref{ni:eq:projection-target} only change finite
ambient dimensions. In the ordinary frame observer replace its previous
$R_j^2/\tau_j^2$ factor by $e_j^{-1}$, and include the same fixed number
of terminal directions. There are still finitely many insertions and
legal full-column revelations at each fixed horizon. No new samples,
no additional successful coins, and no new oracle information are used.
In particular $H$ may depend on every public scale without affecting
the activation count or the informative-gradient envelope.

\subsubsection{Exact transport}

\begin{lemma}[Transport locality without a hidden-displacement threshold]
\label[lemma]{ni:lem:transport}
For the full-vector endpoint construction in
\cref{tr:eq:physical-wrapper,tr:eq:split-chain,tr:eq:split-selectors},
the exact restricted-inverse coupling holds under the global Hessian
and public radial assumptions of \eqref{ni:eq:standing-budgets}, with
finite $H$, after reducing the projection energy as in
\eqref{ni:eq:projection-target}. The source-force condition
\eqref{ni:eq:source-budget} is unnecessary for transport.
\end{lemma}
\begin{proof}
Use the direct sum of the endpoint hard frames as the full nonradial
frame in the expanded physical block; the first-endpoint factor two
is a fixed readout constant. The total radial coefficient is again
bounded by $\delta$ using an orthogonal transverse direction (one may
choose such a direction in an interior physical block). The public
radial coefficient is unchanged up to the fixed $\gamma_n$ scaling.
For the expanded block, use $R_{\rm ph}$ in
\eqref{ni:eq:projection-target}; the first-endpoint readout factor
two is absorbed by decreasing its fixed numerical constant.

For this call let $W_l^-=S_l^t$ be the query-augmented space after
step 1 of the transport observer, before any current-round edge action
or named-column disclosure. Let $W_l^+$ be its enlargement after the
frozen edge actions and named one-hop disclosures in step 2. Include
any parent columns released by these completions in their corresponding
ancestor spaces before solving the current restricted inverse. The
observer supplies frame projections on both sets of spaces. Its energy
event controls the unknown columns at each of these stages, as stated
in Lemma~\ref{dt:lem:observer}. Thus a terminal column newly disclosed
in step 2 has the pre-disclosure bound on $W_l^-$ needed in
\eqref{ni:eq:old-query-component}; columns still undisclosed after
step 2 have the direct projection bound on $W_l^+$. The current
first-endpoint source was already known before this call. The next
first-endpoint source is disclosed only after the current nonlinear
response in step 3 and is available for the following round; it is
not inserted into the current restricted inverse. Solve that inverse
on the direct sum of the $W_l^+$ and the similarly enlarged ancestor
spaces, with all normal coordinates unrestricted.

All predecessor gates, tail masks, terminal receivers and parent
releases depend on last-endpoint $b$ coordinates. The present source
$a$ coordinate at the first endpoint may therefore take any value;
its drift cannot open the next gate. Columns still undisclosed at a
physical layer have the direct projection bound from the restricted
space at that layer. A current column newly transported into an
interior layer has only radial dependence there. When it first becomes
known at the last endpoint, its nonradial coefficient is zero throughout
$|b_s|\le h_J/4$: its predecessor gate and tail-mask derivatives are
flat, its terminal receivers are flat, and its parent release is flat.
Thus the homotopy first-exit argument applies to this new terminal
coordinate with zero source-force term. New ancestor columns satisfy
the same closed-carrier argument as before. Fixed readout factors are
absorbed by decreasing the numerical projection target.

Denote the transport matrix from \eqref{tr:eq:H} by $H_{\rm tr}$
here, to distinguish it from the finite envelope $H$. Its explicit
response $H_{\rm tr}Y$ is outside the near-identity inverse and is
supplied by the one-hop observer. No nonlinear response evaluates
$H_{\rm tr}^{-1}$ at a query.
All edge inputs remain frozen for a round, and the observer's serial
stage schedule is unchanged. Set $m_{\rm out}=1$ and choose its aggregate
energy target below the finitely many adjusted projection targets.
The existing edge-angle and finite-dimension proof then applies without
any relation between the total hidden displacement and a blind width.
\end{proof}

\subsubsection{Uniform calibration on the quadratic range}

\begin{lemma}[Uniform-accuracy recursive inverse locality]
\label[lemma]{waI:lem:uniform-locality}
In the recursive branch $J\ge1$, $T\ge2$, for
$1\le\vartheta\le\kappa^2$, the fixed mean and inactive
potentials above, with \eqref{waI:eq:h0-calibration} and
\eqref{waI:eq:free-middle}, satisfy the locality conclusions of
Lemmas~\ref{ni:lem:no-cascade} and~\ref{ni:lem:transport}
whenever $0<\epsilon\le cL$ for a sufficiently small numerical $c$.
There is no further accuracy restriction depending on conditioning,
variation, or depth. Queries and normal coordinates are unrestricted.
The inactive inverse is exact on every stochastic call, and the mean
inverse is exact on active calls. The transport family uses its exact
mean gradient on every call. The projection dimensions can be fixed
in advance for any prescribed finite horizon.
\end{lemma}
\begin{proof}
Lemmas~\ref{waI:lem:hidden} and~\ref{waI:lem:extension} provide
the finite global envelope and both Hessian budgets, including every
normal query. The only further analytic condition for the ordinary
family is \eqref{ni:eq:source-budget}. Here
\[
 g_J=\frac{C_b\epsilon}{K w_*^J},\qquad
 d=\frac{L}{4\sqrt{2J}},\qquad v_f\ge1.
\]
Consequently its left side is bounded by
\begin{equation}\label{ni:eq:source-simple}
 C\frac{\epsilon}{L}
 \frac{C_g^{J+1}\sqrt J\,w_*^{-J}}{\kappa}.
\end{equation}
The supremum of the last factor is finite on
$1\le\vartheta\le\kappa^2$ by
\eqref{waG:eq:polylog-envelope}. The same final numerical choice of
$c$ makes it at most $1/128$. For transport no source-force condition
is needed. The positive projection targets
\eqref{ni:eq:projection-target} are met by the finite dimensions in
the information theorem below. The two homotopy lemmas apply.
\end{proof}

\subsection{The complete informative gradient}

\begin{lemma}[The signal envelope at uniform accuracy]
\label[lemma]{waI:lem:signal}
Under the same choices there are numerical $C_B>1,p_B$ such that the
original conservative active/inactive oracle has a valid full-gradient
envelope
\begin{equation}\label{waI:eq:B}
 B\le C_B^D D^{p_B}
          \frac{\epsilon^2}{L\vartheta\kappa^3}.
\end{equation}
This includes every macro, vector, and normal gradient component.
There is no dependence on $n_0$ or on $\Delta$.
\end{lemma}
\begin{proof}
The difference is exactly
$E-E^0=(\Delta R/d)q_{8\tau}(q_J)$.
Its joint first derivative, including macro derivatives, is at most
$C_g^D(\ell_J/d)h_J\tau$.
The Hessian smallness established above permits the same conjugate
envelope estimate as \cref{rec:lem:conjugate-envelope}; hence
\[
 B\le C m C_g^D\frac{\ell_J}{d}h_J\tau
 \le C m C_g^D\frac{L_0}{d}
                  \frac{\Gamma_D\eta^{2J}h_0^2}{\vartheta}.
\]
Retaining the allocation factor and substituting the calibration gives
\begin{equation}\label{waI:eq:B-expanded}
 B\le
 C(C_gC_{\rm fine})^D D C_b^2
 \frac{m^2}{\mu}\,
 \frac{\epsilon^2}{d^2K^2\vartheta w_*^{2J}}.
\end{equation}
Here $m^2/\mu=9L/(4\kappa)$, $d^{-2}=32J/L^2$, and
$K\ge c_K\kappa$.  Absorbing $w_*^{-2J}$ and the fixed constants
into $C_B^D$, and $DJ$ into $D^{p_B}$, proves \eqref{waI:eq:B}.
Increasing $C_{\rm fine}$ to prove uniform locality therefore costs
only a fixed factor per depth in the signal.  It introduces no
polynomial loss in $\kappa$ or $\vartheta$.

With $p=B^2/(\sigma_g^2+B^2)$ the same fixed scalar sample law
$g_\xi=g^0+(\xi/p)(g-g^0)$ has unbiased full gradients and global
variance at most $\sigma_g^2$.  No assertion concerning variance
of function values is used.  The exact transport family uses the
mean potential and needs no Bernoulli multiplier.
\end{proof}

\subsection{The same information count under the new locality certificate}

\begin{theorem}[Information counts without $H/\tau$ smallness]
\label[theorem]{ni:thm:information-count}\label[theorem]{waI:thm:information-count}
Assume the same fixed recursive mean and inactive potentials, the global
primitive budgets, and the global output margin $8\epsilon$ whenever
the final macro column is incomplete. Suppose \cref{ni:lem:no-cascade}
holds for the ordinary conservative family and \cref{ni:lem:transport}
holds for its exact transport variant. Set
\[
 \mathcal N=S\prod_{j=0}^J n_j\ge64,\qquad
 \pi=\frac{B^2}{\sigma_g^2+B^2},
\]
where $B>0$ bounds the complete informative gradient for the unchanged
fixed sample law. In sufficiently large finite dimensions, every
randomized adaptive algorithm making at most $Q\le\mathcal N/(16\pi)$
ordinary calls has output gradient norm at least $8\epsilon$ with
probability at least $13/16$. For the exact transport family, with
$n\ge5$, the probability is at least $3/4$ for
$Q\le n\mathcal N/64$. Arbitrary measurable outputs are allowed.
The corresponding worst-case expected costs are
$\Omega(\mathcal N/\pi)$ and $\Omega(n\mathcal N)$.
No comparison between $H$ and a blind width is assumed.
\end{theorem}
\begin{proof}
Use $K_x=S$ macro columns and $K_j=S\prod_{i=0}^j n_i$ columns in the
independent level-$j$ frame, with the same lexicographic groups.
One successful coin may disclose the next deepest column. Once all
columns of a child subtree have been disclosed, its parent column is
free. Thus after $a$ successes the prefix lengths are
\[
 r_x(a)=\left\lfloor\frac{a}{\prod_{i=0}^J n_i}\right\rfloor,
 \qquad r_j(a)=\left\lfloor
        \frac{a}{\prod_{i=j+1}^J n_i}\right\rfloor.
\]
These are the same fixed-function prefix identities verified in the
source-support argument above. No column is disclosed on a coin-zero
call to aid computation of a mean inverse.

Fix the horizon first. Let $m_j=2Q+1$ and $m_x=Q+1$ count vector and
terminal insertions, with the output directions inserted after the
algorithm fixes its output. Choose $e_j>0$ satisfying
\eqref{ni:eq:projection-target}, and choose
$e_x\le(b_0/(128R_x))^2$. Sufficient dimensions for the original
projection-energy argument are
\begin{equation}\label{waI:eq:dimensions}
 d_j^{\rm amb}\ge C\left[m_j+K_j+2+
 m_j(K_j+e_j^{-1})
       \log\frac{16(J+2)m_jK_j}{\delta_{\rm prob}}\right],
\end{equation}
with the analogous expression for the macro frame and, if necessary,
an additional maximum with $K_j+2$. The latter also supplies the
transverse direction in \cref{ni:lem:radial}. Take
$\delta_{\rm prob}=1/8$. The proof of Lemma~\ref{hier:lem:projection} controls
all pre-coin insertions and all post-revelation known spaces, including
its full-column energy correction. Its inputs are solely the finite
numbers of insertions and revelations and the prescribed projection
targets; it has no numerical accuracy assumption.

Conditional on the private seed and disclosed history, rotations fixing
each known space preserve its residual-frame law. Restricted inverses
are measurable functions of the disclosed projections, so their scalar
normal outputs impose no further conditioning. The resulting projection
event has probability at least $7/8$. On it
\cref{ni:lem:no-cascade} proves equality of the required inverse in each
sample branch. Using the same frames, coins and private seed, induction
therefore gives identical actual and virtual queries, sampled gradients,
and terminal outputs. A number of successes less than $\mathcal N$
leaves the last macro column unknown. The terminal insertion puts it
in the incomplete interval. Since
\[
 \Pr\{\operatorname{Binomial}(Q,\pi)\ge\mathcal N\}
 \le Q\pi/\mathcal N\le1/16,
\]
the global margin holds with probability at least $13/16$.

For transport apply \cref{dt:lem:observer} with $m_{\rm out}=1$ and
an aggregate energy target smaller than all adjusted physical endpoint
targets in \cref{ni:lem:transport}. Choose the separate ancestor-frame
dimensions by \eqref{waI:eq:dimensions}. Allocate failure probabilities
$1/8$ to transport and $1/8$ in total to those ancestor observers.
The exact matrix response is supplied by frozen one-hop actions;
the nonlinear part couples by \cref{ni:lem:transport}. In particular,
newly disclosed terminal and ancestor columns remain in their wider
blind intervals on the current call, so no second stage is activated
by the inverse in that call. The existing one-hop schedule completes
at most $1+Q/(n-1)<\mathcal N$ stages for
$Q\le n\mathcal N/64$, $n\ge5$, and $\mathcal N\ge64$.
The macro margin consequently holds with probability at least $3/4$.

Finally take $Q_0=\lfloor\mathcal N/(16\pi)\rfloor$ or
$Q_0=\lfloor n\mathcal N/64\rfloor$, respectively, and choose all
dimensions at that deterministic horizon before execution. A uniformly
correct algorithm of worst-case expected cost at most $Q_0/8$ can be
truncated at $Q_0$, using the origin if it has not stopped. Markov's
inequality gives truncation probability at most $1/8$ and original
barrier probability at most $1/8$, under the same fixed hard prior.
The truncated barrier probability is at most $1/4$, contradicting
its fixed-horizon lower bound. The displayed floors lose only numerical
constants since $\mathcal N\ge64$. Averaging over the fixed prior
and taking a worst instance proves the asserted minimax expected costs.
\end{proof}

Lemma~\ref{waI:lem:uniform-locality} verifies these information
hypotheses on the quadratic variation range. The geometry supplies
the floors and the $8\epsilon$ margin under
$\epsilon\le c\min\{L,\sqrt{L\Delta}\}$. Neither the activation
count nor the full-gradient signal changes when the projection target
is decreased. All remaining factors $C^D D^p$ are bounded by fixed
powers of $\log(2+\vartheta)$.

\section{An independent upper-gradient noise lower bound}
\label{app:upper-noise-lower}

The lower-noise construction does not charge upper-gradient noise:
its upper oracle is exact. We now construct a separate family in which
the entire lower oracle is a public quadratic, while a conservative
upper sample hides the next micro direction. The full informative
upper-gradient amplitude is of order $\epsilon/\kappa$.

\begin{theorem}[Conservative upper noise with an exact lower oracle]
\label[theorem]{uf:thm:lower}
There are numerical constants $c,C>0$ such that, for every
$L,\Delta>0$, $\kappa\ge2$, $\sigma_f\ge0$, and
\begin{equation}\label{uf:eq:accuracy}
 0<\epsilon\le c\min\{L,\sqrt{L\Delta}\},
\end{equation}
the exactly prescribed primitive class with $\mu=L/\kappa$ satisfies
\begin{equation}\label{uf:eq:rate}
 \mathsf T_{\rm SFO}(\epsilon)
 \ge c\left(L\kappa^2\Delta\epsilon^{-2}
       +L\kappa^4\sigma_f^2\Delta\epsilon^{-4}\right).
\end{equation}
The family has $C_f,L_f,L_g\le L$, lower modulus at least $\mu$,
initial gap at most $\Delta$, and $y^*(0)=0$. Its lower Hessian is
constant, so it meets every nonnegative Hessian-variation budget,
including zero. The lower gradients are exact, and the upper
gradients are conservative with globally bounded variance
$\sigma_f^2$. The result holds for arbitrary randomized adaptive
queries and arbitrary terminal outputs, both for fixed horizons
and worst-case expected query costs. Its hyperobjective satisfies
$\operatorname{Lip}(\nabla F)\le CL\kappa^2$.
\end{theorem}

\subsection{The direct upper placement}

Use the base macro--micro geometry of Appendix~\ref{app:lower},
with all chain and radial curvature scales replaced by
$\ell=\eta L$, where $\eta>0$ is a sufficiently small fixed
numerical constant. Put
\[
 m=3\mu/2,\qquad c_{xy}=L/4,\qquad
 K=c_{xy}/m=\kappa/6.
\]
First suppose $K$ exceeds the fixed threshold of the base micro
certificate. With its already fixed profiles and geometry constants,
take
\[
 \beta=b/K,\qquad M=\lfloor\zeta^6K^2\rfloor,
 \qquad R_u=C_u b/\zeta.
\]
The common radial penalty has coefficient $\nu=\zeta\ell$.
Let $W_i$, $\psi_i$, $R_i$, $Q_i$, $A_m$, and $P_u$ be precisely
the functions used in \eqref{eq:F} under these substitutions. Define
\begin{equation}\label{uf:eq:mean}
 \begin{split}
 f(x,v,u)&=\mathcal F(x,u)
 =A_m(x)+P_u(u)+\sum_{i=1}^S\psi_i(x)W_i(u)
                   +\sum_{i=1}^SR_i(x)Q_i(u),\\
 g(x,v,u)&=\tfrac m2\|u\|^2-c_{xy}\langle v,u\rangle.
 \end{split}
\end{equation}
Here the complete micro core is in the upper primitive. There is
no lower residual, no normal coordinate, and no hidden frame in $g$.
Its exact response is $u^*=Kv$, and the true hyperobjective is
exactly the base geometric function $F(x,v)=\mathcal F(x,Kv)$.

The precomposition Hessian bound in the proof of
Theorem~\ref{hier:thm:mean} applies to this complete micro core:
\[
 \|D^2_{(x,u)}\mathcal F\|_\infty\le C\ell,
 \qquad \|\nabla_u\mathcal F\|_\infty\le C\ell b.
\]
The latter bound follows directly from the one-active selector,
$\ell\beta\sqrt M\le\ell\zeta^3 b$, the terminal receiver
bound, $\|\nabla P_u\|\le\nu R_u$, and
$|R_i|\|\nabla Q_i\|\le C\ell b$. None of these estimates
depends on $S$ or $\Delta/L$. Taking $\eta$ small gives $L_f\le L$;
the accuracy condition below gives $C_f\le L$.

The joint quadratic Hessian of $g$ has norm
\[
 \frac{m+\sqrt{m^2+4c_{xy}^2}}2
 \le\frac{3+\sqrt{13}}8L<L,
\]
and its lower Hessian is $mI\succeq\mu I$. All lower Hessian
blocks are constant. The affine response gives
$\operatorname{Lip}(\nabla F)\le C\ell\max\{1,K^2\}
\le CL\kappa^2$.

The already proved global margin supplies
\begin{equation}\label{uf:eq:margin}
 U_S^\top B_{R_m}(x)<b
 \quad\Longrightarrow\quad
 \|\nabla F(x,v)\|\ge c_*\ell b.
\end{equation}
The same geometry gives coercivity and
$F(0)-\inf F\le C(S+1)\ell b^2$. These statements concern the
same functions in \eqref{uf:eq:mean}, since direct substitution
produces the geometric hyperobjective without any lower lift.

\subsection{A fixed globally conservative upper sample}

Let $\rho_i$ be the centered micro frontier residual used in
\eqref{eq:lower-residual}; it is a function of the latent coordinates
$V_i^\top B_u(u)$. Fix $\tau=\beta/16$ and flatten those latent
coordinates at resolution $\tau$, obtaining $\rho_{i,\tau}$.
Set
\[
 R(x,u)=\sum_i\psi_i(x)\rho_i(u),\qquad
 R_\tau(x,u)=\sum_i\psi_i(x)\rho_{i,\tau}(u),\qquad
 D=R-R_\tau,
\]
and define the global scalar functions
\begin{equation}\label{uf:eq:sample}
 f_0=f-D,\qquad
 f_\xi=f_0+\frac\xi pD,
 \qquad p=\frac{B_f^2}{\sigma_f^2+B_f^2},
 \qquad \xi\sim\operatorname{Ber}(p).
\end{equation}
All these functions are fixed after the frames are sampled. In
particular $f_0$ does not depend on a transcript or an exposed prefix.

The frontier and flattening bounds give, globally,
\[
 |\rho_i-\rho_{i,\tau}|\le C\ell\beta\tau,
 \qquad
 \|\nabla_u(\rho_i-\rho_{i,\tau})\|\le C\ell\beta.
\]
The macro selector derivative is at most $C/b$, and at most one
selector has a nonzero jet. Thus
\begin{align*}
 \|\nabla_uD\|&\le C\ell\beta,\\
 \|\nabla_xD\|&\le C\ell\beta\tau/b\le C\ell\beta,
 \qquad \nabla_vD=0.
\end{align*}
Choose $B_f=C_B\ell\beta$ to dominate this full Euclidean
gradient norm. Returning $(\nabla f_\xi,\nabla g)$ is globally
unbiased, and
\[
 \mathbb E\|\nabla f_\xi-\nabla f\|^2
 =(p^{-1}-1)\|\nabla D\|^2\le\sigma_f^2.
\]
Each response is the gradient of a globally defined scalar sample
function. The lower variance is zero. Each coin is fresh and
independent of both frames, the algorithm's seed, and previous coins.

\subsection{Information accounting with the coin in the upper oracle}

The role of the stochastic response must be checked directly here:
the earlier information theorem placed its coin in the lower oracle.
Use the same two independent Haar frames and projection observer as
in Section~\ref{sec:hierarchy-oracle}. The observer inserts the
current macro query $x$ and both micro directions $v,u$, and discloses
every frame column's projections onto the enlarged known spaces.
After an active coin it reveals the next micro column, and reveals
its macro parent for free when the micro block is complete.

At a query with every unseen micro coordinate at most $\tau/2$,
the micro part of the inactive function is
\[
 W_i-\rho_i+\rho_{i,\tau}.
\]
Its latent gradient uses only the known prefix: $W_i-\rho_i$ is
the base carrier, and the flattened residual has the same prefix
support. Its terminal receiver is flat on an unknown terminal.
The release satisfies $Q_i=Q_i'=0$ there, so no new macro vector
is present in the full upper gradient. All selector gradients use
known macro columns, and future selectors have zero value and zero
first derivative. An active sample can add at most the next micro
column through $D$. Radial derivatives add only the public query
direction and scalar coefficients computed from the disclosed
projections. These assertions include the macro part of the upper
gradient; they are not restricted to $f_y$.

The exact lower gradient is
$\nabla_vg=-c_{xy}u$, $\nabla_ug=m u-c_{xy}v$, and
$\nabla_xg=0$. It is public and contains no frame information.
Thus placing the coin in the upper response makes no change to the
activation count. To define the virtual protocol off the projection
event, compute all scalar coefficients from the disclosed projections
and truncate upper vector contributions to the old prefix, or to
the enlarged prefix after an active coin. Return the public lower
gradient unchanged. This is measurable for every realization; on
the projection event no nonzero term is discarded.

Choose the dimensions in \eqref{hier:eq:dimension}, with macro
threshold $b/4$ and micro threshold $\tau$. Include the terminal
directions after the output is fixed. The same conditional rotational
invariance argument and actual--virtual induction show that, for
$N\le SM/(16p)$, an unfinished last macro column and the good
projection event occur with probability at least $13/16$.
Equation \eqref{uf:eq:margin} then gives an $8\epsilon$ gradient
barrier once $b$ has been calibrated below. This applies to arbitrary
adaptive algorithms and unqueried terminal outputs.

For expected costs choose the dimensions for the predetermined
horizon $n=\lfloor SM/(16p)\rfloor$. If a uniformly correct
algorithm had expected cost below $n/8$, truncation would lose at
most $1/8$ probability. Its expected-gradient guarantee gives an
original norm at most $4\epsilon$ with probability at least $3/4$;
the truncated success probability would therefore be at least
$5/8$, contradicting the fixed-horizon upper probability $3/16$.
Averaging over frames and selecting one mean instance proves the
randomized minimax conclusion, with the original fresh sample law
on that fixed instance.

\subsection{Calibration and the bounded-condition bridge}

Take
\[
 b=C_b\epsilon/\ell,\qquad
 S=\left\lfloor c_S\Delta/(\ell b^2)\right\rfloor.
\]
Choose $C_b$ first for the $8\epsilon$ margin, then $c_S$ for the
gap, and finally the accuracy constant in \eqref{uf:eq:accuracy}.
The latter ensures that the floor argument is at least sixty-four and
$C\ell b\le L$. The fixed lower threshold on $K$ ensures the
$M$ floor. Hence
\[
 SM\asymp L\kappa^2\Delta\epsilon^{-2},\qquad
 B_f\asymp\epsilon/\kappa.
\]
The activation cost $SM/p$ is therefore the two terms in
\eqref{uf:eq:rate}. This construction works for every positive
gap; no lower bound on $\Delta/L$ has been imposed.

It remains to cover the fixed interval of condition numbers below
the hierarchical threshold. Use the one-chain geometry of
Section~\ref{sec:hierarchy-fallbacks}, now directly in an outer
variable $x$. Denote its complete wrapped chain and public radial
penalty by
\[
 H_{\rm out}(x)=H(U^\top B(x))+P(x).
\]
Use curvature $\ell=\eta L$, width
$h=C_h\epsilon/\ell$, length
$T=\lfloor c_T\Delta/(\ell h^2)\rfloor$, and the existing
radial radius proportional to $h\sqrt T$. Set
\[
 f(x,y)=H_{\rm out}(x),\qquad g(x,y)=\tfrac\mu2\|y\|^2.
\]
Then $f_y=0$, both primitive gradient budgets hold, and the
global end-game and gap follow from the one-chain certificate
\eqref{fb:eq:gap-margin}, with no affine amplification.
The same residual-minus-flattened-residual sample construction, at
resolution $h/16$, has full informative amplitude $B_f\le C\ell h$.
A one-frame projection observer therefore gives
\[
 \Omega\!\left(L\Delta\epsilon^{-2}
        +L\sigma_f^2\Delta\epsilon^{-4}\right).
\]
On a fixed numerical interval $2\le\kappa\le\kappa_h$, this is
at least a fixed numerical multiple of \eqref{uf:eq:rate}; the
loss is at most $\kappa_h^4$. The common accuracy condition
ensures the length floor in this branch as well. Taking the minimum
of the finitely many numerical constants proves
Theorem~\ref{uf:thm:lower} under exactly prescribed budgets.

\section{Endpoint transport: geometry and deterministic complexity}
\label{app:det-transport-geometry}

We replace the deepest amplified vector of the recursive construction by
a chain of quadratic transport blocks.  A source profile is read at the
first endpoint, whereas every predecessor gate and terminal receiver is
read at the last endpoint.  The response graph identifies these two
readouts.  Away from that graph they remain distinct, which is essential
for the information argument.  All claims about the primitive functions
below hold on their entire Euclidean domains.  The structural certificates below first use
the baseline width calibration.  The same endpoint construction with
the independent base length in \cref{app:uniform-accuracy-geometry}
has the wider final accuracy range \eqref{eq:uniform-accuracy-main};
its locality proof is supplied in \cref{waI:sec:information}.

\begin{theorem}[An exact-quadratic deterministic family]
\label[theorem]{tr:thm:quadratic}
There are numerical constants $c,C>0$ such that, for every
$L,\Delta>0$, $\kappa\ge2$, and
\begin{equation}\label{tr:eq:quadratic-accuracy}
 0<\epsilon\le c\min\{L,\sqrt{L\Delta}\},
\end{equation}
a family with an exact quadratic lower objective satisfies
\begin{equation}\label{tr:eq:quadratic-rate}
 \mathsf T_{\rm SFO}(\epsilon)
       \ge cL\kappa^{5/2}\Delta\epsilon^{-2}.
\end{equation}
It has the exactly prescribed primitive budgets
$C_f,L_f,L_g\le L$, lower modulus at least $L/\kappa$,
constant lower Hessian, $y^*(0)=0$, gap at most $\Delta$, and
\[
 \operatorname{Lip}(\nabla F)\le CL\kappa^2.
\]
Both gradients are exact. The bound covers arbitrary randomized
adaptive algorithms, arbitrary measurable terminal outputs, and
worst-case expected query costs in sufficiently large finite dimensions.
By decreasing the fixed construction curvature, the last constant
can be accommodated by the fixed numerical $C_F$ in the definition
of $\mathcal C_{\rm curv}$.
\end{theorem}

The proof uses the finite-dimensional transport observer in
Lemma~\ref{dt:lem:observer}, together with the geometric and support
certificates proved below.  In particular it does not multiply a
standalone strongly convex optimization lower bound by an outer
iteration count.

\subsection{A matrix with amplified endpoints and bounded full response}

The endpoint matrix is the one in \citet[Lemma~4.1]{CJZ}.  We give
the calculations needed here, including its full response norm.
For $n\ge5$ put $m=n-1$ and let $A_n$ be tridiagonal, with endpoint
diagonals $(m+1)/m^2$, interior diagonals $2$, endpoint off-diagonals
$-m^{-1/2}$, and other off-diagonals $-1$.

\begin{lemma}[Endpoint matrix certificate]
\label[lemma]{tr:lem:matrix}
The matrix just defined satisfies
\begin{equation}\label{tr:eq:matrix-budgets}
 \frac{1}{2(n-1)^2}I\preceq A_n\preceq4I,\qquad
 \alpha_n:=(A_n^{-1})_{nn}=\frac{2(n-1)^2}{3},\qquad
 (A_n^{-1})_{1n}=\frac{\alpha_n}{2}.
\end{equation}
For $2\le i\le n-1$,
\[
 \frac{(A_n^{-1})_{in}}{\alpha_n}
     =\frac{n+i-2}{2(n-1)^{3/2}}.
\]
Consequently, if $\gamma_n=\|A_n^{-1}e_n\|/\alpha_n$, then
\begin{equation}\label{tr:eq:gamma}
 \gamma_n^2=\frac{11}{6}-\frac{5}{8(n-1)}
                      +\frac{1}{24(n-1)^2},
 \qquad \frac54\le\gamma_n^2<\frac{11}{6}.
\end{equation}
\end{lemma}
\begin{proof}
Let $B$ be the interior Dirichlet Laplacian of size $m-1$.
Its least eigenvalue is
$4\sin^2(\pi/(2m))\ge4/m^2$.  Eliminating the interior gives
the endpoint Schur complement
\[
 \Sigma=\frac1{m^2}
             \begin{pmatrix}2&-1\\-1&2\end{pmatrix}.
\]
The minimizing interior vector for endpoints $e=(e_1,e_2)$ is
$Te$, where
\[
 T_{j1}=\frac{m-j}{m\sqrt m},\qquad
 T_{j2}=\frac{j}{m\sqrt m},\qquad 1\le j\le m-1.
\]
Summing the two arithmetic-square sequences gives
$\|T\|^2=(m-1)/(2m)\le1/2$.  If an arbitrary vector has
interior part $z+Te$, its quadratic form is
$z^\top Bz+e^\top\Sigma e$, while
\[
 \|e\|^2+\|z+Te\|^2
       \le2(\|e\|^2+\|z\|^2).
\]
This proves the lower spectral bound.  The absolute row sums are
at most $4$, proving the upper bound.  Inverting $\Sigma$ gives
the two endpoint entries in \eqref{tr:eq:matrix-budgets}.
The interior formula follows from the affine interpolation $Te$.
Finally sum the squares of these entries:
\[
 \gamma_n^2=\frac54+
     \frac{1}{4m^3}\sum_{k=m+1}^{2m-1}k^2
 =\frac{11}{6}-\frac{5}{8m}+\frac{1}{24m^2}.
\]
\end{proof}

Choose a sufficiently large fixed numerical threshold $\kappa_{\rm tr}$.
For $\kappa\ge\kappa_{\rm tr}$ define
\begin{equation}\label{tr:eq:normalization}
 \begin{gathered}
 n=\left\lfloor\sqrt{\kappa/128}\right\rfloor,\qquad
 \delta_H=L/16,\qquad c_{\rm end}=L/4,\qquad m_0=3\mu/2,\\
 K=\frac{c_{\rm end}}{\delta_H}\alpha_n=4\alpha_n,\qquad
 c_0=m_0K,\qquad L_0=\eta_0L .
 \end{gathered}
\end{equation}
Here $\eta_0>0$ is sufficiently small and fixed.  Increase
$\kappa_{\rm tr}$ so that $n\ge5$ and all physical middle-chain
floors used below exceed their fixed thresholds.  There are numerical
$c_K,C_K>0$ such that
\begin{equation}\label{tr:eq:K-uniform}
 c_K\kappa\le K\le\kappa/48,\qquad c_0\le L/32,\qquad
 n\asymp\sqrt\kappa.
\end{equation}
Every level of the construction uses this same $K$.
In particular, the isotropic levels use coupling $c_0=m_0K$;
the value $\kappa/6$ from the earlier normalization is not retained.

Let the physical transport blocks be $Y_1,\ldots,Y_n\in\mathbb R^D$.
Take independent Haar orthogonal maps $O_l:\mathbb R^D\to\mathbb R^D$,
put $Q_n=I$, $Q_l=O_l^\top Q_{l+1}$, and let
$\mathcal Q=\operatorname{diag}(Q_1,\ldots,Q_n)$.  Define
\begin{equation}\label{tr:eq:H}
 H=\delta_H\mathcal Q(A_n\otimes I_D)\mathcal Q^\top.
\end{equation}
Its off-diagonal block from layer $l$ to $l+1$ is a known scalar
multiple of $O_l$.  Lemma~\ref{tr:lem:matrix} gives
\begin{equation}\label{tr:eq:H-slack}
 H\succeq4\mu I,\qquad \|H\|\le L/4.
\end{equation}
The full joint Hessian of
$\frac12Y^\top HY-c_{\rm end}\langle v,Y_n\rangle$
has norm at most $\|H\|+c_{\rm end}\le L/2$.

\subsection{Split endpoint chains and exact graph identity}

At each original deepest leaf let $V$ denote its group of columns
in the common deepest frame.  Sample that whole frame independently
of the transports and set $V_l=Q_lV$ for every group.  Thus
$O_lV_l=V_{l+1}$, and all groups at a given physical layer are
mutually orthogonal.  No carrier uses $Q_l$ as a separately
queryable linear map: it uses only its endpoint frame columns.

Let $R,\nu$ be the old common deepest radius and radial coefficient.
Write $\gamma=\gamma_n$, and put
\begin{equation}\label{tr:eq:physical-wrapper}
 \begin{gathered}
 R_{\rm ph}=\gamma R,\qquad \nu_{\rm ph}=\nu/\gamma^2,\qquad
 \mathcal B(Y)=\frac{Y}{\sqrt{1+\|Y\|^2/R_{\rm ph}^2}},\\
 a_i(Y)=2\langle (V_1)_i,\mathcal B(Y)_1\rangle,\qquad
 b_i(Y)=\langle (V_n)_i,\mathcal B(Y)_n\rangle,\\
 P_{\rm ph}(Y)=\nu_{\rm ph}R_{\rm ph}^2
             \left(\sqrt{1+\|Y\|^2/R_{\rm ph}^2}-1\right).
 \end{gathered}
\end{equation}
The norm in $\mathcal B$ is the norm of the full $n$-block vector.
This is the full-vector wrapper throughout this appendix.

For a leaf of original length $r$, curvature $\ell$, and width $h$,
replace its chain by
\begin{equation}\label{tr:eq:split-chain}
 \begin{split}
 C_{\rm sp}(a,b)=\ell h^2\bigg[
   -\Phi(a_1/h)+\sum_{i=2}^r
   \{\Psi(-b_{i-1}/h)\Phi(-a_i/h)
       -\Psi(b_{i-1}/h)\Phi(a_i/h)\}\bigg].
 \end{split}
\end{equation}
Every occurrence of the terminal coordinate in $D_+,D_-,Q$
is replaced by $b_r$.  Every predecessor activation and every
tail mask uses $b$ alone.  Thus, with $A(t)=\Psi(t)+\Psi(-t)$,
\begin{equation}\label{tr:eq:split-selectors}
 \psi_1=H_1(b),\qquad
 \psi_i=A(b_{i-1}/h)H_i(b),\qquad
 R_i=-\ell h^2\psi_i(b)\Phi(a_i/h).
\end{equation}
The deepest payload is centered as before.  Its centered frontier is
\begin{equation}\label{tr:eq:split-frontier}
 R_{\rm sp}=-\ell h^2\sum_i
         \psi_i(b)\{\Phi(a_i/h)-\Phi(0)\}.
\end{equation}
Ancestor weights multiply this expression exactly as in the recursive
construction.

\begin{lemma}[Exact response trace]
\label[lemma]{tr:lem:trace}
The solution of $HY=c_{\rm end}e_nv$ is
\begin{equation}\label{tr:eq:linear-response}
 Y_l^*(v)=\frac{c_{\rm end}}{\delta_H}
                 (A_n^{-1})_{ln}Q_lv.
\end{equation}
For every $v\in\mathbb R^D$, including its components outside the hard
frame,
\[
 2Q_1^\top Y_1^*=Y_n^*=Kv,\qquad
 \|Y^*\|=\gamma K\|v\|.
\]
Consequently
\begin{equation}\label{tr:eq:exact-readout}
 a_i(Y^*)=b_i(Y^*)=\langle V_i,B_R(Kv)\rangle,\qquad
 P_{\rm ph}(Y^*)=P_{\rm old}(Kv).
\end{equation}
Replacing the deepest payload and radius by
\eqref{tr:eq:physical-wrapper}--\eqref{tr:eq:split-frontier}
therefore preserves the entire old hyperobjective as a function of
all outer variables, provided every level uses the common $K$ in
\eqref{tr:eq:normalization}.
\end{lemma}
\begin{proof}
Equation \eqref{tr:eq:linear-response} follows by inversion of
\eqref{tr:eq:H}.  The endpoint and norm identities are
Lemma~\ref{tr:lem:matrix}.  At this response,
\[
 \sqrt{1+\|Y^*\|^2/R_{\rm ph}^2}
      =\sqrt{1+K^2\|v\|^2/R^2},
\]
and $\nu_{\rm ph}R_{\rm ph}^2=\nu R^2$.
This proves both identities in \eqref{tr:eq:exact-readout}.
On the graph every split edge is the corresponding original
two-sided edge, every gate and receiver is unchanged, and the public
radial penalty has the same value.  Substituting into the recursive
payload identity proves equality of the functions, not only their
values or gradients at the origin.
\end{proof}

\begin{lemma}[Global split-chain derivative and support ledger]
\label[lemma]{tr:lem:split-ledger}
The split replacement preserves the selected-carrier ledger in
Lemma~\ref{rec:lem:carrier-jets}, with a change of numerical constants
only and with no factor depending on the transport length $n$.
In particular:
\begin{enumerate}
\item At every off-graph point at most one selector in
\eqref{tr:eq:split-selectors} has a nonzero jet.  Its first three
derivative bounds and square-root first-derivative bound are the
original ones.
\item Before radial composition the split chain satisfies
\[
 |C_{\rm sp}|\le Cr\ell h^2,\quad
 \|DC_{\rm sp}\|\le C\ell h\sqrt r,\quad
 \|D^2C_{\rm sp}\|\le C\ell,\quad
 \|D^3C_{\rm sp}\|\le C\ell/h.
\]
The same orders hold after adding the old terminal receivers and
composing with \eqref{tr:eq:physical-wrapper}.
\item The centered frontier has joint derivatives
$\|D^jR_{\rm sp}\|\le C\ell h^{2-j}$ for $0\le j\le3$.
Flattening only its $a$ coordinates at a scale $\tau\le h/16$
gives
\[
 |R_{\rm sp}-R_{{\rm sp},\tau}|\le C\ell h\tau,\qquad
 \|D(R_{\rm sp}-R_{{\rm sp},\tau})\|\le C\ell h .
\]
The same statements with $C^{J+1}$ hold after ancestor weighting.
\item Suppose $|b_i|\le h/4$ for all $i>k$.
The latent gradient of the full split chain uses only
$a_1,\ldots,a_{k+1}$ and $b_1,\ldots,b_k$.
Subtracting the centered frontier removes the $a_{k+1}$ direction.
Every unseen terminal receiver and release gate is flat.
\end{enumerate}
\end{lemma}
\begin{proof}
The selectors depend only on $b$.  Take the last coordinate with
$|b_i|\ge h/2$, with index zero if none exists.  All earlier tail
masks have flat zero jets and all later predecessor activations
have flat zero jets.  The original one-active proof and its
square-root estimates apply verbatim in these latent coordinates.
At equality $|b_i|=h/2$, the relevant cutoff jets vanish, so the
argument includes the boundary.

Each $a_i$ enters only its incoming edge and each $b_i$ only the
next edge.  This bounded-degree graph and the fixed profile
derivatives give the displayed norm bounds in $(a,b)$, independently
of $r$.  The endpoint linear readout has norm at most $2$, and the
full radial wrapper has bounds $1,C/R_{\rm ph},C/R_{\rm ph}^2$ for
its first three derivatives.  Since $R_{\rm ph}\ge R$, the same
radius calculation as in Lemma~\ref{rec:lem:weighted-local}
absorbs the $O(\ell h\sqrt r)$ first derivative.  Receiver
derivatives use $b_r$ and the old width scales, so their bounds
also change by numerical constants only.

For the support assertion use the identity
\[
 \Psi(-b_{i-1}/h)\Phi(-a_i/h)
 -\Psi(b_{i-1}/h)\Phi(a_i/h)
 =-A(b_{i-1}/h)\Phi(a_i/h)
       +C_\Phi\Psi(-b_{i-1}/h).
\]
All edges after $k+1$ have flat predecessor activations.  The tail
mask of edge $k+1$ is flat one, and subtracting its residual
cancels its entire source dependence.  The centering compensation
$\ell h^2\Phi(0)\sum_i\psi_i(b)$ is not a global constant, but on
this blind tail it uses only known predecessor coordinates.
Receivers are flat for $|b_r|\le h/4$, and so is the wider release.

For flattening, apply the original scalar flattening profile to
each $a_i$, leaving every $b$ activation and mask unchanged.
Only one selector jet contributes.  The scalar profile estimates
give the two stated difference bounds.  Joint derivatives in $b$
introduce factors $h^{-1}$, which are bounded by the finer
flattening scales.  These same bounds verify the hypotheses of
the fine multiscale lemma with $b$ and the ancestors treated as
selector variables.  Multiplication by the ancestor products has
exactly the width ratios in the proof of
Lemma~\ref{rec:lem:carrier-jets}.  It therefore loses at most a
fixed factor per level, with no dependence on $n$ or on the
ambient dimensions.
\end{proof}

\subsection{A hybrid partial Legendre lift with the same primitive budgets}

First consider the recursive branch $J\ge1$.  Use the hierarchy
in \eqref{rec:eq:geometric-scales}, but with the common $K$ from
\eqref{tr:eq:normalization}.  The unchanged vector blocks are
$u_0,\ldots,u_{J-1}$; the block $u_J$ is replaced by $Y$.
Put
\[
 \mathfrak u=(u_0,\ldots,u_{J-1},Y).
\]
Let $\widetilde f_0,\widetilde r_j,\widetilde R_{\rm fr},
\widetilde A_J$ denote the explicit split replacements of
\eqref{rec:eq:upper-base}, the selected carriers, the centered
frontier, and its coarse complement.  Every parent-to-deep release
uses the terminal $b$ readout.  Other releases are unchanged.
These definitions still satisfy the telescoping identity
\begin{equation}\label{tr:eq:split-telescoping}
 \widetilde{\mathcal F}
   =\widetilde f_0+\sum_{j=1}^J\widetilde r_j
                   +\sum_{j=1}^{J-1}P_j+P_{\rm ph}.
\end{equation}

Use $d=L/(4\sqrt{2J})$.  Build the hidden dual extension explicitly
as
\begin{equation}\label{tr:eq:extension-hidden}
 \widetilde E_h=
 \sum_{j=1}^{J-1}\frac{\widetilde r_j}{d}q_{W_j}(q_j)
 +\frac{\widetilde A_J}{d}q_{W_J}(q_J)
 +\sum_l\frac{\widetilde e_l}{d}q_{w_l}(q_J),
\end{equation}
where $W_j=h_jn_j^{1/3}$ and the $\widetilde e_l$ decompose
$\widetilde R_{\rm fr}$ by source-only flattening.  Use the same
fine parameter and normal widths as in \eqref{rec:eq:normal-scales},
enlarging their fixed constants for Lemma~\ref{tr:lem:split-ledger}.
The radial extension of the expanded block is
\begin{equation}\label{tr:eq:extension-radial}
 \widetilde E_{r,J}(p_Y,t_J)
 =\frac{\nu_{\rm ph}R_{\rm ph}}d\,t_J
     \left(\sqrt{R_{\rm ph}^2+\|p_Y\|^2+t_J^2}-R_{\rm ph}\right).
\end{equation}
The other radial extensions are \eqref{rec:eq:radial-extension}.
Write $\widetilde E$ for their sum with $\widetilde E_h$, and
\[
 \widetilde\Psi_x(z)=\tfrac12\|z\|^2+\widetilde E(x,z).
\]
Only the lower variables are conjugated.  For the primal lower
variable $y=(\mathfrak u,s_1,\ldots,s_J,\widetilde s_1,\ldots,
\widetilde s_J)$ define
\begin{equation}\label{tr:eq:hybrid-primitives}
 \begin{split}
 g(x,v;y)={}&m_0\widetilde\Psi_x^*(y)
       +\tfrac12Y^\top(H-m_0I)Y
       -c_0\sum_{j=0}^{J-1}\langle v_j,u_j\rangle
       -c_{\rm end}\langle v_J,Y_n\rangle,\\
 f(x,v;y)={}&\widetilde f_0(x,\mathfrak u)
                      +d\sum_{j=1}^J(s_j+\widetilde s_j).
 \end{split}
\end{equation}

\begin{lemma}[Hybrid lift, exact graph, and complete budget]
\label[lemma]{tr:lem:hybrid}
With the recursive depth and curvature recurrence
\eqref{rec:eq:depth}--\eqref{rec:eq:curvature-recurrence}, there are
fixed numerical constants $C_{\rm acc},q_{\rm acc},C_b$ for which
\begin{equation}\label{tr:eq:accuracy-depth}
 b=\frac{C_b\epsilon}{L_0(w_*\eta)^J},\qquad
 0<\epsilon\le cC_{\rm acc}^{-J}(J+1)^{-q_{\rm acc}}
           \min\{L/(\kappa\vartheta),\sqrt{L\Delta}\}
\end{equation}
ensures that \eqref{tr:eq:hybrid-primitives} has the exactly
prescribed primitive budgets.  Its lower optimum is
\begin{equation}\label{tr:eq:hybrid-response}
 \begin{gathered}
 u_j^*=Kv_j\quad(0\le j<J),\qquad
 Y_l^*=\frac{c_{\rm end}}{\delta_H}(A_n^{-1})_{ln}Q_lv_J,\\
 s_j^*=\widetilde r_j(x,\mathfrak u^*)/d,\qquad
 \widetilde s_j^*=P_j(u_j^*)/d\quad(j<J),\qquad
 \widetilde s_J^*=P_{\rm ph}(Y^*)/d.
 \end{gathered}
\end{equation}
The exact hyperobjective is the original recursive $F$ with
the common $K$ in \eqref{tr:eq:normalization}.  It has
$F(0)=0$, $y^*(0)=0$, initial gap at most $CSL_0b^2$, and
\begin{equation}\label{tr:eq:global-margin}
 Z_S(x)<b\quad\Longrightarrow\quad
           \|\nabla F(x,v)\|\ge8\epsilon
\end{equation}
after increasing the fixed $C_b$.  Here $Z_S$ is the final
bounded macro coordinate of the original construction.
\end{lemma}
\begin{proof}
Lemma~\ref{tr:lem:split-ledger} gives the same joint carrier bounds
as Lemma~\ref{rec:lem:carrier-jets}.  For a coarse normal term
these imply, explicitly,
\[
 \|D^2\widetilde E_{h,j}\|
       \le C^{J+1}(\ell_j/d)h_jn_j^{2/3},\qquad
 \|D^3\widetilde E_{h,j}\|
       \le C^{J+1}(\ell_j/d)n_j^{1/3}.
\]
The deepest fine sum has bounds
$C^{J+1}(\ell_J/d)h_J\sqrt v$ and
$C^{J+1}(\ell_J/d)v$.  The old fine parameter pays the latter;
the recurrence $a_jn_j^{1/3}\le CT$ pays the coarse sum.
The radial extension \eqref{tr:eq:extension-radial} has Hessian
bound $C\nu_{\rm ph}R_{\rm ph}/d$ and third-derivative bound
$C\nu_{\rm ph}/d$.  These are at most their old bounds, since
$\nu_{\rm ph}R_{\rm ph}=\nu R/\gamma$ and
$\nu_{\rm ph}=\nu/\gamma^2$.  Thus the proof of
Lemma~\ref{rec:lem:extension-budget}, with the displayed ledger
and the same numerical-slack choices, gives
\begin{equation}\label{tr:eq:E-budget}
 \|D^2\widetilde E\|\le\delta_E,\qquad
 \|D^3\widetilde E\|\le \rho/(C_{\rm Leg}m_0),\qquad
 \|D\widetilde E_h\|_\infty
       \le C^{J+1}(J+1)^p b^2/K^2=:H_{\rm hid}.
\end{equation}
All norms here are joint norms, including the macro and normal
variables; in particular no mixed $x$ block is omitted.
The constants do not depend on $n$.

Before the Legendre perturbation the lower Hessian is block
diagonal, with blocks $H$, $m_0I$ on the ordinary vectors, and
$m_0I$ on the normals.  It is at least $m_0I$ by
\eqref{tr:eq:H-slack}.  Lemma~\ref{rec:lem:legendre} bounds the
joint Hessian perturbation by $Cm_0\delta_E$.  Taking
$C\delta_E\le1/6$ leaves lower modulus at least $\mu$.
The deep coupled quadratic has joint norm at most $L/2$.
The other coupled blocks have norm at most $m_0+c_0<L/2$ in
the large-$\kappa$ branch, and the normal blocks also have norm
less than $L/2$.  Their operator norms do not add over blocks.
The same small perturbation therefore leaves $L_g\le L$.
Only that perturbation has a nonconstant Hessian, and
\eqref{tr:eq:E-budget} and the joint Legendre formula give the
required Hessian-block Lipschitz bounds at most $\rho$.

The base upper carrier still has
\[
 \|D^2\widetilde f_0\|\le CL_0,\qquad
 \|D_{\mathfrak u}\widetilde f_0\|\le CL_0b.
\]
For $J>1$ it does not involve the expanded deepest block.
For $J=1$ its only new dependence is a parent release with
value $O(L_0h_0^2)$ and width $h_0$ in the terminal endpoint;
the readout norm is bounded, so the same two bounds follow.
Choose $\eta_0$ small for $L_f\le L$.  Equation
\eqref{tr:eq:accuracy-depth} makes the geometric lower gradient
at most $L/2$, while the $2J$ normal coefficients have joint norm
$d\sqrt{2J}=L/4$.  Thus $\|f_y\|_\infty<L$.

At normal optimality the dual normal coordinates vanish.
Because $\widetilde E(x,p,0)=0$ identically in $p$, its tangential
gradient vanishes there, and the primal tangent equals the dual
tangent.  The ordinary stationarity equations reduce to
$m_0u_j=c_0v_j$, and the expanded equation reduces to
$HY=c_{\rm end}e_nv_J$.  The normal trace derivatives give
the remaining identities in \eqref{tr:eq:hybrid-response}.
Strong convexity makes this the unique global optimum.
Equations \eqref{tr:eq:split-telescoping} and
\eqref{tr:eq:exact-readout} identify the full hyperobjective.
Lemma~\ref{rec:lem:global} now gives its gap, origin, and global
margin.  Indeed its lower bound is
$c w_*^{J+1}\eta^J L_0b=cC_bw_*\epsilon$.
Increase $C_b$ before choosing the final accuracy constant to
make this at least $8\epsilon$.
\end{proof}

\subsection{No nonlinear inverse shortcut and serial support}

\begin{lemma}[Quantitative inverse locality for the expanded block]
\label[lemma]{tr:lem:inverse-locality}
Let $\tau_*$ be the smallest blind latent width used in the
recursive proof, including the deepest fine width divided by
eight.  Constants in \eqref{tr:eq:accuracy-depth} can be chosen
so that
\begin{equation}\label{tr:eq:blind-budget}
 \tau_*\ge c\eta^Jb/(K\vartheta),\qquad
 H_{\rm hid}\le\tau_*/128.
\end{equation}
If the relevant undisclosed endpoint columns have squared
projection at most $e_*$ on each query span, with
\[
 e_*\le c(\tau_*/R_{\rm ph})^2
\]
for a sufficiently small fixed $c$, every such endpoint coordinate
at the exact or restricted inverse point remains in its flat
blind interval.  In particular the nonlinear inverse can disclose
at most the eligible next first-endpoint column and cannot carry
that column to the last endpoint.
\end{lemma}
\begin{proof}
The width lower bound is the calculation following
\eqref{rec:eq:hidden-small}; the split replacement changes no
normal width.  More explicitly, the fine width is bounded below
by $c h_Ja_J/\vartheta=c\eta^Jb/(K\vartheta)$, and every other
blind width is larger up to a fixed factor.  Therefore
\[
 \frac{H_{\rm hid}}{\tau_*}
 \le C^{J+1}(J+1)^p
              \frac{b\vartheta}{K\eta^J}.
\]
Substitution of \eqref{tr:eq:accuracy-depth} and
$K\asymp\kappa$ makes this at most $1/128$ after increasing
$C_{\rm acc},q_{\rm acc}$ and decreasing the fixed accuracy
constant.  This estimate contains no transport-length factor.

Write $p_Y$ for the inverse dual tangent at a primal query $Y$.
The public radial extension has tangential derivative
$\gamma_{\rm pub}p_Y$, with
$|\gamma_{\rm pub}|\le\nu_{\rm ph}R_{\rm ph}/d\le1/8$.
Consequently its inverse equation is
\[
 p_Y=\frac{Y-D_{p_Y}\widetilde E_h}{1+\gamma_{\rm pub}}.
\]
It differs by at most $2H_{\rm hid}$ from a scalar multiple of
the query.  The radial wrapper is $1$-Lipschitz and bounded by
$R_{\rm ph}$.  Since the source readout has the extra factor
two, every undisclosed coordinate obeys
\begin{equation}\label{tr:eq:inverse-projection}
 |a_i(p_Y)|,\ |b_i(p_Y)|
       \le2R_{\rm ph}\sqrt{e_*}+CH_{\rm hid}.
\end{equation}
This holds without a bound on the query norm.
Choose the numerical constants in the width and energy budgets
to put the right side inside the specified flat interval.

For clarity, the virtual inverse can be defined on every history.
Include every query direction and every legally disclosed endpoint
column in the known subspaces.  Minimize the strongly convex dual
potential restricted to their direct sum and all normal coordinates;
project the virtual tangent responses back to those subspaces.
All hard-frame projections on these spaces are supplied by the
observer, so this operation is measurable from its transcript even
off the good event.  The restricted equation has the same preceding
bound, with an orthogonal projection on the hidden derivative.
On the good event, Lemma~\ref{tr:lem:split-ledger} puts the full
hidden gradient in the same spaces, apart from the one eligible
source column, which the protocol includes at that round.
The restricted solution then solves the unrestricted equation
and is the exact inverse by uniqueness.

Interior transport layers have no nonradial framed derivative:
their dependence in the carriers is only through the full radial
norm.  Such derivatives are multiples of the corresponding query
or inverse components.  At the two endpoints, all nonradial terms
use only the certified prefix and the eligible source column.
The quadratic part is returned as $HY$; it is not part of the
nonlinear inverse.  In particular its actual oracle response is
\[
 HY-c_{\rm end}e_nv_J+m_0(p_Y-Y),
\]
and no operation in this formula evaluates $H^{-1}$ at a query.
\end{proof}

\begin{lemma}[Serial nonlinear support]
\label[lemma]{tr:lem:serial}
On the simultaneous blind event, enumerate the deepest source
stages in the order of the original recursive prefix protocol.
Their number is
\begin{equation}\label{tr:eq:stages}
 \mathcal N=S\prod_{j=0}^Jn_j.
\end{equation}
In the stronger observer, completed subtrees may reveal their
parent column for free.  Even with this convention, the next
deepest source column can be revealed only after the preceding
one has been disclosed at the last physical layer.
The final outer macro column can be released only after all
$\mathcal N$ deepest stages have completed.
\end{lemma}
\begin{proof}
Within a leaf let $k$ terminal columns be disclosed.  While
$b_{k+1}$ is blind, all predecessor activations for $a_{k+2}$
and later sources are flat zero.  Lemma~\ref{tr:lem:split-ledger}
shows that the nonlinear response uses the first $k$ terminal
columns and at most source $k+1$.  Once that source has been
disclosed, repeating nonlinear queries cannot disclose its
terminal copy or the following source: all the relevant $b$
gates and tail-mask derivatives remain flat.
Lemma~\ref{tr:lem:inverse-locality} proves the same statement
for the baseline calibration. For the independent-base calibration,
Lemma~\ref{ni:lem:transport} proves it on the enlarged variation
range, at arbitrary normal queries.
The public radial parts supply only query-parallel directions.

At every ancestor level the identity
$A_I=H_I-\sum_kR_{I,k}$ removes its unknown incoming frontier.
Its parent release $R_{I,k}Q_{I,k}$ is flat until the child
terminal is disclosed.  Future selectors have zero value and
first derivative on a blind tail.  The source/terminal split
does not change these identities or the ancestor widths.
Thus the original recursive prefix induction still has one
eligible incomplete path.  Completing a leaf can release its
parent column for free, but an unreleased child prevents
advancing that parent and prevents reaching a future leaf.
Inducting up the tree gives precisely $n_J$ stages for each
deepest leaf, $n_{J-1}$ such children for each preceding node,
and finally $S\prod_{j=0}^Jn_j$ stages before the final macro
column is released.  This is a support statement about the
same fixed functions, not an additional oracle restriction.
\end{proof}

\paragraph{Coupling the virtual transcript to the actual oracle.}
Fix the algorithm's private seed and use the same sampled frames and
edge maps for both executions.  At each round the stronger observer
supplies the exact edge actions needed for $HY$, the frame projections
on the current query spans, and the legally disclosed prefix columns.
The virtual nonlinear response uses only these data: it retains the
known prefix and the eligible source, restricts the dual inverse to
their query-augmented subspaces as in
Lemma~\ref{tr:lem:inverse-locality}, and projects the returned nonlinear
directions onto those subspaces.  This rule is defined and measurable
on every history, including histories outside the good event.
It invokes no unobserved residual edge action.  On the simultaneous
blind event, the support statements of Lemma~\ref{tr:lem:split-ledger}
and the ancestor prefix identities make every omitted future
dependence have zero relevant jet.  The restricted inverse then
satisfies the full inverse equation and equals the true inverse by
uniqueness.  Thus the exact upper and lower gradient responses coincide with
their actual values at that query. This appendix uses the mean
potential on every call; its eligible first-endpoint source is
included by the transport observer on every such call.  Starting from the empty transcript,
induction shows that the two executions choose identical queries and
receive identical gradients throughout the horizon.  Their terminal
outputs, being functions of the same seed and transcript, are also
identical.  Inserting the terminal directions into the observer is
used only to verify the final projection event; it does not alter
the algorithm's output.  Consequently the observer's good-event
margin is a statement about the original oracle execution.

\subsection{The exact-quadratic branch and proof of its rate}

For $J=0$ use the bounded two-sided base middle chain from
Lemma~\ref{rec:lem:weighted-local}, with
\[
 h_0=b/K,\qquad n_0=\lfloor c_nK^2\rfloor,\qquad
 b=C_b\epsilon/L_0,\qquad
 S=\lfloor c_S\Delta/(L_0b^2)\rfloor .
\]
Replace that chain by the split version above and put its
complete bounded payload, including the public radial term,
in $f$.  Explicitly,
\begin{equation}\label{tr:eq:quadratic-primitives}
 \begin{split}
 f(x,v;Y)&=A_{\rm out}(x)+P_{\rm ph}(Y)
       +\sum_i\psi_i^o(x)
          \{C_{{\rm sp},i}(a,b)+D_+(b_{i,n_0})
               +D_-(b_{i,n_0})-C_{{\rm sp},i}(0,0)\}\\
       &\hspace{12mm}+\sum_iR_i^o(x)Q_i(b_{i,n_0}),\\
 g(x,v;Y)&=\tfrac12Y^\top HY-c_{\rm end}\langle v,Y_n\rangle .
 \end{split}
\end{equation}
The lower variable is just $Y$; there are no normal coordinates,
no fine normal widths, and no hidden inverse perturbation in
this branch.

\begin{proof}[Proof of Theorem~\ref{tr:thm:quadratic}]
First suppose $\kappa\ge\kappa_{\rm tr}$.
Lemma~\ref{tr:lem:trace} identifies the hyperobjective of
\eqref{tr:eq:quadratic-primitives} with the original base
bounded macro--middle function under amplification $K$.
Lemma~\ref{rec:lem:global}, applied with $J=0$, gives its
global margin and gap.  The value and gradient of a selected
middle payload are $O(L_0b^2)$ and $O(L_0b)$, since
$n_0\asymp K^2$ and $h_0=b/K$.  Its Hessian is $O(L_0)$.
The unselected radial derivative is also $O(L_0b)$.
The macro selector scale is $b$, and its release scale is
$b$, so all joint and mixed Hessian blocks of $f$ are
$O(L_0)$ and $\|f_Y\|_\infty\le CL_0b$.
Thus choosing $L_0=\eta_0L$ and $\epsilon\le cL$ gives
$L_f,C_f\le L$.  The lower budgets follow from
\eqref{tr:eq:H-slack}; its Hessian is constant.

Choose $c_S$ small for the gap and then the constant in
\eqref{tr:eq:quadratic-accuracy} small enough that $S\ge64$
and $S\asymp\Delta/(L_0b^2)$.  Increase $C_b$ first so
that the global margin is at least $8\epsilon$.
The unchanged precomposition Hessian on the graph is
$O(L_0)$, whence
$\operatorname{Lip}(\nabla F)\le C(1+K)^2L_0\le CL\kappa^2$.
Every initial vector is zero and $Y^*(0)=0$.
There are
\[
 \mathcal N=Sn_0\asymp L\kappa^2\Delta\epsilon^{-2}
\]
serial deepest stages.  Lemma~\ref{tr:lem:serial} holds with
no nonlinear inverse, using a blind width $h_0/16$.
Apply Lemma~\ref{dt:lem:observer} with $m_{\rm out}=1$ and the
corresponding bounded-readout energy, including terminal output directions.
The original macro observer can independently make its
projection failure probability at most $1/8$; together with
the transport observer the good event has probability at
least $3/4$.

The observer propagates only query-augmented old subspaces
by one edge per round, and a newly disclosed source column
does not propagate in that round.  By the serial support
lemma, it completes at most $1+Q/(n-1)$ deepest stages in
$Q$ rounds, even if the first source is initially free.
For $Q\le n\mathcal N/64$, $n\ge5$, and $\mathcal N\ge64$,
this is strictly less than $\mathcal N$.
The last macro column is therefore undisclosed on the good
event.  The output projection bound makes $Z_S(\widehat x)<b$,
so the output gradient is at least $8\epsilon$ with probability
at least $3/4$.  This proves the fixed-horizon lower bound
$\Omega(n\mathcal N)$.

For completeness, this probability certificate also proves
the expected-cost version without a bounded-gradient
assumption at exceptional outputs.  Fix
$Q_0=\lfloor n\mathcal N/64\rfloor$ and suppose an algorithm
has worst-case expected cost at most $Q_0/8$.  Truncate it at
$Q_0$, outputting the origin if it has not stopped.  Its
probability of truncation is at most $1/8$, averaged over the
same fixed finite-dimensional hard distribution.  Apply the
fixed-horizon certificate to this truncated algorithm.
On original runs that finish by $Q_0$ the outputs coincide;
therefore the original output has gradient at least
$8\epsilon$ with probability at least $3/4-1/8=5/8$.
Its expected gradient norm exceeds $\epsilon$, a contradiction.
All dimensions were chosen at this deterministic $Q_0$
before the experiment.

For $2\le\kappa<\kappa_{\rm tr}$, use the already proved
family in Theorem~\ref{uf:thm:lower} with $\sigma_f=0$.
Its lower objective is quadratic, both gradients are exact,
its curvature is $O(L\kappa^2)$, and its deterministic lower
term is $cL\kappa^2\Delta\epsilon^{-2}$ under the same
accuracy condition.  Since $\sqrt\kappa\le
\sqrt{\kappa_{\rm tr}}$ is a fixed numerical constant,
this implies \eqref{tr:eq:quadratic-rate} after reducing $c$.
This completes all condition numbers and the full stated
accuracy range.
\end{proof}

\section{A deterministic delay through orthogonal transport}
\label{app:det-transport-info}

The following argument is for the exact-gradient oracle.  It adds a
deterministic transport delay; it does not multiply a Bernoulli waiting
time by that delay.  Ambient dimensions may be extremely large, but are
finite and are fixed before the algorithm runs.

\subsection{A conditional angle estimate}

For a frame $V$ and an index set $I$, write $V_I$ for its indicated
columns.  For a subspace $W$ define its frame energy on $I$ by
\[
 \mathcal E(W,V_I)=\|P_WV_I\|_{\mathrm F}^2.
\]
Operator norms are used unless a Frobenius subscript is displayed.

\begin{lemma}[A fresh transport block has a nondegenerate innovation]
\label[lemma]{dt:lem:angle}
Let $V,Z\in\mathbb R^{D\times T}$ be orthonormal frames and let
$O$ be an orthogonal map with $OV=Z$.  Conditional on these full frames,
its restriction between their orthogonal complements is Haar.  A
history has disclosed $O$ on a subspace $A$; put $B=OA$. Assume
$D>T+\dim A$. Conditional on that history, the remaining
complementary isometry is Haar.
All the spaces and inputs $A,S,U$, and the protected index set $K$
below are fixed, or measurable, under this current conditioning;
they cannot depend on the still-undisclosed isometry.
Let $K$ be a protected set with $V_K\subseteq A$ and
$Z_K\subseteq B$.  Let $S$ contain $A$ and suppose
\[
 \mathcal E(S,V_{K^c})\le e,\qquad e\le 2^{-12}.
\]
Let $U$ be an orthonormal basis of a subspace of $S\cap A^\perp$,
with $r$ columns.  A destination subspace $W$ contains $B$ and may
contain any additional vectors fixed under the current conditioning;
write $d_W=\dim W$.  Then, except with conditional probability at most
\[
 \frac{64 r d_W}{D-T-\dim A},
\]
the block $(I-P_W)OU$ has smallest singular value at least $1/2$.
On this event, for every $I\subseteq K^c$, adding $OU$ to $W$ changes
its energy by at most
\begin{equation}\label{dt:eq:energy-update}
 \mathcal E(W+\operatorname{span}(OU),Z_I)
 -\mathcal E(W,Z_I)
 \le 8\{e+\mathcal E(W,Z_I)\}.
\end{equation}
The angle assertion also holds when $W$ has large energy on some
protected or newly disclosed columns: its hypothesis concerns the
source $S$, not the destination energy.
\end{lemma}
\begin{proof}
Remove the protected columns from an orthonormal basis of $A$, calling
the remaining basis $A_0$.  Decompose $A_0=A_h+A_c$ and $U=U_h+U_c$
into the full-frame span and its orthogonal complement.  The hard
components have Frobenius norms at most $\sqrt e$.  Since
$A_0^\top U=0$,
\[
 A_c^\top U_c=-A_h^\top U_h,\qquad
 A_c^\top A_c\succeq(1-e)I.
\]
Therefore the part of $U_c$ in the already determined complementary
domain has norm at most $2e$.  Its remaining part has norm at most
one and belongs to the as yet undetermined complementary domain.
Conditional on the disclosed action, an orthonormal basis of that
remaining part has a Haar image in a space of dimension at least
$D-T-\dim A$.  The expected squared Frobenius norm of its projection
onto $W$ is at most $r d_W/(D-T-\dim A)$.  Markov's inequality
makes this projected norm at most $1/8$ with the stated probability.
The image of the hard component has norm at most $\sqrt e$, and the
old complementary component costs at most $2e$.  Hence
\[
 \|P_WOU\|\le\sqrt e+2e+1/8<1/2.
\]
Because $OU$ is orthonormal, the required singular-value lower bound
follows.

For the energy assertion take
$M=(I-P_W)OU$ and its orthonormalization
$\widehat M=M(M^\top M)^{-1/2}$.  The inverse square root has norm
at most two.  Also
\[
 \|Z_I^\top M\|_{\mathrm F}
 \le \|Z_I^\top OU\|_{\mathrm F}
      +\|Z_I^\top P_WOU\|_{\mathrm F}
 \le \sqrt e+\sqrt{\mathcal E(W,Z_I)}.
\]
Here $Z_I^\top O=V_I^\top$ and $I\subseteq K^c$.
Squaring after orthonormalization proves \eqref{dt:eq:energy-update}.
\end{proof}

\begin{lemma}[Disclosing one named column]
\label[lemma]{dt:lem:column}
Let $v$ be a unit column orthogonal to every column of a frame $F$.
If $\|P_Wv\|^2\le e<1/2$, then
\[
 \mathcal E(W+\operatorname{span}(v),F)
 \le (1+2e)\mathcal E(W,F).
\]
\end{lemma}
\begin{proof}
The new unit direction is
$(v-P_Wv)/\sqrt{1-\|P_Wv\|^2}$.
Its squared frame projection is at most
$\mathcal E(W,F)e/(1-e)$ because $F^\top v=0$.
\end{proof}

\subsection{The network distribution and its two conditioning arguments}

There are $n$ physical spaces, each $\mathbb R^D$.  Sample independent
Haar $T$-frames $V_1,\ldots,V_n$.  Conditional on them, sample each
edge map $O_l$ independently and uniformly from
\[
 \{O\in O(D):OV_l=V_{l+1}\}\qquad(1\le l<n).
\]
Equivalently, sample $V_n$ Haar and independent Haar edge maps, and
define $V_l=O_l^\top V_{l+1}$ backwards.  In particular the residual
edge maps on the orthogonal complements of the full frames are
independent Haar isometries.

We use two distinct conditioning arguments, and do not interchange them.

First, conditional on the actual observer transcript, fresh query
directions are controlled by layerwise gauge invariance.  For arbitrary
orthogonal $R_l$ fixing the known subspace at layer $l$, the transformation
\[
 V_l\mapsto R_lV_l,\qquad
 O_l\mapsto R_{l+1}O_lR_l^\top
\]
preserves the prior and every already disclosed action and projection.
Consequently an undisclosed frame column has a rotationally uniform
residual in the complement of the known subspace.  In particular, one
new unit query innovation has conditional expected total projection
onto $T$ undisclosed columns at most $T/(D-d)$, where $d$ is the
current known-subspace dimension.

Second, to apply Lemma~\ref{dt:lem:angle} to one edge, condition on
all full hard frames (including any independent auxiliary hierarchy
frames), on the complete maps of every other edge, and on
the current edge's disclosed action transcript.  The global adaptive
history can be reconstructed from these data and the current edge's
successive disclosed actions.  Its queries are therefore adaptive
linear probes of this edge, and constrain it only on its paired domain
$A\mapsto OA$.  Its residual complementary isometry remains Haar.
This remains true after exposing the first direction of a two-direction
batch: the paired domain is enlarged, and the second direction uses
the remaining Haar isometry.  The full-map conditioning used in this
second argument is not used for the first, query-projection argument.

\subsection{A one-hop observer with frozen input spaces}

For clarity, an observer may disclose the first column of the current
deepest stage for free.  It begins with column one at layer one.  In
each round it advances the individually disclosed current column by
one layer.  Once this column reaches layer $n$, the next column at
layer one becomes individually disclosed, for use in the next round.
Completed columns are known at every layer.  Let $K_l^t$ be the set
known at layer $l$ at the start of round $t$.  These sets are prefixes;
neighboring prefixes differ by at most the current column.  Their
schedule completes exactly $\lfloor t/(n-1)\rfloor$ stages in $t$
rounds until all stages are complete; the free initial source column
is not itself a completed stage.
Any parent hierarchy columns can be disclosed for free when their
whole child subtree is complete.  This does not advance an unfinished
deepest stage.

At layer $l$ maintain a known subspace $W_l^t$ containing all known
columns, past queries, and past responses.  The observer supplies every
hard-frame column's projection onto this subspace.  It also records,
for every edge, the exact action on a paired domain
$A_l^t\mapsto O_lA_l^t$, including observations in both directions.

Each round has the following order.
\begin{enumerate}
\item The algorithm fixes all its query vectors using the previous
transcript.  Insert these vectors into the corresponding spaces and
disclose their frame projections.  Call the resulting spaces $S_l^t$.
They are the frozen input spaces for the entire round.
\item Along every edge disclose the images of the whole frozen input
space in each direction.  Disclose the one-hop individually named
columns at the same time.  This gives more information than the true
quadratic gradient.  No vector newly obtained in this step is used
as a transport input in this round.
\item Return the virtual nonlinear responses, using the current
eligible first-end column and the known subspaces.  If the transport
has completed a stage, disclose the next first-end column for the
following round.  Include all resulting vectors in $W_l^{t+1}$.
\end{enumerate}
The dimensions grow at most geometrically: a deterministic bound
$d_t\le C_0 4^{t+1}$ suffices after increasing $C_0$ for the fixed
number of query directions per physical layer.  A larger bound may
be used without changing the argument.

The true quadratic oracle only requires $O_lY_l$ and
$O_l^\top Y_{l+1}$, with fixed numerical edge coefficients.  These
are contained in this stronger observer's responses.  Exact action
on an old subspace also absorbs all backtracking paths; such paths
are not treated as independent innovations.

\subsection{Local edge energies and the simultaneous-batch issue}

There is an important distinction between the local source space for
an edge and the global destination space.  Fix a round and an edge
$e=(l,l+1)$, and freeze
\[
 K_e=K_l^t\cup K_{l+1}^t,\qquad I_e=\{1,\ldots,T\}\setminus K_e.
\]
First pair the columns in $K_e$ across this edge.  This is a legal
one-hop disclosure, since each such column was known at one endpoint
before the round.  Define two \emph{local} spaces using only
\[
 S_l^t+\operatorname{span}(V_{l,K_e}),\qquad
 S_{l+1}^t+\operatorname{span}(V_{l+1,K_e}),
\]
and the two actions of this same edge.  These local spaces do not
include another edge's new response or newly arriving named column.
The current edge's paired domain is contained in its local source
space and its paired range is contained in its local destination space.

Suppose the old and query-augmented global blind energies are small.
Lemma~\ref{dt:lem:column} bounds the initial local energies on $I_e$
by a fixed multiple of those old energies.  In the forward action,
remove the paired domain from the fixed input $S_l^t$.  The residual
input belongs to the local source space, so
Lemma~\ref{dt:lem:angle} applies with the protected set $K_e$.
The global destination may contain previously exposed actions of
other edges and their new named columns.  Conditional on the full
frames and other full maps, these vectors are fixed; they affect
the dimension in the angle estimate but not its source hypothesis.

The angle bound against the global destination also holds against
the smaller local destination.  Apply the energy bound there on the
entire set $I_e$.  This updates the current edge's local energy by
only a fixed factor.  Now enlarge its paired action by the forward
observation.  For the reverse action, remove this enlarged paired
range from the original fixed input $S_{l+1}^t$.  The residual input
is contained in the local space generated by the old input and the
first action; it need not belong to the old input alone.  Its energy
on $I_e$ is the updated local energy just proved.  Apply the same
angle argument to the remaining reverse Haar isometry.  Thus both
directions are covered without pretending that they are independent.

For global merging, track the indices not individually disclosed
\emph{after} the one-hop step.  At an endpoint of this edge this
global future set is a subset of $I_e$: every old named column at
either endpoint has just been paired.  Equation
\eqref{dt:eq:energy-update} therefore bounds the global energy
increment by a fixed multiple of the local source energy and the
current global destination energy.  A column newly arriving via a
different edge is never removed from the source energy by fiat; it
is absent from this edge's local input unless its small old projection
or this edge's own actions put it there.

This fixed-local-set construction is necessary.  Merely declaring
every newly disclosed column known in a global source space would
incorrectly remove it from the energy bound before it was legally
available as a same-round input.

\subsection{Finite dimensions and the serial delay}

\begin{lemma}[Finite-dimensional transport observer]
\label[lemma]{dt:lem:observer}
Fix $n\ge2$, $T\ge1$, a finite horizon $Q$, a nonnegative integer
$m_{\rm out}$, and a target energy $0<e_*\le2^{-12}$.
There is a finite ambient dimension $D$, chosen before execution,
such that the preceding observer has the following property for
every randomized adaptive algorithm.  With probability
at least $7/8$, throughout the horizon the aggregate blind energy
at every layer satisfies
\[
 \mathcal E(W_l^t,(V_l)_{(K_l^t)^c})\le e_*.
\]
The same good event also controls the two spaces used in the
restricted-inverse argument. Let $W_l^+$ and $K_l^+$ denote the
known space and named columns after the complete one-hop step 2,
with its edge outputs and all legal named disclosures registered
together, before the next first-endpoint source is disclosed. Then
\[
 \mathcal E(S_l^t,(V_l)_{(K_l^t)^c})\le e_*,\qquad
 \mathcal E(W_l^+,(V_l)_{(K_l^+)^c})\le e_*.
\]
The first estimate follows after query insertion; the second is the
post-one-hop future-set estimate in the global-merging argument.
The same holds after inserting at most $m_{\rm out}$ terminal output
directions per physical layer. Their values may be arbitrary measurable
functions of the transcript and the private randomness; their number
is included in the dimension choice.
\end{lemma}
\begin{proof}
There are at most $M=Cn(Q+1+m_{\rm out})$ directed block updates,
named-column updates, and query or terminal insertions, for a fixed
sufficiently large numerical $C$.
Let $d_*=C_0(1+m_{\rm out})4^{Q+2}$ bound all relevant subspace dimensions.
The preceding two-direction analysis and
Lemmas~\ref{dt:lem:angle}--\ref{dt:lem:column} show that, whenever
the angle events hold, all global and local energies are bounded
by a recurrence that multiplies the accumulated energy by at most
a fixed numerical $C_1$ per update, plus the fresh query injections.
Choose each injection threshold at most
\[
 e_{\rm in}=\frac{e_*}{(M+1)C_1^{M+1}}.
\]
The conditional gauge-invariance bound and Markov's inequality make
the failure probability of one insertion at most
$T/((D-d_*)e_{\rm in})$.  The angle failure probability of one
block is at most $64d_*^2/(D-T-d_*)$.
Choose $D>2(T+d_*)$ so large that the sum of these probabilities
over $M$ updates is at most $1/8$.  Such a finite $D$ exists;
for example a sufficiently large numerical multiple of
\[
 T+d_*+M\{T/e_{\rm in}+d_*^2\}
\]
suffices.  The energy recurrence then remains below $e_*$, including
every local energy used as an angle hypothesis.  The factor $C_1$
can be increased initially to include the fixed number of local
bookkeeping updates per directed action.

The argument conditions on the algorithm's private randomness as
part of the transcript.  It does not restrict its query norms or
its measurable choice of terminal output.  Terminal directions are
handled by the same insertion estimate.
\end{proof}

For the baseline calibration the energy target may be chosen as in
Lemma~\ref{tr:lem:inverse-locality}. For the enlarged variation range,
use the finite targets of Lemma~\ref{ni:lem:transport}, including the
$R+H$ factor for newly disclosed columns. That lemma proves exact
restricted-inverse coupling without a small total displacement.
The virtual nonlinear correction is defined by known-subspace
restriction even off the projection event; on it the returned gradient
is the actual gradient. The matrix part is supplied exactly by the
transport observer on every realization.

The endpoint chain has its incoming chain receiver at layer one and its
predecessor gates, the two terminal receivers $D_\pm$, and the terminal
release gate $Q$ at layer $n$.  Its tail masks also use only the
layer-$n$ coordinates.  Thus the current first-end column
cannot be followed by the next eligible first-end column until the
current column has reached the last layer.  The observer's automatic
one-hop schedule is already more informative than the actual oracle.
It completes at most $\lfloor Q/(n-1)\rfloor$ deepest stages.  Taking
$Q\le c nT$ with a sufficiently small fixed $c$ therefore leaves
the final outer hierarchy column undisclosed on the good event.
The global hyperobjective margin applies to the arbitrary output.

\begin{theorem}[Transport information and deterministic lower bound]
\label[theorem]{tr:thm:information}
\label[theorem]{dt:thm:det-lower}
Consider the preceding $n$-layer independent orthogonal-transport
distribution with $T$ required deepest source stages.  Assume bounded
endpoint readouts with the split gates and terminal receivers just
specified; on a blind tail the nonlinear gradient uses known prefix
columns and at most the eligible current first-end column.  Assume
public radial directions and a near-identity correction admitting
exact restricted-inverse coupling at finite projection targets, as
proved in Lemma~\ref{tr:lem:inverse-locality} for the baseline
calibration or Lemma~\ref{ni:lem:transport} for the enlarged range.  Suppose
the other hierarchy frames use the original projection observer, and
their final outer column is released only after the $T$ deepest stages
are complete.  If an undisclosed final outer column gives a global
output gradient margin at least $8\epsilon$, then sufficiently large
finite dimensions ensure that every randomized adaptive algorithm
making $Q\le c nT$ exact-gradient calls has output gradient at least
$8\epsilon$ with probability at least $3/4$.  The output may be any
measurable function of the transcript and private randomness.

In particular, assume the exact endpoint-amplified geometric construction
and primitive budget certificate, with $n=\Theta(\sqrt\kappa)$ and
\[
 T=\widetilde\Omega(\vartheta\kappa^2L\Delta\epsilon^{-2}),
\]
in the stated recursive accuracy range.  Then the exact-gradient
complexity for arbitrary randomized adaptive algorithms and arbitrary
outputs is
\[
 \widetilde\Omega(\vartheta\kappa^{5/2}L\Delta\epsilon^{-2}).
\]
The same order holds for worst-case expected query costs.
\end{theorem}
\begin{proof}
Choose the dimension in Lemma~\ref{dt:lem:observer} with
$m_{\rm out}=1$ at the proposed fixed horizon.  The old observers for the other hierarchy frames
have the same conditional projection property and can use independently
larger finite dimensions.  Their failures and the transport failures
can be allocated a total probability at most $1/4$.  The exact split
endpoint graph preserves the previous global gradient margin, which
is scaled to at least $8\epsilon$.  For $Q\le c nT$ the final
outer column is still undisclosed, so the expected output gradient
exceeds $\epsilon$.  This proves the fixed-horizon assertion.
For the expected-cost assertion fix an integer horizon
$q_0=\lfloor c nT\rfloor$ with the preceding fixed-horizon property,
reducing $c$ by a fixed factor if needed.  Suppose an algorithm has
the uniform expected-gradient guarantee and worst-case expected cost
at most $q_0/8$.  Under the same distribution of hard instances,
Markov's inequality gives
\[
 \Pr\{\|\nabla F(\widehat x)\|\ge8\epsilon\}\le1/8,
 \qquad \Pr\{N>q_0\}\le1/8.
\]
Truncate it after $q_0$ calls and use any fixed output on the truncation
event.  Its probability of gradient at least $8\epsilon$ is then
at most $1/4$, whereas the fixed-horizon result gives at least $3/4$.
This contradiction proves the expected-cost lower bound.  If the
integer horizon is zero, the zero-query version of the same barrier
and $\Pr\{N\ge1\}\le\mathbb E N$ give a positive numerical
expected-cost lower bound, which suffices after reducing the final
constant.  In the geometric application the number of stages already
exceeds a fixed numerical threshold.
All oracle responses in this theorem are exact gradients, so the
instances are admissible under any larger stated variance budgets.
\end{proof}

\end{document}